\documentclass{article}
\usepackage[english]{babel}
\usepackage{amsthm,amsfonts,amsbsy, amssymb,amsmath,graphicx}
\usepackage{mathrsfs}
\usepackage{graphics}
\usepackage{afterpage}
\usepackage[all]{xy}
\usepackage{array,longtable}

\usepackage{apptools}
\AtAppendix{\counterwithin{lemma}{section}}
\AtAppendix{\counterwithin{definition}{section}}
\AtAppendix{\counterwithin{proposition}{section}}
\AtAppendix{\counterwithin{theorem}{section}}
\AtAppendix{\counterwithin{corollary}{section}}

\newtheorem{theorem}{Theorem}
\newtheorem*{theorem*}{Theorem}
\newtheorem{lemma}{Lemma}
\newtheorem{corollary}{Corollary}

\newtheorem{proposition}{Proposition}

\theoremstyle{definition}
\newtheorem{definition}{Definition}

\theoremstyle{remark}
\newtheorem{remark}{Remark}
\newtheorem{example}{Example}

\def\Z{{\mathbb Z}}
\def\N{{\mathbb N}}

\def\R{{\mathbb R}}

\DeclareMathOperator{\id}{id}

\newcommand{\sgn}{\mathrm{sgn}}

\newcommand{\ptr}[1]{\raisebox{-0.25\height}{\includegraphics[width=0.7cm]{trait_type_#1.eps}}}

\newcommand{\skcrv}{\raisebox{-0.25\height}{\includegraphics[width=0.5cm]{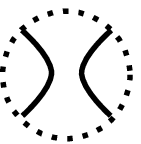}}}
\newcommand{\skcrh}{\raisebox{-0.25\height}{\includegraphics[width=0.5cm]{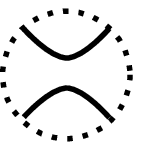}}}
\newcommand{\skcrr}{\raisebox{-0.25\height}{\includegraphics[width=0.5cm]{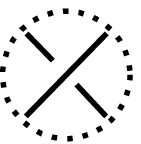}}}
\newcommand{\skcrl}{\raisebox{-0.25\height}{\includegraphics[width=0.5cm]{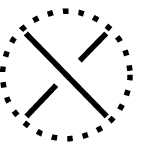}}}
\newcommand{\skcrvirt}{\raisebox{-0.25\height}{\includegraphics[width=0.5cm]{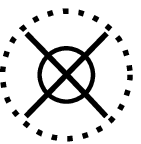}}}

\date{}

\title{Local transformations and functorial maps}
\author{Igor Nikonov}

\begin{document}

\maketitle

\begin{abstract}
Picture-valued invariants are the main achievement of parity theory by V.O. Manturov. In the paper we give a general description of such invariants which can be assigned to a parity (in general, a trait) on diagram crossings. We distinguish two types of picture-valued invariants: derivations (Turaev bracker, index polynomial etc.) and functorial maps (Kauffman bracket, parity bracket, parity projection etc.). We consider some examples of binary functorial maps.

Besides known cases of functorial maps, we present two new examples. The order functorial map is closely connected with (pre)orderings of surface groups and leads to the notion of sibling knots, i.e. knots such that any diagram of one knot can be transformed to a diagram of the other by crossing switching. The other is the lifting map which is inverse to forgetting of under-overcrossings information which turns virtual knots into flat knots. We give some examples of liftable flat knots and flattable virtual ones.

An appendix of the paper contains description of some smoothing skein modules. In particular, we show that $\Delta$-equivalence of tangles in a fixed surface is classified by the extended homotopy index polynomial.
\end{abstract}


\section{Introduction}

One of the major achievements of parity theory proposed by V.O. Manturov in~\cite{M3} are picture-valued invariants of knots. Parity brackets and functorial map allow one to establish minimality of knot diagrams and to construct (counter) examples~\cite{BR,Ch,M7}.

Another source of picture-valued invariants is the theory of chord (crossing) indices of knots. Examples of such invariants are Goldman's bracket~\cite{G} and Turaev's cobracket~\cite{T}. A. Henrich defined a smoothing and a glueing invariants of virtual knots; the latter was proved to be the universal Vassiliev invariant of degree one for virtual knots. Following V.~Turaev, P.~Cahn proposed an elaboration of the glueing invariant~\cite{C}, and Z. Cheng, H. Gao and M. Xu considered a picture-valued invariant based on unoriented smoothing~\cite{CGX}.
A. Gill, M. Ivanov, M. Prabhakar and A. Vesnin used smoothings to construct multi-parameter series of virtual knot invariants~\cite{GIPV}.

Besides, after V. Turaev~\cite{Tsk} and J. Przytycki~\cite{P}, the polynomial knot invariants can be thought of as diagram-valued in the correspondent skein modules.

The goal of the present paper is to provide a unified description of such picture-valued invariants.

\begin{figure}[h]
\centering\includegraphics[width=0.3\textwidth]{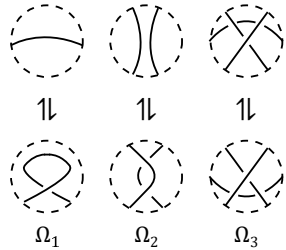}
\caption{Reidemeister moves}\label{fig:reidemeister_moves}
\end{figure}

From the combinatorial viewpoint, a {\em knot} is an equivalence class of diagrams modulo isotopies and Reidemeister moves (Fig.~\ref{fig:reidemeister_moves}).  Given a diagram, we can construct a new diagram by replacing the crossings of the diagram by some small diagrams (tangles) according to some transformation rule. We will call such a transformation \emph{local} because it applies to the smallest parts of the diagram --- the crossings. The result of the transformation will be considered as a diagram in some knot theory which may be exotic (i.e. be theory with moves other than the Reidemeister ones).

\begin{equation*}
\xymatrix{
*+[F]{\txt{source \\ knot theory}} \ar@{-}[r] & *+[F-:<12pt>]{\txt{local\\ transformation\\ rule}} \ar[r] & *+[F]{\txt{destination\\ knot theory}}
}
\end{equation*}

The question is to find an appropriate transformation rule and a destination knot theory so that the result would not depend (up to moves in the destination knot theory) on the choice of the initial diagram of the knot.

Let us proceed with several examples.

\begin{example}[Kauffman bracket~\cite{K0}]\label{exa:kauffman_bracket}

The Kauffman bracket can be considered as a map from framed links to some skein module $\mathscr S$. For a link diagram $D$, its bracket $\langle D\rangle$ is obtained by applying the smoothing rule in Fig.~\ref{fig:kauffman_smoothing} to each crossing of the diagram.

\begin{figure}[h]
\centering\includegraphics[width=0.4\textwidth]{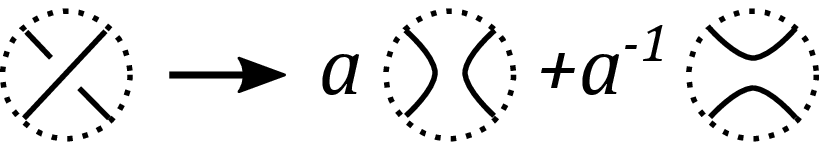}
\caption{Kauffman smoothing rule}\label{fig:kauffman_smoothing}
\end{figure}

The skein module $\mathscr S$ consists of linear combinations of diagrams without crossings (i.e. sets of disjoint embedded circles) over the ring $A=\Z[a,a^{-1}]$ modulo the circle reduction move $O_\delta$ where $\delta=-a^2-a^{-2}$. Then $\mathscr S$ is freely generated over $A$ by the empty diagram. Hence, the bracket $\langle D\rangle$ is a product of an element of $A$ (i.e. a Laurent polynomial in $a$) and the empty diagram. This polynomial is the conventional value of the Kauffman bracket of the link.

\begin{figure}[h]
\centering\includegraphics[width=0.25\textwidth]{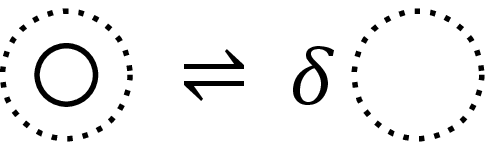}
\caption{Trivial circle reduction move $O_\delta$}\label{fig:circle_remove}
\end{figure}

On the other hand, we can consider the skein module $\mathscr S'$ which consists of linear combinations of link diagrams over $\Z[a,a^{-1}]$ modulo the Kauffman smoothing rule and the move $O_\delta$. Consider the identity smoothing rule which sends a diagram $D$ to the same diagram but it is considered as an element of $\mathscr S'$. Since $\mathscr S'$ is a free module over $\Z[a,a^{-1}]$ generated by the empty diagram, the identity smoothing rule yields another description of Kauffman bracket.
\end{example}

\begin{example}[Turaev cobracket~\cite{Tsk}]\label{exa:turaev_cobracket}
Let $F$ be a compact oriented two-dimensional surface. Consider an oriented flat knot $K$ in $F$, i.e. a free homotopy type of a closed loop in $F$. Combinatorially, a flat knot can be defined as an equivalence class of generically immersed curves (i.e. immersed curves whose self-intersections are double points) modulo isotopies and flat Reidemeister moves (Fig.~\ref{fig:reidemeister_moves_flat_oriented}).

\begin{figure}[h]
\centering\includegraphics[width=0.3\textwidth]{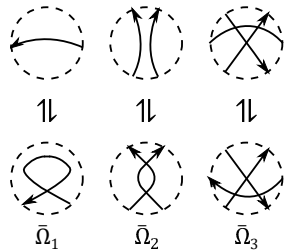}
\caption{Flat Reidemeister moves}\label{fig:reidemeister_moves_flat_oriented}
\end{figure}

The Turaev's delta map is the sum $\Delta(K)=\sum_{c} \Delta_c(K)$ of two-component links with numbered components obtained by applying the cobracket rule (Fig.~\ref{fig:turaev_cobracket}) at the given crossing.

\begin{figure}[h]
\centering\includegraphics[width=0.4\textwidth]{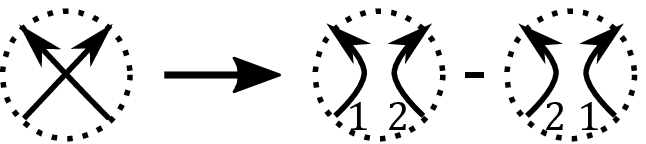}
\caption{Turaev cobracket rule}\label{fig:turaev_cobracket}
\end{figure}

The skein module consists of integral linear combinations of oriented two-component flat diagrams modulo the flat Reidemeister moves and the move $O_0$.
\end{example}

Note that Turaev's cobracket represents another type of picture-valued invariant than Kauffman bracket. For Kauffman bracket, we applied the transformation rule to all crossings simultaneously whereas here we apply the transformation rule only once in each summand. Below we will call  invariants like Kauffman bracket {\em functorial maps} and call latter invariants {\em derivations}.

\begin{example}[Goldman bracket~\cite{G}]\label{exa:goldman_bracket}

Goldman bracket maps oriented two-components flat links with numbered components in the given compact oriented surface $F$ to the free $\Z$-module of flat knots in $F$:
\[
m(L)=\sum_c m_c(L)
\]
where $m_c(L)$ is determined by the bracket rule at the crossing $c$ (Fig.~\ref{fig:goldman_bracket}).

\begin{figure}[h]
\centering\includegraphics[width=0.3\textwidth]{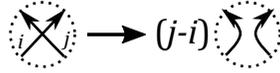}
\caption{Goldman bracket rule}\label{fig:goldman_bracket}
\end{figure}

Goldman bracket demonstrates a new feature: the transformation rule acts differently on different types of crossings. Namely, self-crossings give zero summands to the bracket, and the sign of the summand for a mixed crossing depends on the position of the link components at the crossing.
\end{example}

\begin{example}[Parity bracket~\cite{M3}]\label{exa:parity_bracket}

The parity bracket acts on free knots, i.e. virtual diagrams modulo Reidemeister moves (Fig.~\ref{fig:reidemeister_moves_virtual}), crossing change and flanking move (Fig.~\ref{fig:flanking_move}).

\begin{figure}[h]
\centering\includegraphics[width=0.4\textwidth]{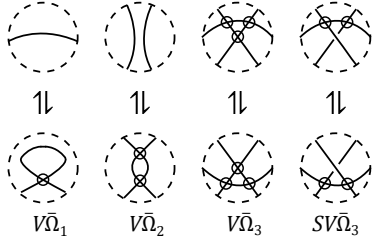}
\caption{Virtual Reidemeister moves}\label{fig:reidemeister_moves_virtual}
\end{figure}

\begin{figure}[h]
\centering\includegraphics[width=0.2\textwidth]{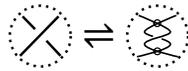}
\caption{Flanking move $Fl$}\label{fig:flanking_move}
\end{figure}

Assume there is a some $\Z_2$-labelling of crossings of a virtual knot diagram. The bracket is the result of smoothing of the diagram in all crossings according to the bracket rule in Fig.~\ref{fig:parity_bracket}.

\begin{figure}[h!]
\centering\includegraphics[width=0.4\textwidth]{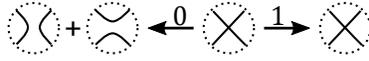}
\caption{Parity bracket rule}\label{fig:parity_bracket}
\end{figure}

The result belongs to the skein module that consists of $\Z_2$-combinations of virtual diagrams modulo the move $\Omega_2$, virtual Reidemeister moves, crossing change, flanking move and the move $O_0$. The bracket is an invariant of free knots when the $\Z_2$-labelling is a parity.

A \emph{parity} is a rule to assign numbers $0$ and $1$ to the (classical) crossings of diagrams of a knot in a way compatible with Reidemeister moves: the parity of a crossing does not change if it does not involved in the Reidemeister move, and the parities of the crossings participating in the move satisfy the parity axioms in Fig.~\ref{fig:parity_axioms}.

\begin{figure}[h]
\centering\includegraphics[width=0.6\textwidth]{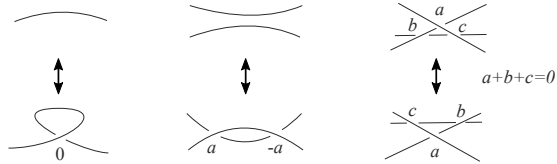}
\caption{Parity axioms}\label{fig:parity_axioms}
\end{figure}
\end{example}

\begin{example}[Parity projection~\cite{M3}]\label{exa:functorial_map}

Let $K$ be a virtual knot. Assume there is some $\Z_2$-labelling of crossings of diagrams of $K$. Virtualise the diagram crossings according to the projection rule in Fig.~\ref{fig:functorial_map_rule}.

\begin{figure}[h]
\centering\includegraphics[width=0.4\textwidth]{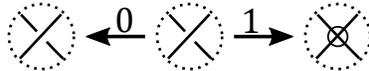}
\caption{Parity projection rule}\label{fig:functorial_map_rule}
\end{figure}

If the labelling is a \emph{weak parity} then the result is an invariant with values in virtual knots.

The definition of a weak parity is obtained from the definition of parity by replacing the condition $a+b+c=0$ for third Reidemeister moves by a weaker condition that if two labels among $a,b,c$ are equal to zero then the third label is zero too.
\end{example}

As the last two examples show, different transformation rules impose different conditions on the crossing labelling.



The paper is organized as follows. Section~\ref{sec:definitions} contains combinatorial descriptions of knot theories and the definition of crossings traits and their specifications (indices and parities). We finish the section with the description of the universal trait for knots and tangles in a fixed surface. In Section~\ref{sec:local_transformations} we define local transformations on knot theories and consider two types of transformations: derivations and functorial maps. We formulate sufficient invariance conditions for derivations (Section~\ref{subsec:derivations}) and functorial maps (Section~\ref{subsec:functorial_maps}). Section~\ref{sec:functorial_maps_examples} contains some examples of functorial maps. We consider unary (Section~\ref{subsec:unary_functorial_maps}) and binary functorial maps (Section~\ref{subsec:binary_functorial_maps}). In Section~\ref{subsect:cc_map_orders} we consider functorial maps which switches diagram crossings. We show how these functorial maps relate to (pre)orderings on the surface group. Here we also introduce the notion of sibling knots. Section~\ref{subsect:lifting} is devoted to the problem of lifting of flat knots to virtual ones. We define liftable flat knots and flattable virtual knots and present some examples of them.

The paper contains three appendices. Section~\ref{app:binary_trait_types} contains classification of binary traits based on their compatibility with the Reidemeister moves. Section~\ref{app:skein_modules} includes description of some skein modules. The first part of the section is devoted to smoothing skein modules. In the second part we show that $\Delta$-skein module can be described by the extended homotopy index polynomial.
Section~\ref{app:functorial_map_table} contains a table of the cases of binary functorial maps discussed in the paper.

\section{Crossing traits in a knot theory}\label{sec:definitions}

\subsection{Knot theories}\label{subsec:knot_theory}




\begin{definition}\label{def:tangle_diagram}
Let $F$ be an oriented compact connected surface. A \emph{tangle diagram} $D$ is an embedded finite graph with vertices of valences $1$ and $4$ such that the set $\partial D$ of vertices of valence $1$ coincides with $D\cap\partial F$ and the vertices of valence $4$ carry additional structure (see Fig.~\ref{fig:crossing_types}).

\begin{figure}[h]
\centering\includegraphics[width=0.4\textwidth]{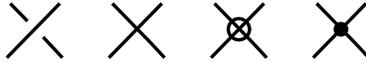}
\caption{Types of crossings: classical, flat, virtual, singular}\label{fig:crossing_types}
\end{figure}

Diagrams are considered up to isotopy fixed on the boundary.

\begin{figure}[h]
\centering\includegraphics[width=0.35\textwidth]{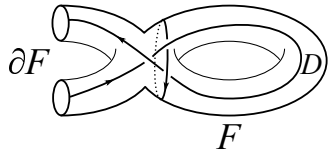}
\caption{An oriented tangle diagram with a long and a closed components}
\end{figure}
\end{definition}

%

Edges incident to a $4$-valent vertex of a tangle diagram split naturally in two pairs of \emph{opposite} edges. Correspondence between opposite edges induces an equivalence relation on the set of edges of the diagram. Equivalence classes of this relation are called \emph{(unicursal) components} of the diagram. A component is \emph{long} if it contains vertices of valency $1$ (in this case the edges of the component form a path in the diagram), otherwise the component is called \emph{closed} (in this case the edges of the component form a cycle).

We say that a diagram is \emph{oriented} if all its components are oriented.

A diagram without long components is a \emph{link} diagram, a link diagram with one closed component is a \emph{knot} diagram. A diagram with one component which is long, is called a \emph{long knot} diagram.

Let us denote some sets of diagrams:
\begin{itemize}
\item $\mathscr D(F,X)$ denotes the set of isotopy classes of unoriented classical diagrams $D\subset F$ with the boundary $\partial D=X\subset \partial D$;
\item $\mathscr D(F)=\mathscr D(F,\emptyset)$;
\item $\mathscr D_+(F,X)$ denotes the set of isotopy classes of oriented classical diagrams;
\item $\bar{\mathscr D}(F,X)$ denotes the set of isotopy classes of diagrams with flat crossings;
\item $\mathscr D^0(F,X)$ denotes the set of isotopy classes of diagrams without crossings.
\end{itemize}

\begin{definition}\label{def:n_tangle}
An \emph{$n$-tangle} is a diagram $D$ in the standard disk $\mathbb{D}^2$ such that $\partial D=X$ where $X\subset \mathbb{D}^2$ is a fixed counterclockwise enumerated set with $2n$ elements. The set of $n$-tangles is denoted by $\mathcal T_n$, and $\mathcal T_n^+$ is the set of oriented $n$-tangles. An $n$-tangle may have crossings of any type.

\begin{figure}[h!]
\centering\includegraphics[width=0.12\textwidth]{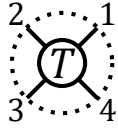}
\caption{A 2-tangle}\label{fig:2tangle}
\end{figure}
\end{definition}


\begin{figure}[h]
\centering\includegraphics[width=0.45\textwidth]{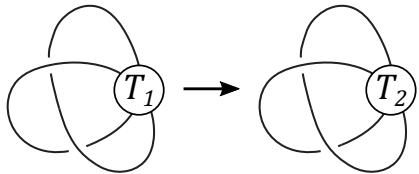}
\caption{Application of a local move to a diagram}\label{fig:local_transformation}
\end{figure}

\begin{definition}\label{def:local_move}
A \emph{local move} is a pair $M=(T_1,T_2)$ of $n$-tangles such that $\partial T_1=\partial T_2$.
\end{definition}

Given a move $M=(T_1,T_2)$, a diagram $D$ and a disk $B\subset F$ such that for $T=D\cap B$  the pair $(B,T)$ is homeomorphic to $(\mathbb{D}^2, T_1)$ (the homeomorphism preserves the orientations of the surfaces), one gets a new diagram $R_M(D,T)$ by replacing the subtangle $T$ with the subtangle homeomorphic to $T_2$ (see Fig.~\ref{fig:local_transformation}).

\afterpage{
\begin{figure}[p]
\centering\includegraphics[width=0.5\textwidth]{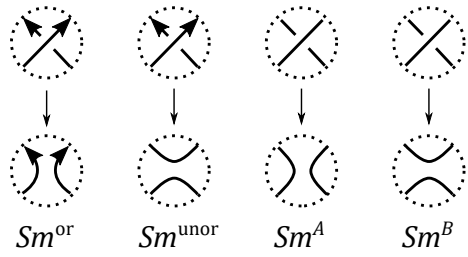}
\caption{Smoothings}\label{fig:smoothing_moves}
\end{figure}

\begin{figure}[p]
\centering\includegraphics[width=0.3\textwidth]{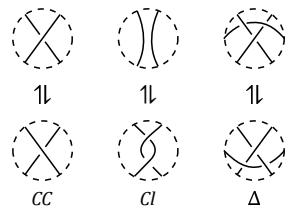}
\caption{Crossing change, clasp move and $\Delta$-move}\label{fig:crossing_change_moves}
\end{figure}

\begin{figure}[p]
\centering\includegraphics[width=0.2\textwidth]{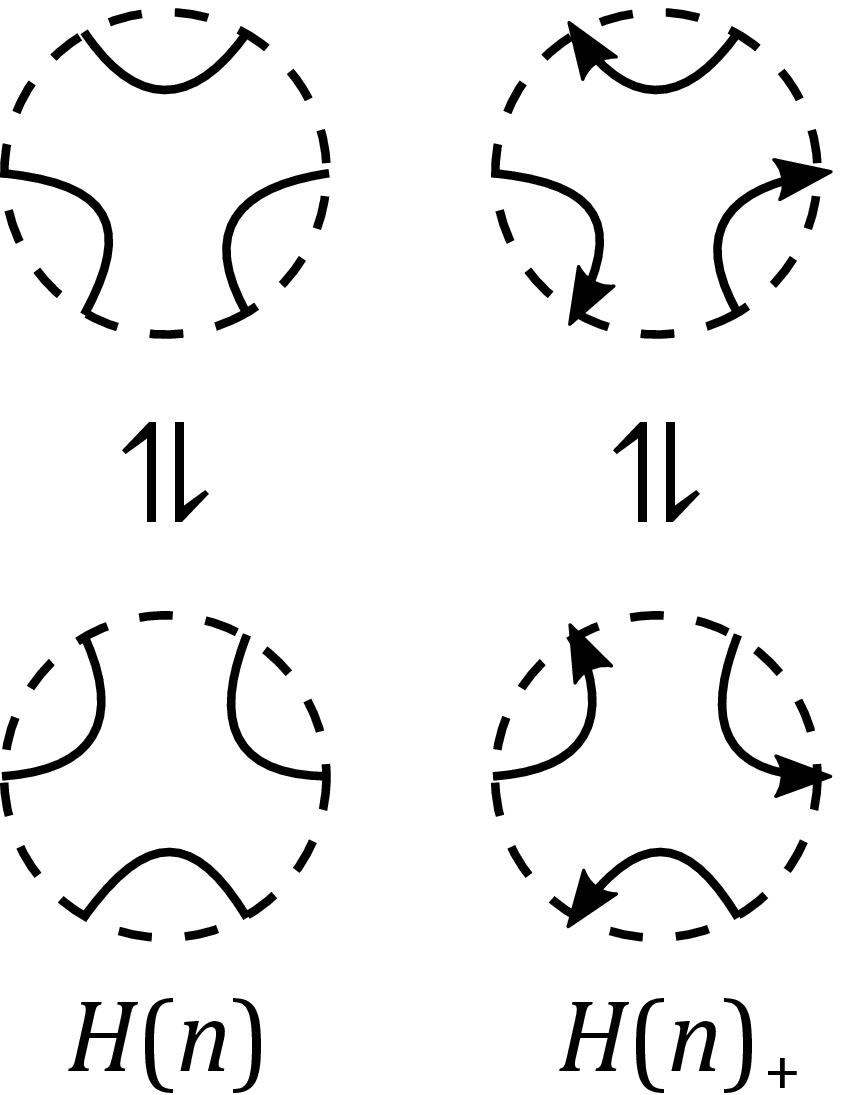}\qquad
\includegraphics[width=0.3\textwidth]{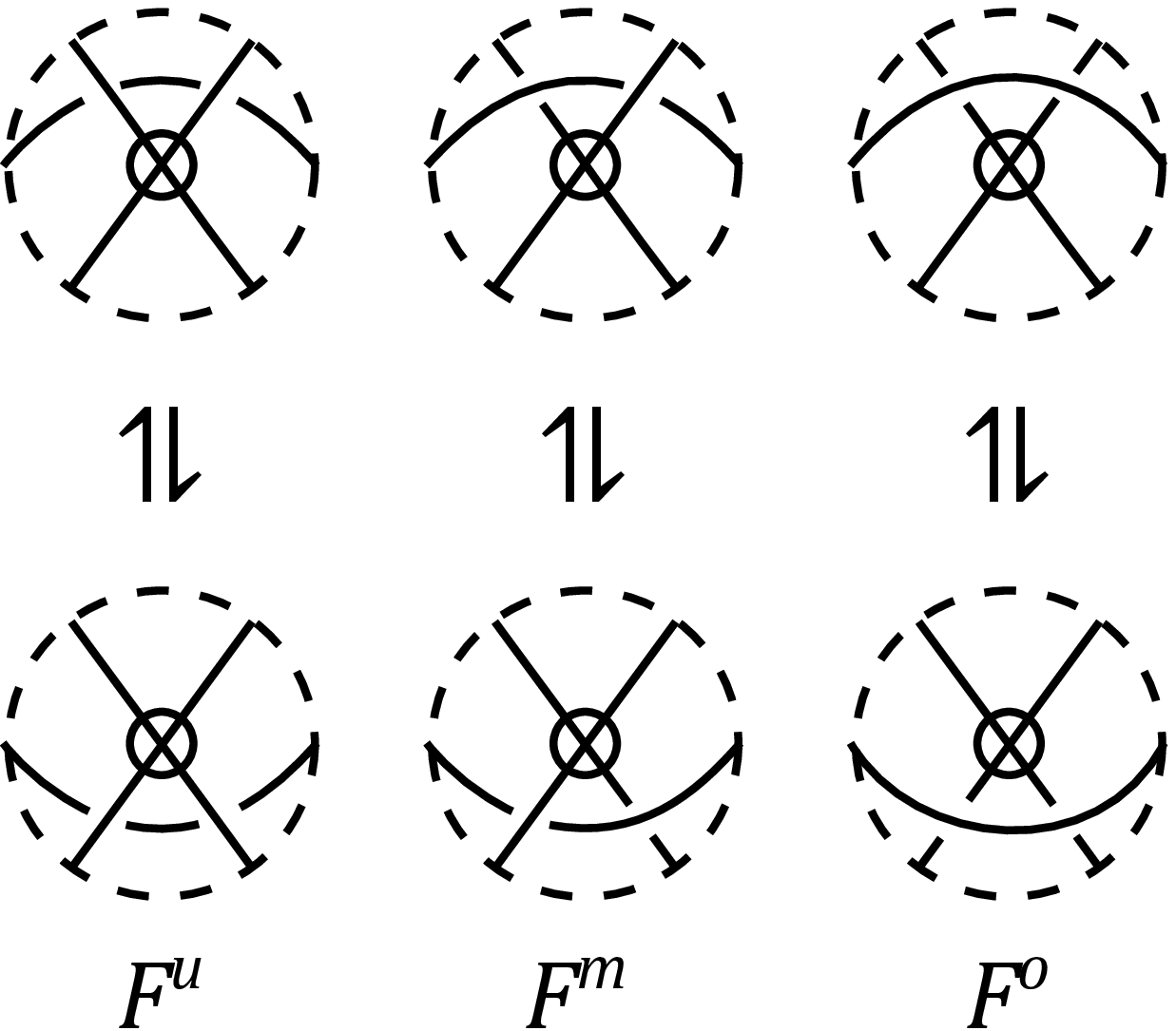}
\caption{$H(n)$-moves and forbidden virtual moves}\label{fig:Hn_forbidden_moves}
\end{figure}
\clearpage
}

\begin{remark}
We will also consider local moves $M=(T_1,T_2)$ with $T_2\in A[\mathcal T_n]$ where $A$ is a coefficient ring. In this case the result of the local move applied to a diagram is a linear combination of diagrams with coefficients in $A$.
\end{remark}

Examples of local moves are:
\begin{itemize}
\item (unoriented) Reidemeister moves $\Omega_1$, $\Omega_2$ , $\Omega_3$ (Fig.~\ref{fig:reidemeister_moves}). Each Reidemeister move actually presents a pair of moves which are inverse to each other;
\item (oriented) flat Reidemeister moves $\bar\Omega_1$, $\bar\Omega_2$, $\bar\Omega_3$ (Fig.~\ref{fig:reidemeister_moves_flat_oriented});
\item virtual Reidemeister moves $V\Omega_1$, $V\Omega_2$ , $V\Omega_3$, $SV\Omega_3$ (Fig.~\ref{fig:reidemeister_moves_virtual});
\item the flanking move $Fl$ (Fig.~\ref{fig:flanking_move});
\item the trivial circle reduction move $O_\delta$ (Fig.~\ref{fig:circle_remove});
\item smoothing moves $Sm^{or}$, $Sm^{unor}$, $Sm^{A}$, $Sm^{B}$ (Fig.~\ref{fig:smoothing_moves});
\item crossing change $CC$, clasp move $Cl$ and $\Delta$-move~\cite{Hab,Mat,MN} (Fig.~\ref{fig:crossing_change_moves});
\item $H(n)$-moves~\cite{HNT} (Fig.~\ref{fig:Hn_forbidden_moves} left);
\item forbidden virtual moves $F^u$, $F^m$, $F^o$~\cite{GPV} (Fig.~\ref{fig:Hn_forbidden_moves} right).
\end{itemize}

\begin{definition}\label{def:knot_theory}
We say that a set of local moves $\mathcal M$ is \emph{finer} than another set of local moves $\mathcal M'$ if any local move $M\in\mathcal M$ can be expressed by a sequence moves from $\mathcal M'$ or their inverses applied to some diagram. Two sets of local moves $\mathcal M$ and $\mathcal M'$ are \emph{equivalent} if $\mathcal M$ is finer than $\mathcal M'$ and $\mathcal M'$ is finer than $\mathcal M$.
A \emph{knot theory} is a set of local moves considered up to equivalence.
\end{definition}

\begin{example}
Reidemeister knot theory and its generalizations can be defined as follows:
\begin{itemize}
\item $\mathcal M_{class}=\{ \Omega_1, \Omega_2, \Omega_3 \}$ is the classical knot theory;
\item $\mathcal M_{virt}=\mathcal M_{class}\cup\{ V\Omega_1, V\Omega_2, V\Omega_3, SV\Omega_3 \}$ is the virtual knot theory~\cite{K1};
\item $\mathcal M_{flat}=\mathcal M_{virt}\cup\{CC\}$ is the flat knot theory~\cite{KK};
\item $\mathcal M_{free}=\mathcal M_{flat}\cup\{Fl\}$ is the free knot theory~\cite{T1};
\item $\mathcal M_{class}^{reg}=\{\Omega_2, \Omega_3 \}$ is the regular classical knot theory;
\item $\mathcal M_{class}^{+}=\{ \Omega_{1a}, \Omega_{1b}, \Omega_{2a}, \Omega_{3a} \}\sim \{ \Omega_{1a}, \Omega_{1b}, \Omega_{2c}, \Omega_{2d}, \Omega_{3b} \}$ is the oriented classical knot theory~\cite{P};
\item $\mathcal M_{class}^{reg+}=\{ \Omega_{2a}, \Omega_{2b}, \Omega_{2c}, \Omega_{2d}, \Omega_{3b} \}$ is the oriented regular classical knot theory.
\end{itemize}

Oriented virtual knot theory $\mathcal M_{virt}^+$ and oriented regular virtual knot theory $\mathcal M_{virt}^{reg+}$ can be defined analogously.
\end{example}

\begin{figure}[h]
\centering\includegraphics[width=0.33\textwidth]{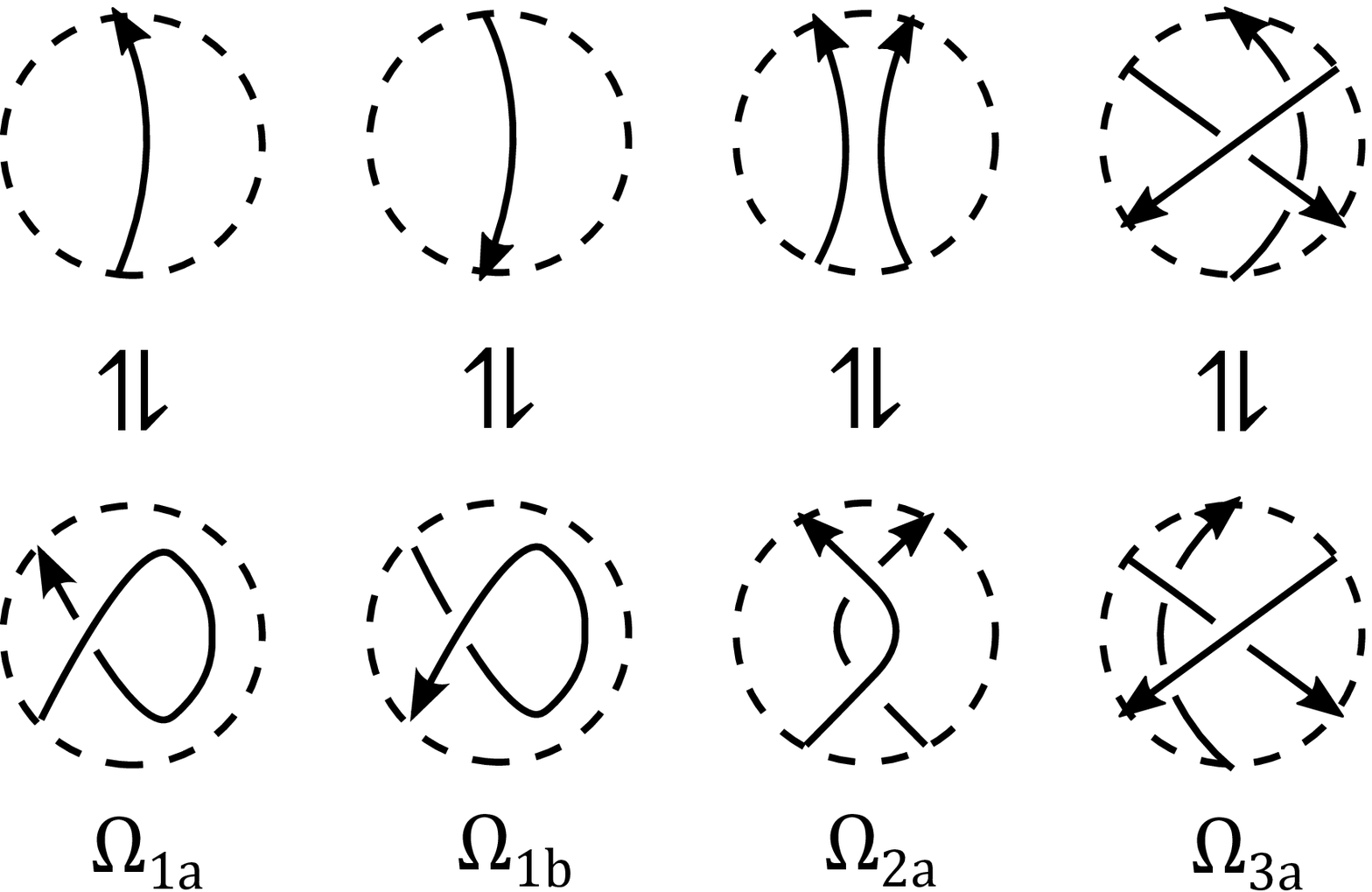}\qquad\qquad
\includegraphics[width=0.45\textwidth]{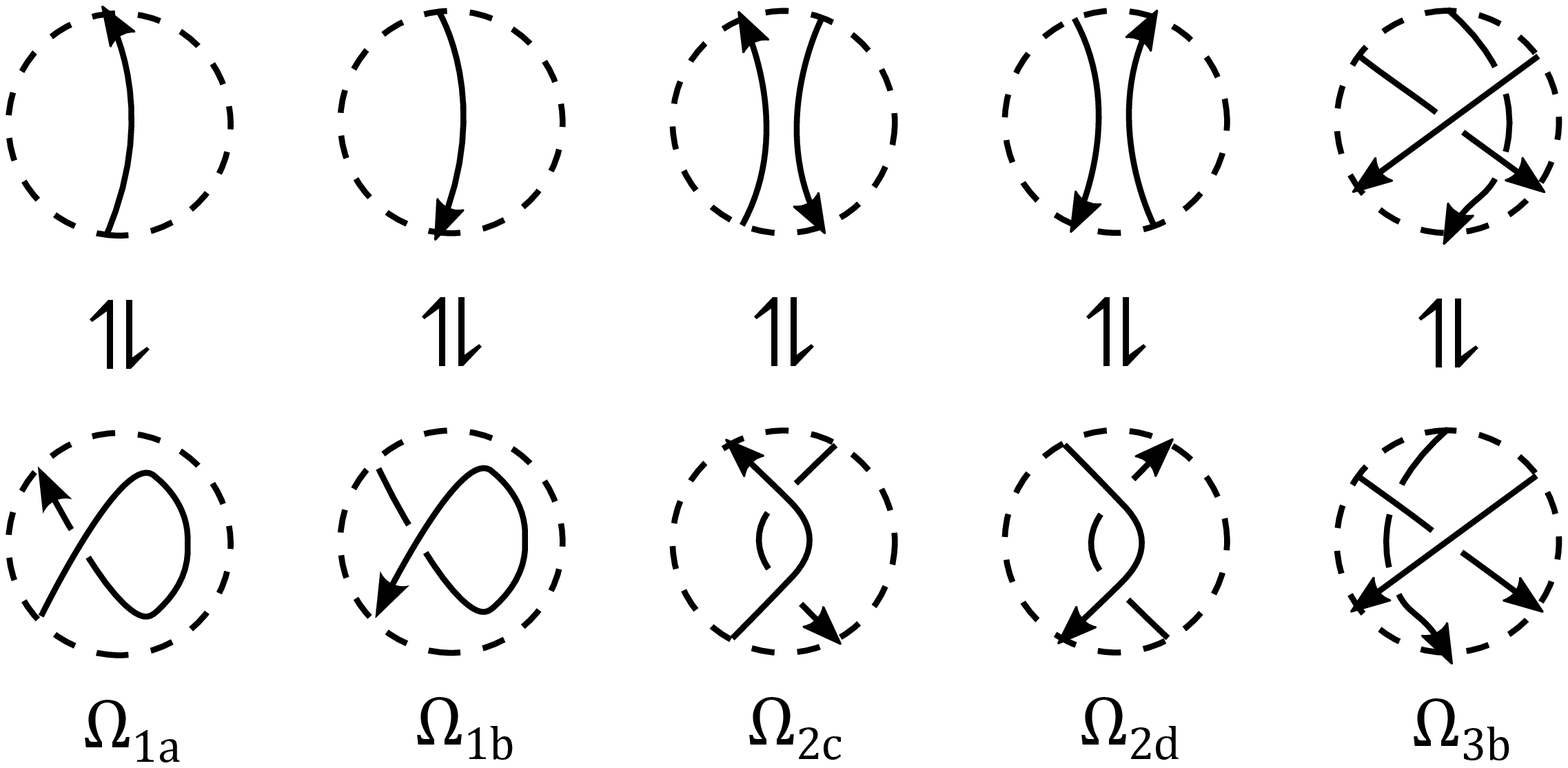}
\caption{Two sets of moves for the oriented classical knot theory}\label{fig:reidemeister_moves_or}
\end{figure}

\begin{figure}[h]
\centering\includegraphics[width=0.45\textwidth]{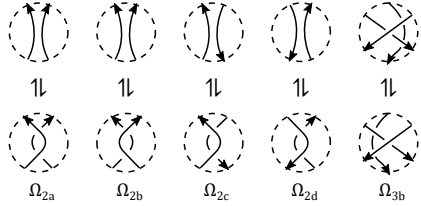}
\caption{The set of moves for the oriented regular classical knot theory}\label{fig:reidemeister_moves_reg_or}
\end{figure}

\begin{definition}\label{def:diagram_category}
Let $A$ be a ring, $\mathcal M$ a knot theory, $F$ a compact oriented surface and $\mathscr D$ a set of diagrams on $F$. Then we can construct three objects:
\begin{itemize}
  \item the \emph{set of knots} $\mathscr K(\mathscr D|\mathcal M)$ which is the set of the equivalence classes of diagrams $D\in\mathscr D$ modulo the local moves $M\in\mathcal M$;
  \item the \emph{skein module} $\mathscr S(\mathscr D|\mathcal M)=A[\mathscr D]/\left<D-R_M(D,T) \mid D\in \mathscr D, M\in\mathcal M\right>$;
  \item the \emph{diagram category} $\mathfrak K(\mathscr D|\mathcal M)$ whose objects are diagrams in $\mathscr D$ and the morphisms are local moves in $\mathcal M$ and their formal compositions.
\end{itemize}
\end{definition}

When $\mathscr D$ is a natural choice of the set of diagrams in $F$ (i.e. $\mathscr D=\mathscr D(F)$ for $\mathcal M=\mathcal M_{class}$)
 we will write $\mathscr K(F|\mathcal M)$, $\mathscr S(F|\mathcal M)$ etc.





\subsection{Traits of crossings}\label{subsec:traits}

Let us recall some notation from~\cite{Nct,Npi}.

Let $\mathfrak K=\mathscr K(\mathscr D|\mathcal M)$ be a classical-type diagram category 
(for example, the diagram category of classical, virtual, flat or free knots). The \emph{crossing functor} is the correspondence $D\mapsto\mathcal C(D)$, $D\in Ob(\mathfrak K)$, where $\mathcal C(D)$ is the set of classical crossings. Any morphism $f\colon D\to D'$ defines a partial bijection $f_*\colon \mathcal C(D)\to \mathcal C(D')$ between the sets of classical crossings.

\begin{definition}[\cite{Npi}]\label{def:trait}
A \emph{trait} with coefficients in a set $\Theta$ is a set of maps $\theta_D\colon\mathcal C(D)\to \Theta$ such that for any Reidemeister move $f\colon D\to D'$ and any crossing $c\in\mathcal C(D)$ one has $\theta_D(c)=\theta_{D'}(f_*(c))$ (if $f_*(c)$ exists).
\end{definition}


\begin{definition}\label{def:trait_types}
Particular cases of traits are the following.
\begin{itemize}
\item An \emph{index} is a trait $\iota$ such that $\iota(c_1)=\iota(c_2)$ for any crossings $c_1,\,c_2$ to which a decreasing second Reidemeister move can be applied.

\item A \emph{signed index} with coefficients in a set $I$ with an involution $\ast\colon I\to I$ is a trait $\sigma$ such that for any crossings $c_1,\,c_2$ to which a decreasing second Reidemeister move can be applied, one has $\sigma(c_1)=\sigma(c_2)^\ast$.
\item A \emph{parity} with coefficients in an abelian group $A$ is a trait $p$ such that for any Reidemeister move the sum of parities of the crossings that take part in the move is equal to zero (Fig.~\ref{fig:parity_axioms}).
\item A \emph{weak parity} is a trait with coefficients in $\Z_2$ such that for any Reidemeister move, the number of odd crossings among the ones which take part in the move can not be equal to $1$.
\end{itemize}
\end{definition}

Any parity is a signed index (the involution here is $x^\ast=-x$, $x\in A$), and a weak parity is an index.

\begin{definition}\label{def:universal_trait}
Let $\mathfrak K$ be a diagram category. A trait $\theta^u$ with coefficients in a set $\Theta^u$ is called the \emph{universal trait} on $\mathfrak K$ if for any trait $\theta$ on $\mathfrak K$ with coefficients in a set $\Theta$ there is a unique map $\psi\colon \Theta^u\to \Theta$ such that $\theta=\psi\circ\theta^u$.
\end{definition}

\begin{remark}
The universal coefficient set $\Theta^u$ can be described as the set of equivalence classes of pairs $(D,c)$, $D\in Ob(\mathfrak K)$, $c\in \mathcal C(D)$, modulo moves which do not eliminate the crossing $c$.
\end{remark}

\begin{theorem}[\cite{Npi}]\label{thm:universal_trait_index}
Let $\theta^u$ be the universal trait on a diagram category $\mathfrak K$. Then $\theta^u$ is the universal signed index on $\mathfrak K$.
\end{theorem}

We will use the following consequence of the theorem.

\begin{proposition}\label{prop:trait_dual_crossing}
Let $\tau$ be a trait on a diagram category $\mathfrak K=\mathscr K(\mathscr D|\mathcal M)$, $D\in \mathscr D$ and $c\in\mathcal C(D)$ a crossing of the diagram $D$. Let $f_i\colon D\to D_i$, $i=1,2$, be morphisms which do not eliminate the crossing $c$, $c_i=(f_i)_*(c)\in\mathcal C(D_i)$ the corresponding crossings, and $c'_i\in\mathcal C(D_i)$, $i=1,2$, be crossings such that a second Reidemeister move can be applied to $c_i$ and $c'_i$. Then $\tau_{D_1}(c'_1)=\tau_{D_2}(c'_2)$.
\end{proposition}
\begin{proof}
By Theorem~\ref{thm:universal_trait_index} $\theta^u(c'_i)=\theta^u(c_i)^\ast=\theta^u(c)^\ast$. Since $\theta^u$ is universal there exists a map $\psi$ between the coefficient sets such that $\tau=\psi\circ \theta^u$. Then $\tau(c'_1)=\psi(\theta^u(c'_1))=\psi(\theta^u(c'_2))=\tau(c'_2)$.
\end{proof}

Let $F$ be an oriented connected compact surface with the boundary $\partial F$. Consider diagrams of oriented tangles with numbered components. Classical Reidemeister moves on these diagrams define the theory of oriented tangles with numbered components in the surface $F$. Let us describe the universal trait for this theory.


Let $D=D_1\cup\cdots\cup D_n$ be an oriented tangle diagram and $v$ a crossing in $D$. Let $v$ be an intersection of components $D_i$ and $D_j$, and $D_i$ be the overcrossing at $v$. Then the \emph{component index} of $v$ is $\tau(v)=(i,j)\in \{1,\dots,n\}^2$.

\begin{figure}[h]
\centering\includegraphics[width=0.12\textwidth]{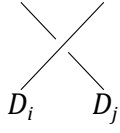}
\caption{A crossing with the component index $\tau(v)=(i,j)$}\label{fig:crossing_type}
\end{figure}

If $v$ is a self-intersection of a long component one defines the \emph{order index} $o(v)\in\Z_2$ depending on whether $v$ is early under- or overcrossing (Fig.~\ref{fig:order_index}).

\begin{figure}[h]
\centering\includegraphics[width=0.4\textwidth]{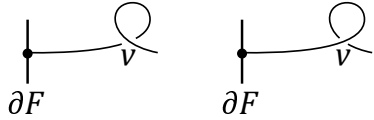}
\caption{Early undercrossing ($o(v)=-1$) and early overcrossing ($o(v)=1$) }\label{fig:order_index}
\end{figure}

Choose a non-crossing point $z_i$ on each component $D_i$ of the diagram. If the component $D_i$ is long then we suppose that $z_i$ is the initial point of the component. For $1\le i,j\le n$ denote the set of homotopy classes of paths from $z_i$ to $z_j$ by $\pi(F,z_i,z_j)$.

The tangle diagram $D$ in $F$ can be considered as the image of a $1$-dimensional manifold $T$ embedded in the thickening $F\times [0,1]$ of the surface $F$ under the natural projection $p\colon F\times [0,1]\to F$. Let $y_i=p^{-1}(z_i)\cap T$, $1\le i\le n$. Any diffeomorphism $\Phi$ on $F\times [0,1]$ such that $\Phi$ is isotopic to the identity, $\Phi(T)=T$ and $\Phi(y_i)=y_i$, induces automorphisms $\Phi_{ij}$ on the sets $\pi(F,z_i,z_j)$. Denote the groups of such automorphisms by $IM_{ij}\subset Aut(\pi(F,z_i,z_j))$.
Let $\hat\pi(F,z_i,z_j)=\pi(F,z_i,z_j)/IM_{ij}$ be the space of orbits under this action.

Let $v$ be a crossing of $D$. If $v$ is a self-crossing of a component $D_i$ then its \emph{homotopy index} is the homotopy class $h(v)=[D_v]\in \hat\pi(F,z_i,z_i)$ of the loop formed by the diagram at the crossing $v$. If $D_i$ is a closed component we assume the loop goes from the undercrossing to the overcrossing (Fig.~\ref{fig:homotopy_index} middle). If $v$ is a mixed crossing of components $D_i$ and $D_j$ and $\tau(v)=(i,j)$ then the homotopy index is the class of the path which goes in the diagram from $z_i$ by the crossing $v$ to $z_j$: $h(v)=[\gamma_{z_i,v}\gamma_{z_j,v}^{-1}]\in\hat\pi(F,z_i,z_j)$.

\begin{figure}[h]
\centering\includegraphics[width=0.25\textwidth]{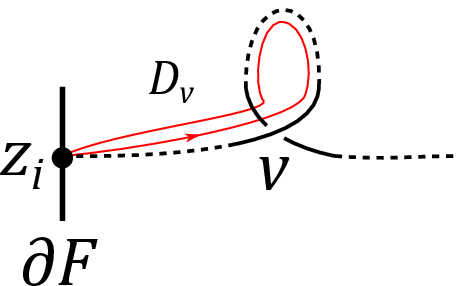}\quad \includegraphics[width=0.22\textwidth]{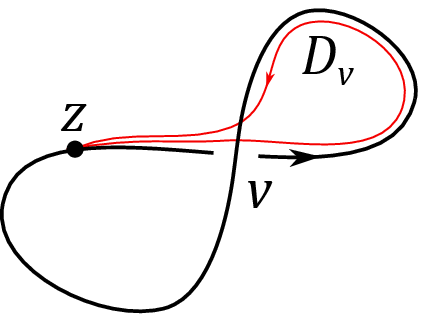}
\quad \includegraphics[width=0.25\textwidth]{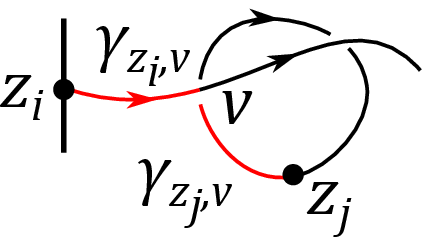}
\caption{Homotopy index for long, closed and mixed case}\label{fig:homotopy_index}
\end{figure}

\begin{theorem}[\cite{Nct}]\label{thm:universal_index}
The universal index $\iota^u$ on the diagram category of tangles in the surface $F$ is composed by the component, order and homotopy indices: $\iota^u=(\tau,o,h)$. The universal signed index is $\sigma^u=(sgn, \tau,o,h)$.
\end{theorem}

An analogous result holds for flat tangles.

Let $D=D_1\cup\cdots\cup D_n$ be a diagram of a flat tangle $T$ and $v$ a crossing of $D$. Then $v$ is an intersection point of some components $D_i$ and $D_j$. We order the component according to the orientation as shown in Fig.~\ref{fig:refined_flat_component_type}. The \emph{flat component signed index} of the crossing $v$ is the ordered pair $\tau^f(v)=(i,j)$.

\begin{figure}[h]
\centering
  \includegraphics[width=0.12\textwidth]{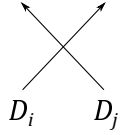}
  \caption{A crossing with flat component index $(i,j)$}\label{fig:refined_flat_component_type}
\end{figure}

Let $v$ be a self-crossing of a long component $D_i$. The crossing $v$ split the component $D_i$ into two halves one of which is closed and the other is long.
We define the \emph{flat order signed index} $o^f(v)$ of the crossing $v$ to be equal to $+1$ if the closed half is left and to be equal to $-1$ if it is right (Fig.~\ref{fig:refined_flat_order_type})

\begin{figure}[h]
\centering
  \includegraphics[width=0.4\textwidth]{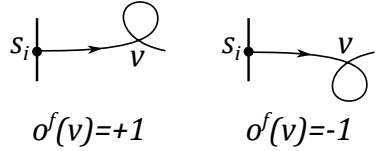}
  \caption{Flat order signed index}\label{fig:refined_flat_order_type}
\end{figure}

The \emph{flat homotopy signed index} $h^f(v)$ of the crossing $c$ is defined as in the non flat case when $v$ is a mixed crossing or a self-crossing of a long component: $h^f(v)=h(v)$. When $v$ is a self-crossing of a closed component $D_i$, we set $h^f(c)=[D^l_v]\in\hat\pi(F,z_i,z_i)$ where $D^l_v$ is the based left half of the component at the crossing $v$ (Fig.~\ref{fig:based_left_half}) and $\hat\pi(F,z_i,z_i)=\pi(F,z_i,z_i)/IM_{ii}(T)$,  where $IM_{ii}(T)\subset Aut(\pi(F,z_i,z_i))$ is the group of automorphisms induced by the symmetries of the tangle $T$.

\begin{figure}[h]
\centering
  \includegraphics[width=0.3\textwidth]{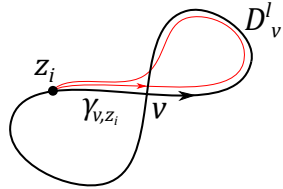}
  \caption{Based left half $\hat D^l_{v,z_i}$}\label{fig:based_left_half}
\end{figure}

We consider the following involutions on the sets which the signed indices take values in:
\[
(i,j)^*=(j,i), \quad 1\le i,j\le n,
\]
for the flat component signed index,
\[
o^*=-o,\quad o\in\{-1,+1\},
\]
for the flat order signed index, and
\[
\bar x^*=\overline{\kappa_{i}x^{-1}},\quad \bar x\in\hat\pi_D(F,z_i,z_i)
\]
for the flat homotopy signed index of self-crossings of a closed component. Here $\kappa_i=[D_i]\in\hat\pi_D(F,z_i,z_i)$ is the homotopy class of the component.

We define the involution to be the identity for the flat homotopy signed index of self-crossings of long components and the inverse map  $\hat\pi(F,z_i,z_j)\to\hat\pi(F,z_j,z_i)$ for mixed crossings.

\begin{theorem}[\cite{Nct}]\label{thm:universal_flat_index}
The universal signed index $\bar\sigma^u$ on the diagram category of flat tangles in the surface $F$ is composed by the flat component, order and homotopy signed indices: $\bar\sigma^u=(\tau^f,o^f,h^f)$.
\end{theorem}

\section{Local transformations}\label{sec:local_transformations}

\subsection{Local transformation rule}\label{subsec:tangle_traits}


\begin{definition}
Let  $\mathfrak K=\mathscr K(\mathscr D|\mathcal M)$ be a diagram category and $A$ a commutative ring. A \emph{local transformation rule} is a trait $\tau$ with values 
in the free module $A[\mathcal T_2]$.
\end{definition}

\begin{figure}[h]
\centering\includegraphics[width=0.7\textwidth]{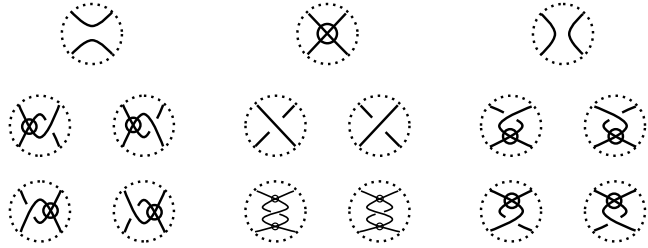}
\caption{Simplest $2$-tangles 
}\label{pic:tangles2}
\end{figure}

Given a diagram $D\in\mathscr D$ and a subset of crossings $C\subset \mathcal C(D)$, we denote the result of application of the local moves $c\mapsto \tau_D(c)$ at all the crossings $c\in C$ by $R_\tau(D,C)$. Denote $R_\tau(D)=R_\tau(D,\mathcal C(D))$ and $R_\tau(D,c)=R_\tau(D,\{c\})$ for $c\in\mathcal C(D)$.

\begin{definition}
Let $\mathfrak K=\mathscr K(\mathscr D|\mathcal M)$ be a diagram category and $\tau$ a local transformation rule.
The \emph{$\tau$-derivation} of a diagram $D\in\mathscr D$ is the sum
\[
d_\tau(D)=\sum_{c\in\mathcal C(D)}R_{\tau}(D,c),
\]
and the \emph{$\tau$-functorial map} of $D$ is $f_\tau(D)=R_\tau(D)$.

Let $\mathscr D'$ be a set of diagrams and $A$ a ring such that for any $D\in\mathscr D$ the sum $d_\tau(D)\in A[\mathscr D']$. Let $\mathcal M'$ be a knot theory on $\mathscr D'$. We say that $d_\tau$ is \emph{invariant} if for any diagrams $D_1,D_2\in\mathscr D$ such that $D_1\sim_{\mathcal M}D_2$ we have $d_\tau(D_1)=d_\tau(D_2)\in\mathscr S(\mathscr D'|\mathcal M')$. The invariance of $f_\tau$ is defined analogously.
\end{definition}

\begin{remark}
If for all crossings the values $\tau(c)$ are 1-term combinations of tangles then $f_\tau(D)$ defines a map to $\mathscr K(\mathscr D'|\mathcal M')$ (or $\mathscr K(\mathscr D'|\mathcal M')\cup\{0\}$).
\end{remark}

\subsection{Derivations}\label{subsec:derivations}


\begin{definition}
Let $\mathcal M$ be a knot theory. A $2$-tangle $T\in A[\mathcal T_2]$ is called \emph{$\Omega_3^+$-compatible} if the tangle diagrams in Fig.~\ref{fig:R3+compatibility} are $\mathcal M$-equivalent. The set of $\Omega_3^+$-compatible $2$-tangles is denoted by $\mathcal T_2^{\Omega_3^+}(\mathcal M)$.

The $\Omega_3^-$-compatibility condition is obtained by switching all the crossing outside the tangle $T$ in Fig.~\ref{fig:R3+compatibility}.

\begin{figure}[h]
\centering\includegraphics[width=0.3\textwidth]{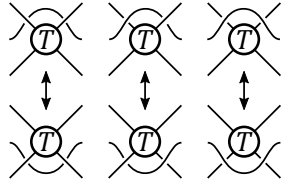}
\caption{$\Omega_3^+$-compatibility condition}\label{fig:R3+compatibility}
\end{figure}
\end{definition}

Consider the maps $i_\pm\colon\mathcal T_2^{\Omega_3^\pm}(\mathcal M)\to \mathcal T_2^{\Omega_3^\mp}(\mathcal M)$ defined in Fig.~\ref{fig:dual_tangles}. If $\Omega_2\in\mathcal M$ then the maps $i_+$ and $i_-$ are inverse to each other.

\begin{figure}[h]
\centering\includegraphics[width=0.7\textwidth]{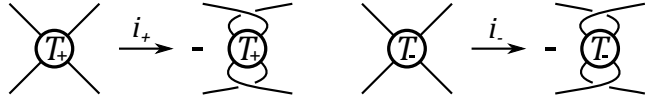}
\caption{Involution on $\Omega_3^\pm$-compatible tangles}\label{fig:dual_tangles}
\end{figure}

Let $\tau$ be a trait. Denote the trait values of the loops of type $l_\pm$ and $r_\pm$ (Fig.~\ref{fig:loop_types}) by $\tau^{l\pm}$ and $\tau^{r\pm}$.

\begin{figure}[h]
\centering\includegraphics[width=0.3\textwidth]{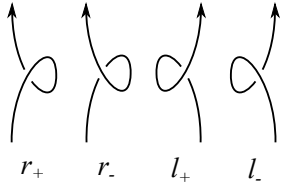}
\caption{Types of loops}\label{fig:loop_types}
\end{figure}

Let $Ann^*(\mathcal M)=\{T\in A[\mathcal T_2]\mid Cl^*(T)=0\in\mathscr S(\mathbb D^2|\mathcal M)\}$, $*=l,r$, where $Cl^r(T)$ and $CL^r(T)$  are the closures of the $2$-tangle $T$ (Fig.~\ref{fig:tangle_closures}).

\begin{figure}[h]
\centering\includegraphics[width=0.5\textwidth]{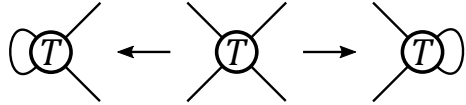}
\caption{A tangle $T$ and its left closure $Cl^l(T)$ and right closure $Cl^r(T)$}\label{fig:tangle_closures}
\end{figure}

\begin{theorem}\label{thm:derivation_invariance}
Let $\mathfrak K=\mathfrak(\mathscr D|\mathcal M)$ be a diagram category with $\mathcal M=\mathcal M_{class}^+$ or $\mathcal M_{virt}^+$, $\tau$ a local transformation rule on it and $\mathcal M'$ a knot theory such that $\mathcal M\subset \mathcal M'$. If
\begin{enumerate}
\item $\tau_D(c)\in\mathcal T^{\Omega_3^+}_2(\mathcal M')$ for any positive crossing $c\in\mathcal C(D)$, $D\in\mathscr D$;
\item for any crossings $c_1,c_2\in\mathcal C(D)$ participating in a second Reidemeister move, $i_{sgn(c_1)}(\tau(c_1))=\tau(c_2)$;
\item $\tau^{r+}\in Ann^{r}(\mathcal M')$ and $\tau^{l+}\in Ann^l(\mathcal M')$;
\end{enumerate}
then the derivation $d_\tau$ is invariant.
\end{theorem}

\begin{proof}
We need to check that for any diagrams $D$ and $D'$ connected by a Reidemeister move the sums $d_\tau(D)$ and $d_\tau(D')$ coincide in the skein module.

The terms in the sums $d_\tau(D)$ and $d_\tau(D')$ for a crossing $c\in\mathcal C(D)$ which does not participate in the Reidemeister move, and the correspondent crossing $c'\in\mathcal C(D')$, are equal because the knot theory $\mathcal M'$ includes Reidemeister moves.  Let us look at the terms for the crossings participating in the move.

1. If the move is a decreasing first Reidemeister move which eliminates a crossing $c\in\mathcal C(D)$ then $d_\tau(D)$ contains an additional term $R_\tau(D,c)$. The crossing $c$ is a loop crossing, and the diagram $R_\tau(D,c)$ contains a closure $Cl^{l/r}(\tau(c))$ as a subtangle. Then  $R_\tau(D,c)=0$ by the third condition of the theorem.

2. Let the move be a second Reidemeister move $\Omega_{2c}$ and crossings $c_1, c_2\in\mathcal C(D)$ participate in the move. Then  $d_\tau(D)-d_\tau(D')=R_\tau(D,c_1)+R_\tau(D,c_2)$ (Fig.~\ref{fig:dual_tangle_R2_eq}).

\begin{figure}[h]
\centering\includegraphics[width=0.7\textwidth]{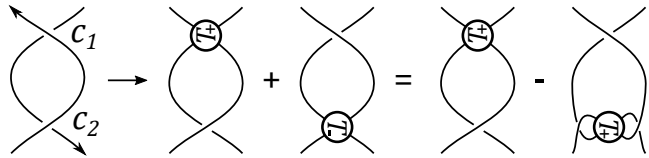}
\caption{Additional terms in the move $\Omega_{2c}$}\label{fig:dual_tangle_R2_eq}
\end{figure}

The second tangle is equivalent to the first one (Fig.~\ref{fig:dual_tangle_R2_second_term}), hence, the terms annihilate and $d_\tau(D)=d_\tau(D')$.

\begin{figure}[h]
\centering\includegraphics[width=0.7\textwidth]{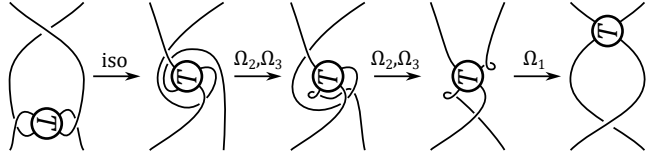}
\caption{Tangle equivalence. The second and the third equivalences rely on $\Omega_3^+$-compatibility}\label{fig:dual_tangle_R2_second_term}
\end{figure}

The move $\Omega_{2d}$ is considered analogously.

3. If the move is a third Reidemeister move $\Omega_{3c}$, $c_1,c_2,c_3\in\mathcal C(D)$ are the crossings taking part in the move, and $c'_1,c'_2,c'_3\in\mathcal C(D')$ are the correspondent crossings in $D'$ then the additional terms are equal: $R_\tau(D,c_i)=R_\tau(D',c'_i)$ by $\Omega_3^+$-compatibility. Hence, $d_\tau(D)=d_\tau(D')$.
\end{proof}

For derivations on regular knots we have an analogous statement.

\begin{theorem}\label{thm:derivation_invariance_regular}
Let $\mathfrak K=\mathfrak(\mathscr D|\mathcal M)$ be a regular classical or virtual diagram category ($\mathcal M=\mathcal M_{class}^{reg+}$ or $\mathcal M_{virt}^{reg+}$), $\tau$ a local transformation rule on it and $\mathcal M$ a knot theory such that $\mathcal M\subset \mathcal M'$. If
\begin{enumerate}
\item $\tau_D(c)\in\mathcal T^{\Omega_3^+}_2(\mathcal M')$ for any positive crossing $c\in\mathcal C(D)$, $D\in\mathscr D$;
\item for any crossings $c_1,c_2\in\mathcal C(D)$ participating in a second Reidemeister move, $i_{sgn(c_1)}\tau(c_1)=\tau(c_2)$;
\item $\tau_D(c)$ is $\Omega_1^+$-compatible (Fig.~\ref{fig:R1+compatibility}) for any positive crossing $c\in\mathcal C(D)$, $D\in\mathscr D$,
\end{enumerate}
then the derivation $d_\tau$ is invariant.
\begin{figure}[h]
\centering\includegraphics[width=0.3\textwidth]{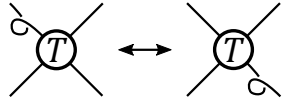}
\caption{$\Omega_1^+$-compatibility condition}\label{fig:R1+compatibility}
\end{figure}
\end{theorem}

\begin{proof}
We can use the proof of Theorem~\ref{thm:derivation_invariance} for second and third Reidemeister moves but in the case of the move $\Omega_{2c}$ when we prove that the two additional terms in the sum coincide up to the sign, we should use $\Omega_1^+$-compatibility (and Whitney trick removing a pair of loops with moves $\Omega_2$ and $\Omega_3$) instead of move $\Omega_1$ in the last equivalence in Fig.~\ref{fig:dual_tangle_R2_second_term}.
\end{proof}


\begin{example}[Glueing invariant (A. Henrich~\cite{Henrich}, P. Cahn~\cite{C})]

Consider a virtual diagram set $\mathscr D_{virt}$ and the virtual knot theory $\mathcal M_{virt}$ on it. The glueing rule $\tau_{glue}$ in Fig.~\ref{fig:singular_rule} defines a local transformation to virtual diagrams with one singular crossing $\mathscr D_{sing}$. The second term in the rule ensures the invariance under first Reidemeister moves.
\begin{figure}[h]
\centering\includegraphics[width=0.3\textwidth]{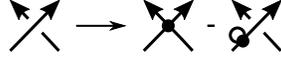}
\caption{Glueing local transformation}\label{fig:singular_rule}
\end{figure}

Let $\mathcal M'$ includes $\mathcal M_{virt}$, the crossing change $CC$ and Reidemeister moves for the singular crossing (Fig.~\ref{fig:singular_reidemeister}).
Then $d_{\tau_{sing}}$ is an invariant derivation with values in flat knots with one singular crossing.

\begin{figure}[h]
\centering\includegraphics[width=0.5\textwidth]{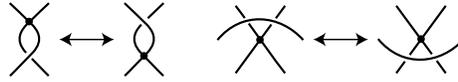}
\caption{Singular Reidemeister moves}\label{fig:singular_reidemeister}
\end{figure}
\end{example}

\begin{example}[Smoothing invariants (A. Henrich~\cite{Henrich}, Z. Cheng, H. Gao, M. Xu~\cite{CGX})]

Consider a virtual diagram set $\mathscr D_{virt}$ and the virtual knot theory $\mathcal M_{virt}$ on it. The smoothing local transformations $\tau_{or}$ and $\tau_{unor}$ (Fig.~\ref{fig:smoothing_rules}) define invariants with respect to the flat knot theory $\mathcal M_{flat}=\mathcal M_{virt}\cup\{CC\}$.
\begin{figure}[h]
\centering\includegraphics[width=0.3\textwidth]{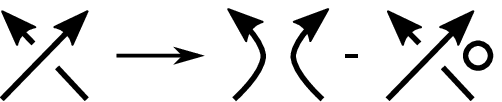}\qquad
\includegraphics[width=0.3\textwidth]{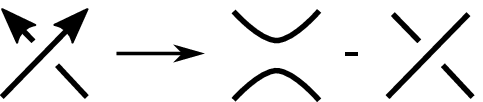}
\caption{Oriented (left) and unoriented (right) smoothing local transformations}\label{fig:smoothing_rules}
\end{figure}
%
\end{example}

A more flexible series of invariant derivations can be constructed by means of crossing indices.

\begin{proposition}\label{prop:index_to_derivation}
Let $\mathcal M$ be an oriented classical or virtual knot theory on a diagram set $\mathscr D$, $A$ a ring and $\mathcal M'\supset \mathcal M$ another knot theory. Let $\{T_1,\dots,T_n\}\subset\mathcal T_2^{\Omega_3^+}(\mathcal M')$ and $\iota_1,\dots, \iota_n$ be indices on $\mathscr D$ with values in $A$ such that
for any $i=1,\dots, n$ either the loop values vanish $\iota_i^{r\pm}=\iota_i^{l\pm}=0$ or $T_i\in Ann^l(\mathcal M')\cap Ann^r(\mathcal M')$. Consider the trait defined by the formula
\[
\tau(c) = \sum_{i=1}^n \iota_i(c)\cdot T_i^{sgn(c)}
\]
where $T_i^{+1}=T_1$ and $T_i^{-1}$ is the dual tangle (Fig.~\ref{fig:dual_tangles} left).
Then the derivation $d_\tau$ is invariant with respect to the knot theory $\mathcal M'$.
\end{proposition}

The proof is a direct check of invariance under Reidemeister moves.

\begin{remark}
1. The previous two examples of derivations correspond to the case $n=1$ and $\iota_1\equiv 1$.

2. Recurrent $F$-polynomials of virtual knots in~\cite{GIPV} can be viewed as consecutive compositions of invariant derivations corresponding to the case $T_i=\skcrv$ or $\skcrh$ and a polynomial invariant of flat knots.
\end{remark}


\subsection{Functorial maps}\label{subsec:functorial_maps}

Let us formulate a sufficient invariance condition for functorial maps.


\begin{definition}
Let $\mathcal M$ be a knot theory and $A$ a commutative ring. Define the sets of ($A$-linear combinations of) $2$-tangles $\Omega_{1r}(\mathcal M), \Omega_{1l}(\mathcal M)\subset \mathcal T_2$, $\Omega_{2}(\mathcal M)\subset \mathcal T_2\times\mathcal T_2$ and $\Omega_{3}(\mathcal M)\subset \mathcal T_2\times\mathcal T_2\times\mathcal T_2$ which obey the equivalence relations in Fig.~\ref{fig:reidemeister_tuples} with respect to $\mathcal M$.
\end{definition}

\begin{figure}[h]
\centering\includegraphics[width=0.6\textwidth]{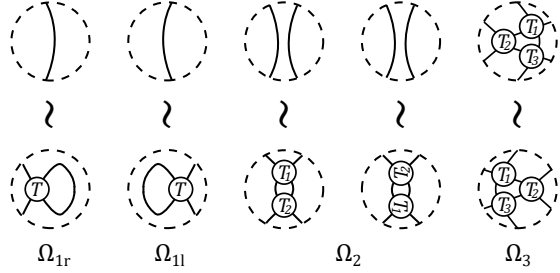}
\caption{Equivalence relations for the Reidemeister tuples}\label{fig:reidemeister_tuples}
\end{figure}

\begin{theorem}\label{thm:functorial_map_invariance}
Let $\mathcal M$ be the oriented classical knot theory on a diagram set $\mathscr D$, $\tau$ a local transformation rule on $\mathscr D$, and $\mathcal M'$ another knot theory. If
\begin{enumerate}
\item $\tau^{r\pm}\in\Omega_{1r}(\mathcal M')$, $\tau^{l\pm}\in\Omega_{1l}(\mathcal M')$;
\item $(\tau(c_1),\tau(c_2))\in\Omega_{2}(\mathcal M')$ for any crossings $c_1,c_2$ to which a second Reidemeister move can be applied;
\item $(\tau(c_1),\tau(c_2),\tau(c_3))\in\Omega_{3}(\mathcal M')$ for any crossings $c_1,c_2, c_3$ to which a third Reidemeister move $\Omega_{3}$ can be applied, 
\end{enumerate}
then the functorial map $f_\tau$ is invariant with respect to the knot theory $\mathcal M'$.
\end{theorem}

\begin{proof}
The invariance under Reidemeister moves directly follows from the definitions of the sets $\Omega_{i}(\mathcal M')$, $i=1,2,3$.
\end{proof}

\begin{remark}
Theorem~\ref{thm:functorial_map_invariance} is also valid when $\mathcal M=\mathcal M_{virt}^+$ and $\mathcal M'$ includes virtual Reidemeister moves and semivirtual third Reidemeister moves with a crossing of each type which appear in the tangle values $\tau(c)$.
\end{remark}

\begin{remark}
The condition of Theorem~\ref{thm:functorial_map_invariance} are necessary for the \emph{local invariance} of the functorial map when the equivalence of diagrams under a Reidemeister move must be established locally, in the disk where the Reidemeister move occurs. Note that diagrams can be equivalent by other reasons, for example, they can be unknot diagrams for some reason (e.g. take $\mathcal M'=\mathcal M_{virt}\cup\{F^u, F^o\}$).
\end{remark}

\begin{remark}
Derivations can be considered as particular case of functorial maps. Let  $\mathfrak K=\mathscr K(\mathscr D|\mathcal M)$ be a diagram category, $A$ a commutative ring and $\tau$ a local transformation rule valued in $A[\mathcal T_2]$. Consider the dual number ring $\tilde A=A[\epsilon]/(\epsilon^2)$, and define a new transformation rule $\tilde\tau$ with values in $\tilde A$ by the formula $\tilde\tau=id+\epsilon\cdot\tau$. Then $f_{\tilde\tau}(D)=D+\epsilon\cdot d_\tau(D)$. If the destination knot theory $\mathcal M'\supset \mathcal M$ then invariance of the functorial map $f_{\tilde\tau}$ is equivalent to invariance of the derivation $d_\tau$.
\end{remark}

The version of Theorem~\ref{thm:functorial_map_invariance} for regular knots formulates as follows.

\begin{theorem}\label{thm:functorial_map_invariance_regular}
Let $\mathcal M=\mathcal M_{class}^{reg+}$ be the oriented regular classical knot theory on a diagram set $\mathscr D$, $\tau$ a local transformation rule on $\mathscr D$, and $\mathcal M'$ another knot theory. If
\begin{enumerate}
\item $(\tau(c_1),\tau(c_2))\in\Omega_{2}(\mathcal M')$ for any crossings $c_1,c_2$ to which a second Reidemeister move can be applied;
\item $(\tau(c_1),\tau(c_2),\tau(c_3))\in\Omega_{3}(\mathcal M')$ for any crossings $c_1,c_2, c_3$ to which a third Reidemeister move $\Omega_{3}$ can be applied, 
\end{enumerate}
then the functorial map $f_\tau$ is invariant with respect to the knot theory $\mathcal M'$.
\end{theorem}

\section{Examples of functorial maps}\label{sec:functorial_maps_examples}

Below we consider simplest schemes of functorial maps on oriented classical or virtual knots and formulate invariance conditions for them. We also calculate the result of functorial maps on knots in a fixed surface.

\subsection{Unary functorial maps}\label{subsec:unary_functorial_maps}


\begin{definition}\label{def:unary_functorial_map}
Let $\{T_+, T_-\}\in\mathcal T_2$ (or $A[\mathcal T_2]$). The local transformation rule $\tau_D(c)=T_{sgn(c)}$ is called \emph{unary}.
\end{definition}

\begin{example}\label{ex:unary_maps}
\begin{enumerate}
\item $\tau=id$ is a unary local transformation rule which is invariant with respect to $\mathcal M_{class}$. The image of the functorial map is the knot it is applied to: $f_\tau(K)=K$;

    A more elaborated example appears if one adds the skein relations corresponding to a polynomial invariant (Conway, Jones, HOMFLY-PT etc.)  to the destination knot theory $\mathcal M'=\mathcal M_{class}$. Then the skein module can be identified with a Laurent polynomial ring, and the image of the functorial map is the Conway (Jones, HOMFLY-PT etc.) polynomial of the knot.

\item $\tau=CC$ is a unary local transformation rule which is invariant with respect to $\mathcal M_{class}$. The image of the functorial map is the mirror knot $f_\tau(K)=\bar K$;
\item The virtualizing map which replaces each classical crossing with the virtual one, is a unary local transformation rule which is invariant with respect to the pure virtual Reidemeister moves $\{V\Omega_1, V\Omega_2, V\Omega_3\}$. The image of the functorial map is the trivial knot (or trivial link).
\end{enumerate}
\end{example}

\subsubsection{Oriented smoothing}\label{subsect:unary_or_smoothing}

Let $T_+=T_-=\skcrv$. The invariance conditions for the corresponding local transformation rule $\tau$ are shown in Fig.~\ref{fig:unary_or_smoothing}.

\begin{figure}[h]
\centering\includegraphics[width=0.5\textwidth]{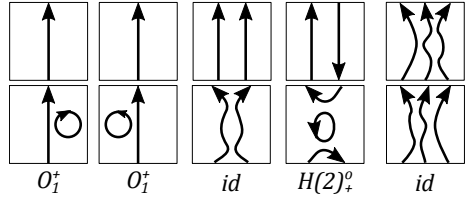}
\caption{Invariance conditions for the oriented smoothing}\label{fig:unary_or_smoothing}
\end{figure}
Thus, we can take $\mathcal M'=\{O^+_1, H(2)^o_+\}\sim\{H(2)_+\}$ for the destination knot theory.

For regular knots invariance of the functorial map will take place in the knot theory $\mathcal M'=\{H(2)^o_+\}$.

Let $F$ be an oriented compact surface. We have the following description of the oriented smoothing functorial map on links in the surface $F$.
\begin{proposition}\label{prop:unary_or_smoothing}
Let $D\in\mathscr D_+(F)$ be an oriented classical link diagram. Then
\begin{enumerate}
\item  $f_\tau(D)=[D]\in H_1(F,\Z)$ when $\mathcal M'=\{H(2)_+\}$;
\item  $f_\tau(D)=([D],rot(D))\in H_1(F,\Z)\times\Z_{\bar\chi(F)}$ in the regular case $\mathcal M'=\{H(2)^o_+\}$.
\end{enumerate}
Here $rot(D)$ is the rotation number of the diagram (see Definition~\ref{def:rotation_number}) and $\bar\chi(F)=\chi(F)$ if $\partial F=\emptyset$ and $\bar\chi(F)=0$ if $\partial F\ne\emptyset$.
\end{proposition}

\begin{proof}
By Proposition~\ref{prop:crossingless_skein_modules}, the element $f_\tau(D)\in\mathscr K^0(F|H(2)^o_+)$ is described by the pair
\[([f_\tau(D)],rot(f_\tau(D)))\in H_1(F,\Z)\times\Z_{\bar\chi(F)}.
\]
But $[D]=[f_\tau(D)]$ and $rot(D)=rot(f_\tau(D))$.
\end{proof}

\subsubsection{Non-oriented Smoothing}

Let $T_+=T_-=\skcrh$. The invariance conditions for the corresponding local transformation rule $\tau$ are shown in Fig.~\ref{fig:unary_unor_smoothing}.

\begin{figure}[h]
\centering\includegraphics[width=0.5\textwidth]{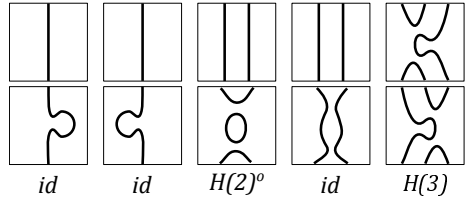}
\caption{Invariance conditions for the non-oriented smoothing}\label{fig:unary_unor_smoothing}
\end{figure}
Thus, we can take $\mathcal M'=\{H(2)^o, H(3)\}\sim\{H(2)^o\}$ (Lemma~\ref{lem:H20=H3}) for the destination knot theory.
For regular knots the knot theory $\mathcal M'$ fits too.

Let $F$ be an oriented compact surface. We have the following description of the unoriented smoothing functorial map on links in the surface $F$.
\begin{proposition}\label{prop:unary_unor_smoothing}
Let $D\in\mathscr D_+(F)$ be an oriented classical link diagram. Then
\[
f_\tau(D)=([D],\rho(D)+P'_D(1)+n(D))\in H_1(F,\Z_2)\times\Z_2
\]
when $\mathcal M'=\{H(2)^o\}$. Here $\rho(D)$ is the offset of the diagram, $P_D(t)$ is the based index polynomial (see Section~\ref{app:skein_modules}) and $n(D)$ is the number of crossings in $D$.
\end{proposition}

\begin{proof}
By Proposition~\ref{prop:crossingless_skein_unor_modules}, the element $f_\tau(D)\in\mathscr K^0(F|H(2)^o)$ is described by the pair
\[
([f_\tau(D)],\rho(f_\tau(D)))\in H_1(F,\Z_2)\times\Z_{2}.
\]
We have $[D]=[f_\tau(D)]$. By Corollary~\ref{cor:unor_invariants_after_smoothing}, $\rho(f_\tau(D))=\rho(D)+P'_D(1)+n(D)$.
\end{proof}

\subsubsection{$A$-smoothing}\label{subsect:A-smoothing}
Let $T_+=\skcrv,\ T_-=\skcrh$. The invariance conditions for the corresponding local transformation rule $\tau$ are shown in Fig.~\ref{fig:unary_A_smoothing}.

\begin{figure}[h]
\centering\includegraphics[width=0.5\textwidth]{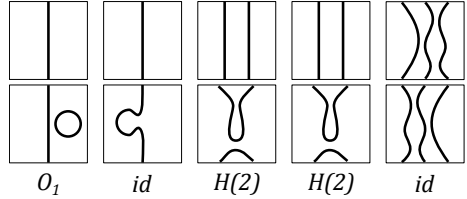}
\caption{Invariance conditions for $A$-smoothing}\label{fig:unary_A_smoothing}
\end{figure}
Thus, we can take $\mathcal M'=\{O_1, H(2)\}\sim\{H(2)\}$ for the destination knot theory both in classical and regular classical case.

Let $F$ be an oriented compact surface.
\begin{proposition}\label{prop:unary_A_smoothing}
Let $D\in\mathscr D_+(F)$ be an oriented classical link diagram. Then
$f_\tau(D)=[D]\in H_1(F,\Z_2)$ when $\mathcal M'=\{H(2)\}$.
\end{proposition}
\begin{proof}
The element $[D]=[f_\tau(D)]$ represents the diagram $f_\tau(D)$ in $\mathscr K^0(F|H(2))\simeq H_1(F,\Z_2)$.
\end{proof}

\subsubsection{Kauffman smoothing}

Let $T_+=a\skcrv+b\skcrh$, $T_-=a'\skcrh+b'\skcrv$. Following the reasonings for the usual Kauffman bracket~\cite{K0} one gets the relations $a'=b^{-1}$ and $b'=a^{-1}$. Then the local transformation rule $\tau$ determines a two-parametric invariant map from regular link diagrams to the skein module of diagrams without crossings modulo the move $O_\delta$, $\delta=-(\frac ab+\frac ba)$.

For regular classical knot (i.e. knots in the disk $\mathbb D^2$ or the sphere $S^2$), we get a bracket invariant in two variables
\[
\langle D\rangle=\sum_{s}a^{\alpha_+(s)-\beta_-(s)}b^{\beta_+(s)-\alpha_-(s)}\left(-\frac ab-\frac ba\right)^{\gamma(s)}
\]
where $s\in\{0,1\}^{\mathcal C(D)}$ is a state of the diagram $D$, $\alpha_\epsilon(s)=|s^{-1}(0)\cap\mathcal C_\epsilon(D)|$ is the number of crossings with the sign $\epsilon$ smoothed by type $0$ in $s$, $\beta_\epsilon(s)=|s^{-1}(1)\cap\mathcal C_\epsilon(D)|$, $\epsilon=\pm$, and $\gamma(s)$ is the number of components in the smoothed diagram $D_s$.

The bracket after a normalization becomes an invariant.
\begin{proposition}
The polynomial
\[
X(D)=\left(-\frac{a^2}b\right)^{-wr(D)}\langle D\rangle
\]
where $wr(D)$ is the writhe number of the diagram $D$, is an invariant of oriented classical links.
\end{proposition}

Note that after the variable change $\tilde a=\sqrt{\frac ab}$ the polynomial $X(D)$ becomes the usual Jones polynomial in the variable $\tilde a$.

\subsection{Binary functorial maps}\label{subsec:binary_functorial_maps}


\begin{definition}\label{def:binary_functorial_map}
A local transformation rule $\tau$ is called \emph{binary} if there are tangle sets $\mathcal T_0=\{T_{0-}, T_{0+}\}$ and $\mathcal T_1=\{T_{1-}, T_{1+}\}$ such that for any crossing $c$ $\tau(c)=T_{k,\sgn(c)}$ for some $k=0,1$. The set $\{\mathcal T_0,\mathcal T_1\}$ is the \emph{scheme} of the local transformation rule $\tau$.
\end{definition}

\begin{remark}\label{rem:binary_functorial_map}
With some abuse of notation we will write $\tau(c)=k$, $k=0,1$, instead of $\tau(c)=T_{k,\sgn(c)}$. Thus, for a fixed scheme $\{\mathcal T_0, \mathcal T_1\}$, the local transformation rule can be viewed as a trait with values in $\Z_2$.
\end{remark}

Let us consider several examples of binary functorial maps. We will assume that the local transformation rule $\tau$ is not constant (otherwise $\tau$ is a unary transformation rule).

\subsubsection{Scheme $(Sm^{or}, Sm^{unor})$}

Let $\mathcal T_0=Sm^{or}=\{\skcrv,\skcrv\},\ \mathcal T_1=Sm^{unor}=\{\skcrh,\skcrh\}$. The invariance conditions for the corresponding local transformation rule $\tau$ are shown in Fig.~\ref{fig:or_unor_smoothing}.

\begin{figure}[p]
\centering\includegraphics[width=0.35\textwidth]{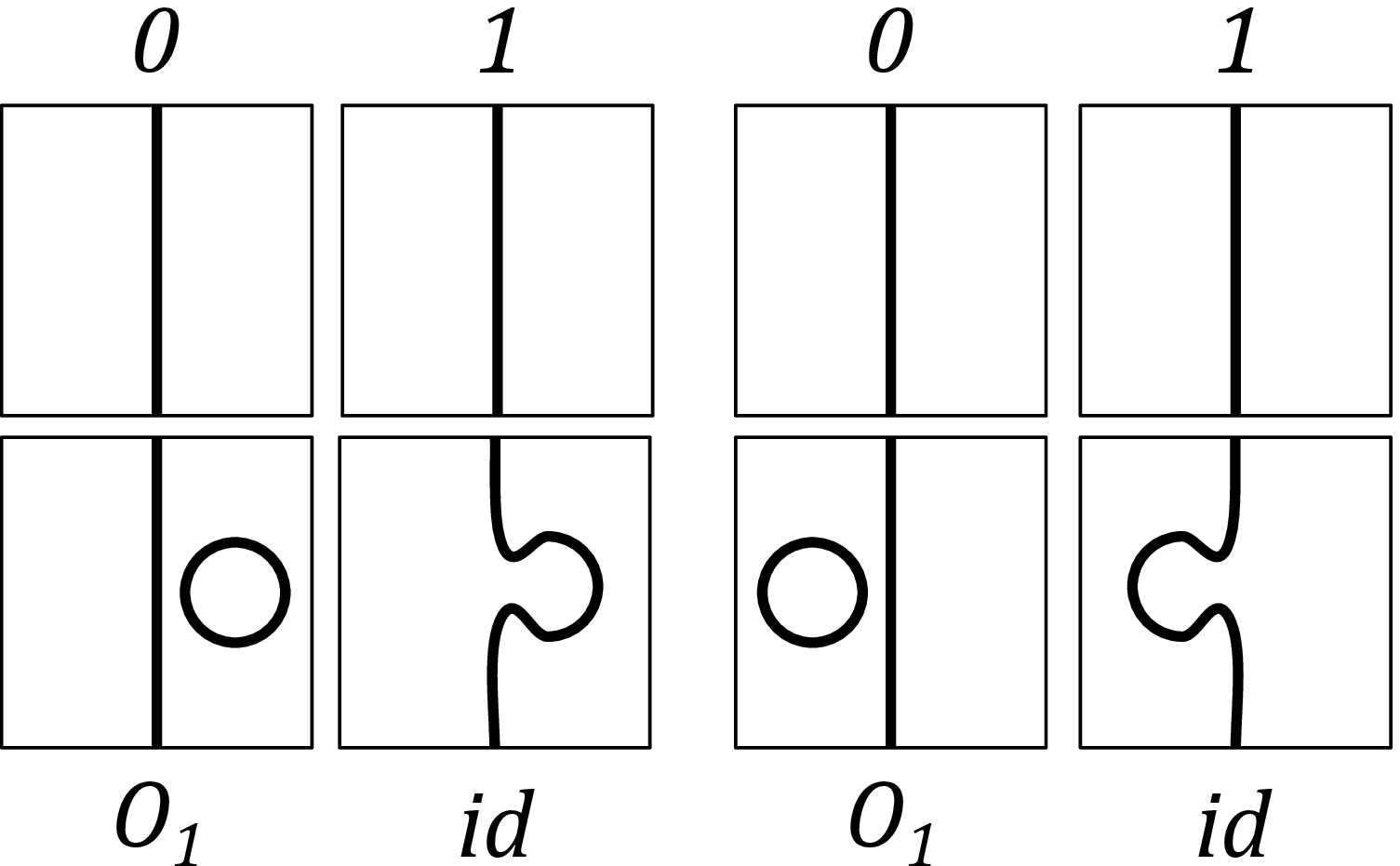}\\ \medskip
 \includegraphics[width=0.7\textwidth]{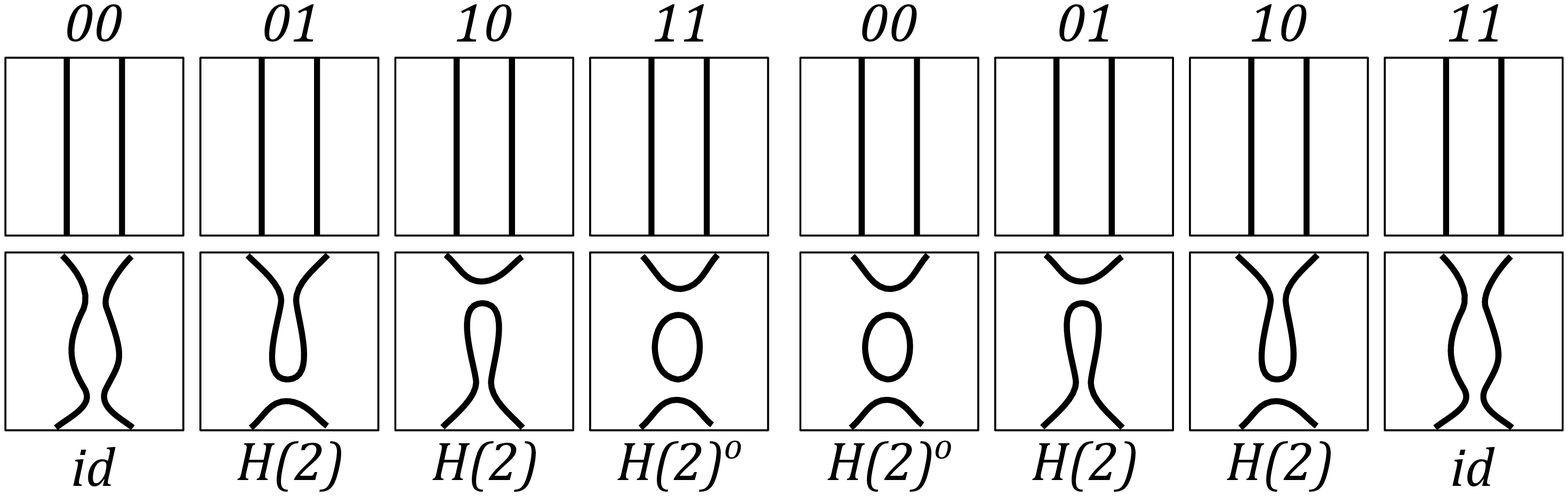}\\ \medskip
\includegraphics[width=0.7\textwidth]{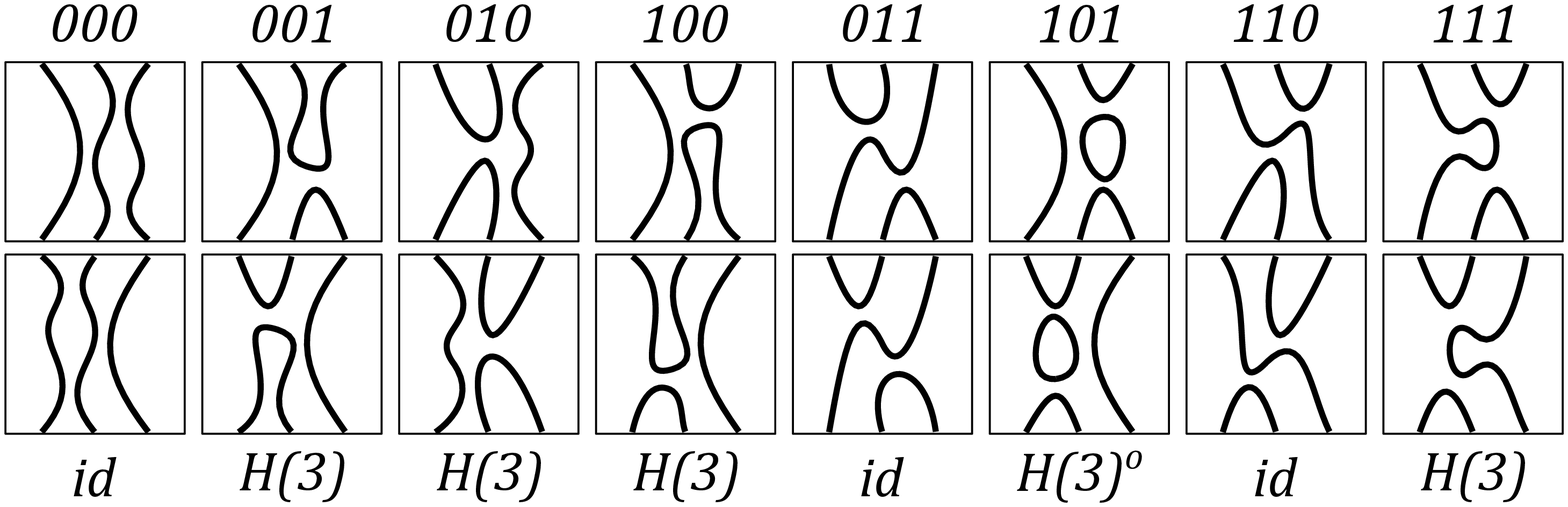}
\caption{Invariance conditions for the scheme $(Sm^{or}, Sm^{unor})$}\label{fig:or_unor_smoothing}
\end{figure}

Let us analyze relations between properties of the local transformation rule and the possible destination knot theory.

If $\tau$ is not an index then a second Reidemeister move can occur with crossings that have different trait values. Then we get the move $H(2)$. The move $H(2)$ produces the moves $H(2)^o$, $H(3)$ and $H(3)^o$. Hence, the map $f_\tau$ is invariant in the destination knot theory $\mathcal M'=\{H(2)\}$.

Let $\tau$ be an index. Since $\tau$ is not constant, there is a crossing $c$ such that $\tau(c)=1$. Then we can modify the diagram so that the crossing $c$ participates in a Reidemeister move $\Omega_{2a}$. Thus, we obtain the move $H(2)^o$ to be included in the destination knot theory $\mathcal M'$.

If there is no loop values $\tau^{l\pm}$, $\tau^{r\pm}$ equal to $0$ then the move $H(2)^o$ insures invariance of the functorial maps under first and second Reidemeister moves. Since the move $H(2)^o$ generates the move $H(3)$ and the move $H(3)^o$ (the latter is a composition of two $H(2)^o$ moves), the functorial map will be invariant when $H(2)^o\in\mathcal M'$.

If there is an loop $c$ with $\tau(c)=0$ then we get the move $O_1$. The moves $O_1$ and $H(2)^o$ generate the move $H(2)$. Hence, $H(2)\in\mathcal M'$. Since the other transformations in the invariance conditions of the scheme are generated by $H(2)$, the move $H(2)$ ensures invariance of the functorial map $f_\tau$.

For the functorial map on regular oriented classical knots, we have $\mathcal M'=\{H(2)^o\}$ when $\tau$ is an index, and $\mathcal M'=\{H(2)\}$ otherwise.

We can summarize the reasonings above in the following table.

\begin{center}
\begin{tabular}{|c|c|c|}
\hline
$\mathcal M$ & $\tau$ & $\mathcal M'$\\
\hline
 & trait & $H(2)$\\
$\mathcal M_{class}^+$ & index, $\forall\tau^{r/l\pm}=1$ & $H(2)^o$\\
 & index, $\exists\tau^{r/l\pm}=0$ & $H(2)$\\
\hline
$\mathcal M_{class}^{reg+}$ & trait & $H(2)$\\
 & index & $H(2)^o$\\
\hline
\end{tabular}
\end{center}

Let $F$ be an oriented compact surface.

\begin{proposition}\label{prop:or_unor_smoothing}
Let $\tau$ be a local transformation rule with the scheme $(Sm^{or}, Sm^{unor})$ on the diagram category $\mathfrak K_+(F)=\mathfrak K(\mathscr D_+(F)|\mathcal M_{class}^+)$. If $\tau$ is an index with the loop values $\tau^{r/l\pm}$ all equal to $1$ then the functorial map $f_\tau$ in invariant in the knot theory $\mathcal M'=\{H(2)^o\}$ and
\[
f_\tau(D)=([D],\rho(D)+P_D'(1)+n^\tau(D))\in H_1(F,\Z_2)\times\Z_2
\]
where $\rho(D)$ is the offset, $P_D(t)$ the based index polynomial and $n^\tau(D)=|\tau^{-1}_D(1)|$ the number of odd crossings in $D$. Otherwise, the functorial map is invariant in the knot theory $\mathcal M'=\{H(2)\}$ and
$
f_\tau(D)=[D]\in H_1(F,\Z_2).
$
\end{proposition}

\begin{proof}
Consider the case $\tau$ is an index with the loop values $\tau^{r/l\pm}$ all equal to $1$. Then the class of the diagram $f_\tau(D)$ in $\mathscr K^0(F|H(2)^o)$ is determined by the pair $([f_\tau(D)],\rho(f_\tau(D)))\in H_1(F,\Z_2)\times\Z_2$. But $[f_\tau(D)]=[D]$. By Corollary~\ref{cor:unor_invariants_after_smoothing} $\rho(f_\tau(D))=\rho(Sm^{or}(D))+n^\tau(D)=\rho(D)+P_D'(1)+n^\tau(D)$ where $Sm^{or}(D)$ is the diagram obtained by oriented smoothings of all crossings.
\end{proof}

\subsubsection{Scheme $(Sm^{or}, id)$}

Let $\mathcal T_0=Sm^{or}=\{\skcrv,\skcrv\},\ \mathcal T_1=id=\{\skcrr,\skcrl\}$. The invariance conditions for the corresponding local transformation rule $\tau$ are shown in Fig.~\ref{fig:id_or_smoothing}.

\begin{figure}[p]
\centering\includegraphics[width=0.35\textwidth]{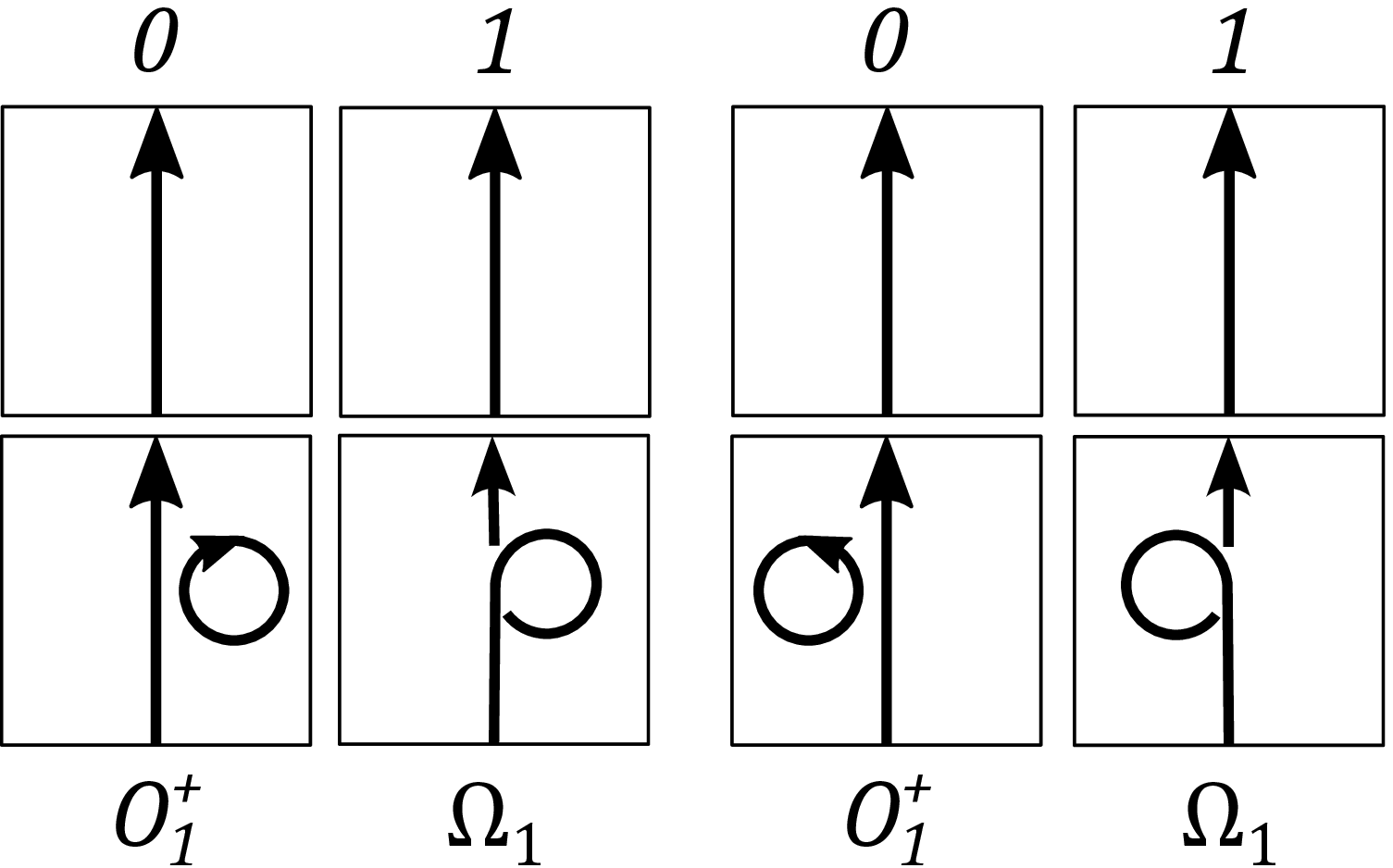}\\ \medskip
 \includegraphics[width=0.7\textwidth]{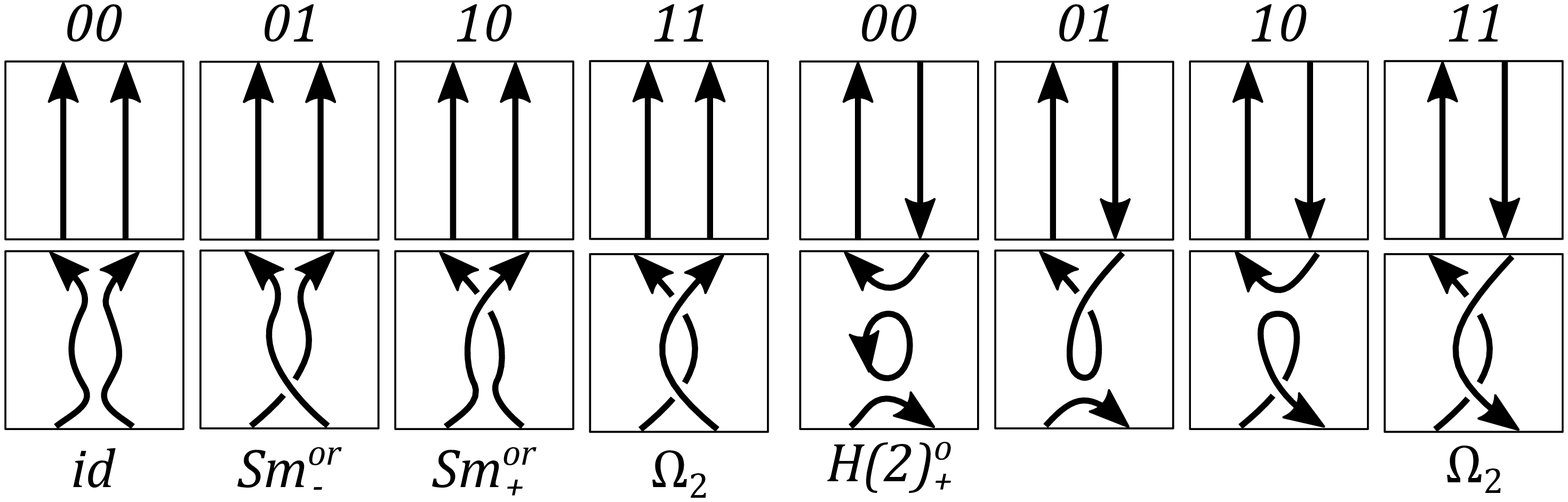}\\ \medskip
\includegraphics[width=0.7\textwidth]{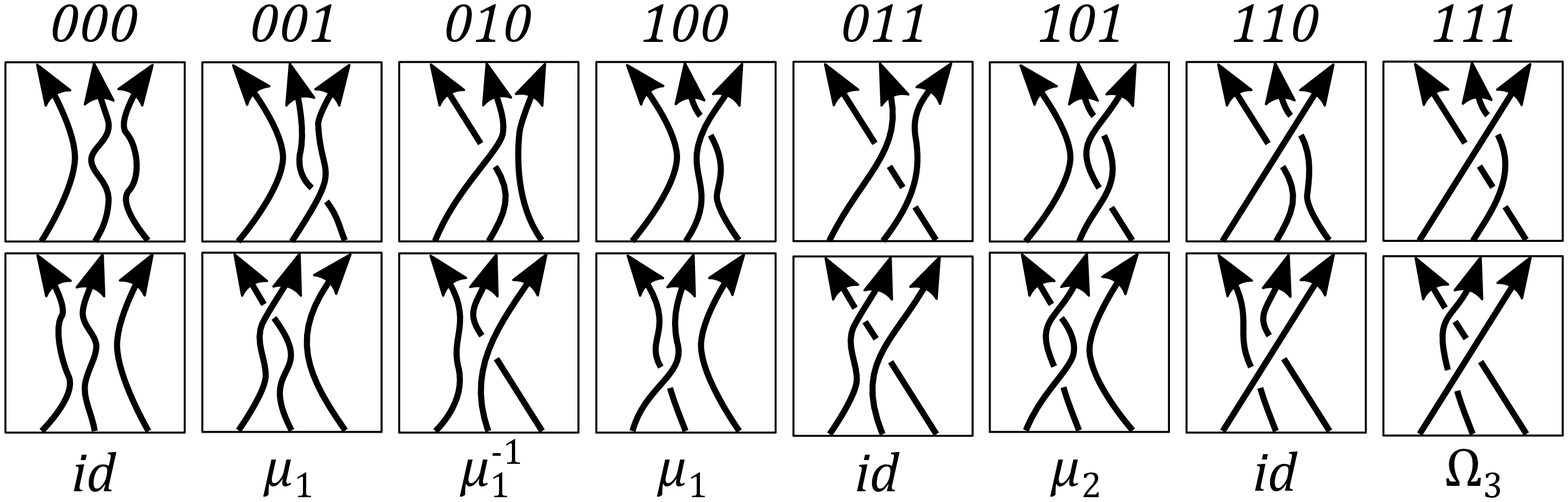}
\caption{Invariance conditions for the scheme $(Sm^{or}, id)$}\label{fig:id_or_smoothing}
\end{figure}

If $\tau$ is not an index then there is a second Reidemeister move which induces the oriented smoothing $Sm^{or}_\pm$ of a positive or a negative crossing. Assume it is the oriented smoothing $Sm^{or}_+$. Then we can apply this move to all positive crossings in $f_\tau(D)$, hence, we can think that $\tau(c)=0$ for all the positive crossings. Let $c\in \mathcal C(D)$ be a negative crossing such that $\tau(c)=1$. Modify the diagram $D$ so that $c$ can take part in a move $\Omega_{2a}$ or $\Omega_{2b}$ with another positive crossing $c'$. Then the invariance condition for this move gives the move $Sm^{or}_-$. Hence, all the negative crossings can be smoothed. Thus, this case is reduced to the unary functorial map $f_{Sm^{or}}$.

If $\tau$ is an index then the invariance for second Reidemeister moves yields the moves $\Omega_2$ and $H(2)_+^o$.

Consider the invariance condition for the third Reidemeister move. If $\tau$ is an index but not a parity then we get either the move $\mu_1$ (the cases $001$ and $100$) or the move $\mu_1^{-1}$ (the case $010$) or the move $\Omega_{3b}$ (the case $111$). The move $\mu_1$ (applied thrice) produces the move $\Omega_{3b}$, on the other hand, the moves $H(2)_+^o$, $\Omega_2$ and $\Omega_{3b}$ generate the move $\mu_1$. Thus, invariance for the third Reidemeister move is ensured by the move $\Omega_{3b}$ (together with the moves $H(2)_+^o$ and $\Omega_2$).

Assume the rule $\tau$ is a parity. Since $\tau$ is not constant, for diagrams in a connected surface, we can create a diagram to which a move $\Omega_{3b}$ of type $101$ can be applied. Then invariance for the third Reidemeister move yields the move $\mu_2$ which is equivalent (modulo the moves $H(2)_+^o$ and $\Omega_2$) to the move $2\Omega_\infty$ (Fig.~\ref{fig:2R_infinity_move}).

\begin{figure}[h]
\centering\includegraphics[width=0.25\textwidth]{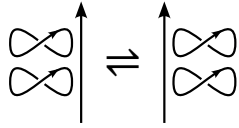}
\caption{The move $2\Omega_\infty$}\label{fig:2R_infinity_move}
\end{figure}

Consider the invariance condition for the first Reidemeister moves. We assume that $\tau$ is an index. If all the loop values are odd: $\tau^{r\pm}=\tau^{l\pm}=1$ then the invariance for first Reidemeister moves yields the move $\Omega_1$.

If $\tau^{r+}(=\tau^{l-})=0$ and $\tau^{r-}(=\tau^{l+})=1$ then we get the moves $O^+_1$ (elimination of a trivial circle oriented counterclockwise) and $\Omega_{1a}$. The moves $O^+_1$ and $H(2)_+^o$ generate the move $H(2)_+$, and the moves $\Omega_{1a}$, $H(2)_+^o$ and $\Omega_2$ generate other variants of first Reidemeister move.

If all the loop values are even: $\tau^{r\pm}=\tau^{l\pm}=0$ (for example, when $\tau$ is a parity) then we get the move $O_1$ and hence the move $H(2)_+$.

We can summarize the reasoning above in the following table.

\begin{center}
\begin{tabular}{|c|c|c|}
\hline
$\mathcal m$ & $\tau$ & $\mathcal M'$\\
\hline
 & trait & $Sm^{or}_\pm, H(2)_+$\\
 & index, $\tau^{r+}=\tau^{r-}=0$ & $H(2)_+,\Omega_2,\Omega_3$\\
$\mathcal M_{class}^+$ & index, $\tau^{r+}=\tau^{r-}=1$ & $H(2)^o_+,\Omega_1,\Omega_2,\Omega_3$\\
 & index, $\tau^{r+}\ne\tau^{r-}$ & $H(2)_+,\Omega_1,\Omega_2,\Omega_3$\\
 & parity & $H(2)_+,\Omega_2,2\Omega_\infty$\\
\hline
  & trait & $Sm^{or}_\pm, H(2)_+$\\
$\mathcal M_{class}^{reg+}$  & index & $H(2)_+^o,\Omega_2,\Omega_3$\\
& parity  & $H(2)^o_+,\Omega_2,2\Omega_\infty$\\
\hline
\end{tabular}
\end{center}

Let $F$ be an oriented compact surface.

\begin{proposition}\label{prop:binary_or_id}
Let $\tau$ a local transformation rule with the scheme $(Sm^{or}, id)$ on the diagram set $\mathscr D_+(F)$.
\begin{itemize}
\item If $\tau$ is an index and not a parity, with the loop values $\tau^{r+}=\tau^{r-}=0$ then $f_\tau(D)=([D],wr^{\tau}(D))\in H_1(F,\Z)\times\Z$ where $wr^\tau(D)=\sum_{v\colon \tau(v)=1}sgn(v)$ is the $\tau$-odd writhe of the diagram;
\item If $\tau$ is an index with the loop values $\tau^{r+}=\tau^{r-}=1$ then $f_\tau(D)=([D],rot(D)+wr^\tau(D))\in H_1(F,\Z)\times\Z_2$ where $rot(D)$ is the rotation number of the diagram;
\item If $\tau$ is an index with the loop values $\tau^{r+}\ne\tau^{r-}$ then $f_\tau(D)=[D]\in H_1(F,\Z)$;
\item If $\tau$ is a parity then  $f_\tau(D)=([D], (P_D(t)-P^{\tau}_D(t))\bmod 2, \lfloor \frac{wr^\tau(D)}2\rfloor)$  where $P_D(t)$ is the based index polynomial of $D$ and
\[
P^{\tau}_D(t)=\sum_{v\colon \tau(v)=0}sgn(v)t^{ind_z(v)}
\]
is the $\tau$-even based index polynomial considered as an element of the quotient of the set $\Z_2[t,t^{-1}]/(t^{\mu(D)}-1)$ modulo the action $g(t)\mapsto tg(t)$, and $\lfloor \frac{wr^\tau(D)}2\rfloor\in\Z$.
\end{itemize}
\end{proposition}
\begin{proof}
Let $\tau$ be an index and not a parity, with the loop values $\tau^{r+}=\tau^{r-}=0$. By Corollary~\ref{cor:smoothing_skein modules_or}, $f_\tau(D)$ is determined by the homology class $[f_\tau(D)]$ and the writhe $wr(f_\tau(D))$. But $[f_\tau(D)]=[D]$ and $wr(f_\tau(D))=wr^\tau(D)$.

The other cases are proved analogously.
\end{proof}

\subsubsection{Scheme $(Sm^{unor}, id)$}

Let $\mathcal T_0=Sm^{unor}=\{\skcrh,\skcrh\},\ \mathcal T_1=id=\{\skcrr,\skcrl\}$. The invariance conditions for the corresponding local transformation rule $\tau$ are shown in Fig.~\ref{fig:id_unor_smoothing}.

\begin{figure}[p]
\centering\includegraphics[width=0.35\textwidth]{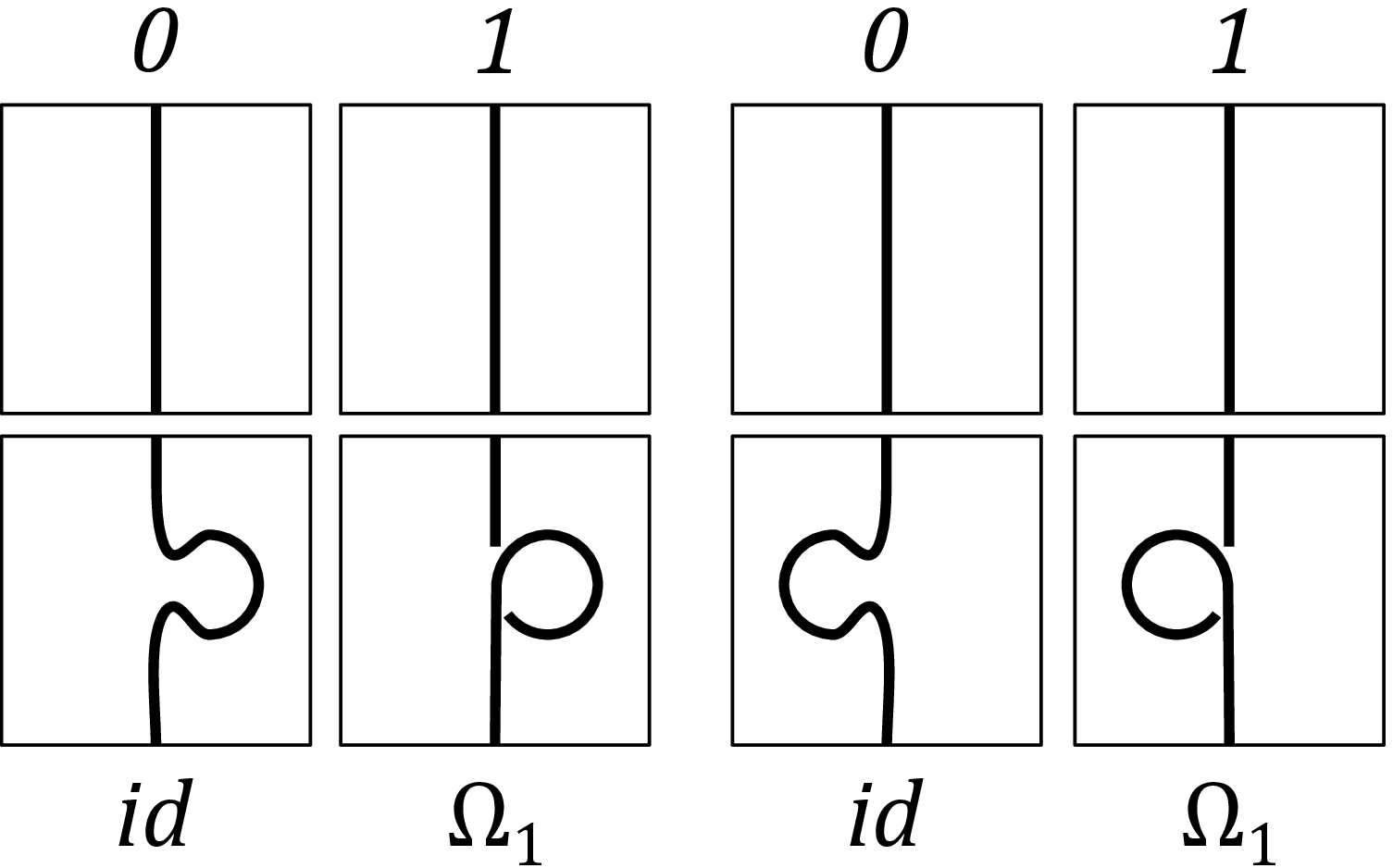}\\ \medskip
 \includegraphics[width=0.7\textwidth]{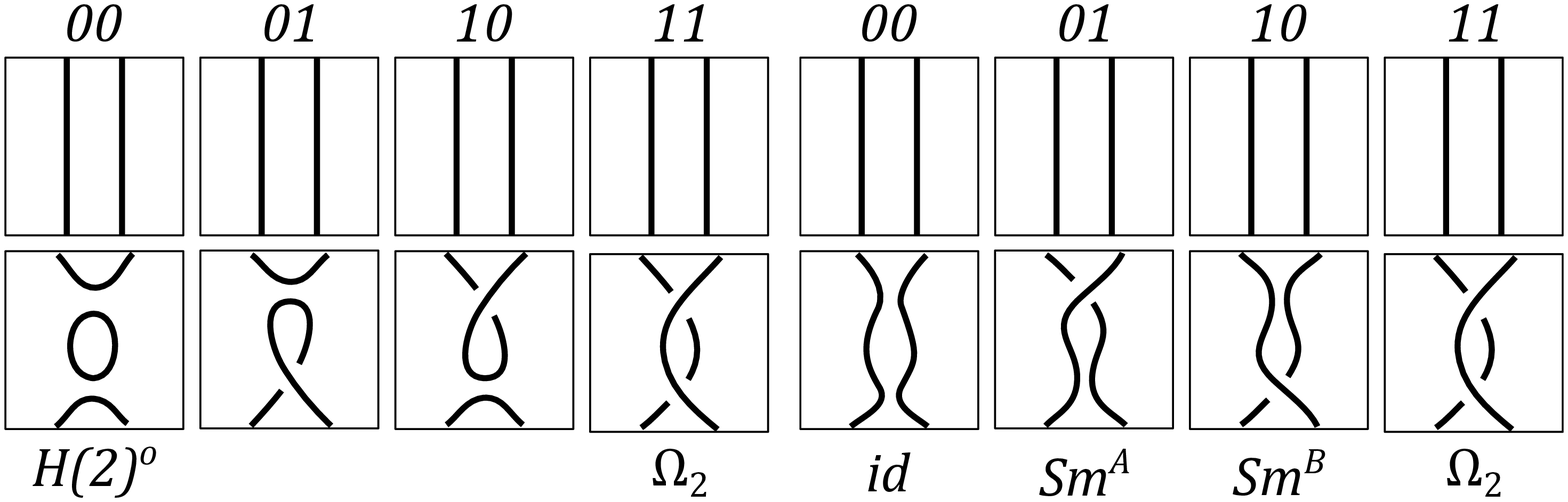}\\ \medskip
\includegraphics[width=0.7\textwidth]{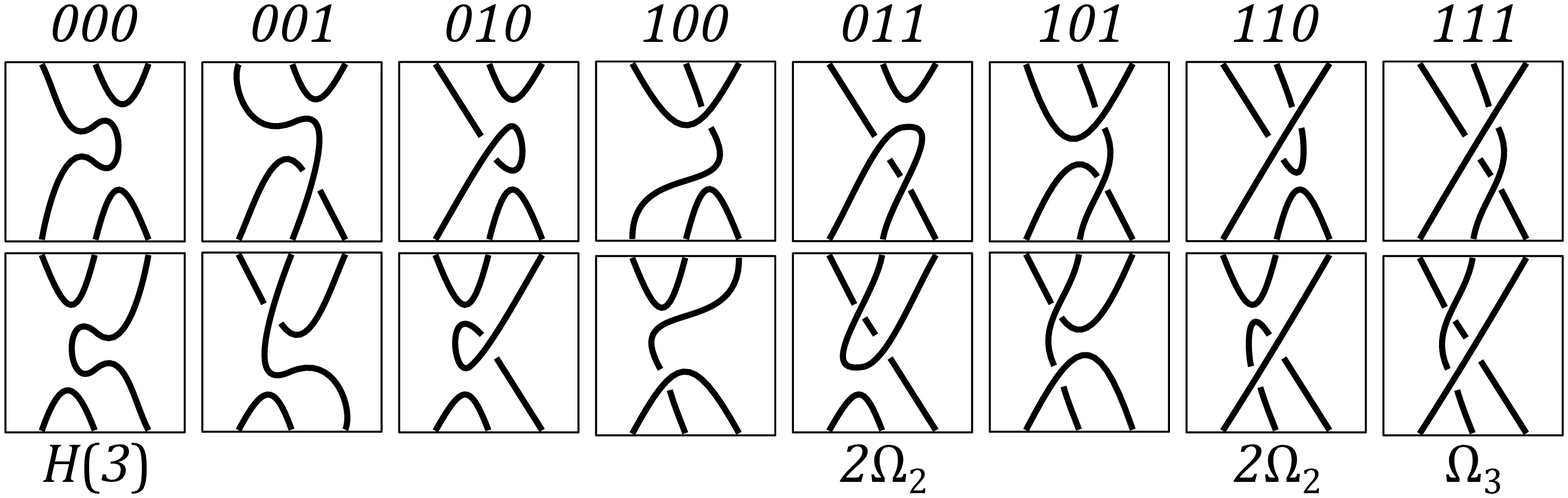}
\caption{Invariance conditions for the scheme $(Sm^{unor}, id)$}\label{fig:id_unor_smoothing}
\end{figure}

The possible invariance cases for the functorial map are analogous to those in the scheme $(Sm^{unor}, id)$ with the difference that here the diagrams $f_\tau(D)$ are not oriented.

Let $\tau$ be not an index, for example, we have a second Reidemeister move of type $01$. Then we get the move $Sm^A$.
Then we can smooth all the crossings in $f_\tau(D)$ and reduce the $f_\tau$ to a binary functorial map with the scheme  $(Sm^{or}, Sm^{unor})$.

If $\tau$ is an index then the invariance for second Reidemeister moves yields the moves $H(2)^o$ and $\Omega_2$.

The invariance under third Reidemeister moves gives the moves which are equivalent (modulo $H(2)^o$ and $\Omega_2$) either to the move  $\Omega_3$ (cases $100$, $010$, $001$, $111$) or to the move $2\Omega_\infty$ (the case $101$). The case $000$ of third Reidemeister move is generated by $H(2)^o$. Thus, if $\tau$ is not a parity then the invariance for $\Omega_3$ is ensured by the move $\Omega_3$. When $\tau$ is a parity, one needs the move $2\Omega_\infty$ for the invariance.

The invariance under first Reidemeister moves yields the move $\Omega_1$ if there is an odd loop value of $\tau$, and requires no additional moves when all the loop values are even.

Thus, we have the following table for the scheme $(Sm^{unor}, id)$.

\begin{center}
\begin{tabular}{|c|c|c|}
\hline
$\mathcal M$ & $\tau$ & $\mathcal M'$\\
\hline
 & trait & $Sm^{A}$ or $Sm^B$, $H(2)^o$\\
$\mathcal M_{class}^{+}$ & index, $\exists\tau^{r\pm}=1$ & $H(2)^o,\Omega_1,\Omega_2,\Omega_3$\\
 & index, $\tau^{r+}=\tau^{r-}=0$ & $H(2)^o,\Omega_2,\Omega_3$\\
 & parity & $H(2)^o,\Omega_2,2\Omega_\infty$\\
\hline
 & trait & $Sm^{A}$ or $Sm^B$, $H(2)^o$\\
$\mathcal M_{class}^{reg+}$ & index & $H(2)^o,\Omega_2,\Omega_3$\\
 & parity & $H(2)^o,\Omega_2,2\Omega_\infty$\\
\hline
\end{tabular}
\end{center}

\begin{proposition}\label{prop:binary_unor_id}
Let $F$ be a surface and $\tau$ a local transformation rule with the scheme $(Sm^{unor}, id)$ on the diagram set $\mathscr D(F)$.
\begin{itemize}
\item If $\tau$ is an index and not a parity, with the loop values $\tau^{r+}=\tau^{r-}=0$ then
\[
f_\tau(D)=([D],2\rho(D)+2(P_D^\tau)'(1)+2n^\tau(D)+wr_{odd}(D)-P^\tau_D(-1))\in H_1(F,\Z_2)\times\Z_4;
\]
\item If $\tau$ is an index and $\exists\tau^{r\pm}=1$ then $f_\tau(D)=[D]\in H_1(F,\Z_2)$;
\item If $\tau$ is a parity then
\[
f_\tau(D)=([D],\rho(D)-(P_D^\tau)'(1)+n^\tau(D),wr_{odd}(D)-P^\tau_D(-1))
\]
when $[D]=0$, and
\[
f_\tau(D)=([D],2\rho(D)+2(P_D^\tau)'(1)+2n^\tau(D)+wr_{odd}(D)-P^\tau_D(-1))
\]
when $[D]\ne 0$.
\end{itemize}
Here $P_D^\tau(t)=\sum_{v\colon \tau(v)=0}sgn(v)t^{ind_z(v)}$ is the $\tau$-even based index polynomial, and $n^\tau(D)$ is the number of crossings $v$ in $D$ such that $\tau(v)=0$.
\end{proposition}
\begin{proof}
By Corollary~\ref{cor:smoothing_skein modules_unor}, the image of $f_\tau(D)$ in the skein module is determined by $\rho(f_\tau(D))$ and $wr_{odd}(f_\tau(D))$. By Corollary~\ref{cor:unor_invariants_after_smoothing}, $\rho(f_\tau(D))=\rho(D)-(P_D^\tau)'(1)+n^\tau(D)$ and $wr_{odd}(f_\tau(D))=wr_{odd}(D)-P^\tau_D(-1)$.
Then Corollary~\ref{cor:smoothing_skein modules_unor} implies the statements of the proposition.
\end{proof}

\subsubsection{Scheme $(Sm^A, id)$}
Let $\mathcal T_0=Sm^{A}=\{\skcrv,\skcrh\},\ \mathcal T_1=id=\{\skcrr,\skcrl\}$. The invariance conditions for the corresponding local transformation rule $\tau$ are shown in Fig.~\ref{fig:id_A_smoothing}.

\begin{figure}[p]
\centering\includegraphics[width=0.35\textwidth]{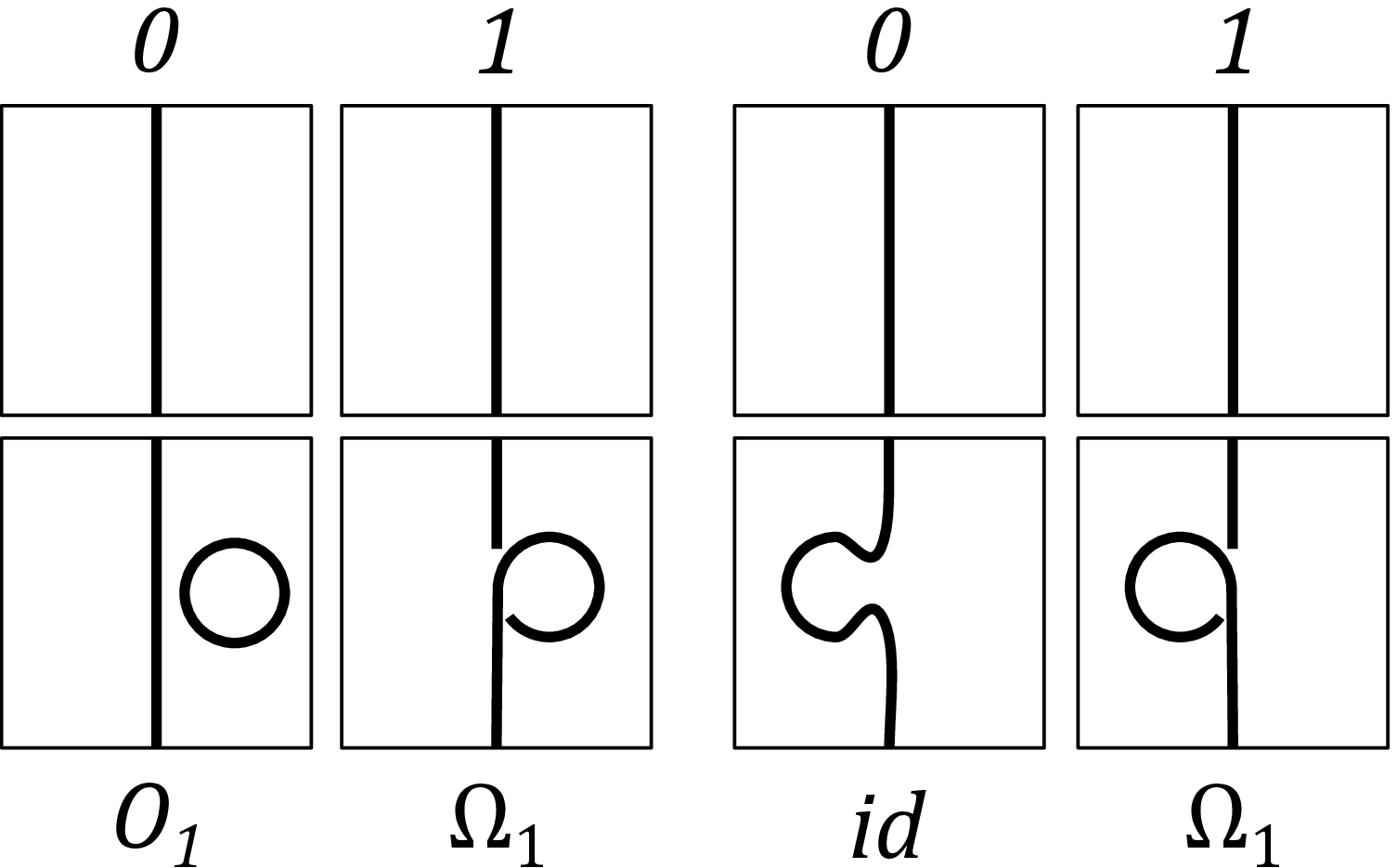}\\ \medskip
 \includegraphics[width=0.7\textwidth]{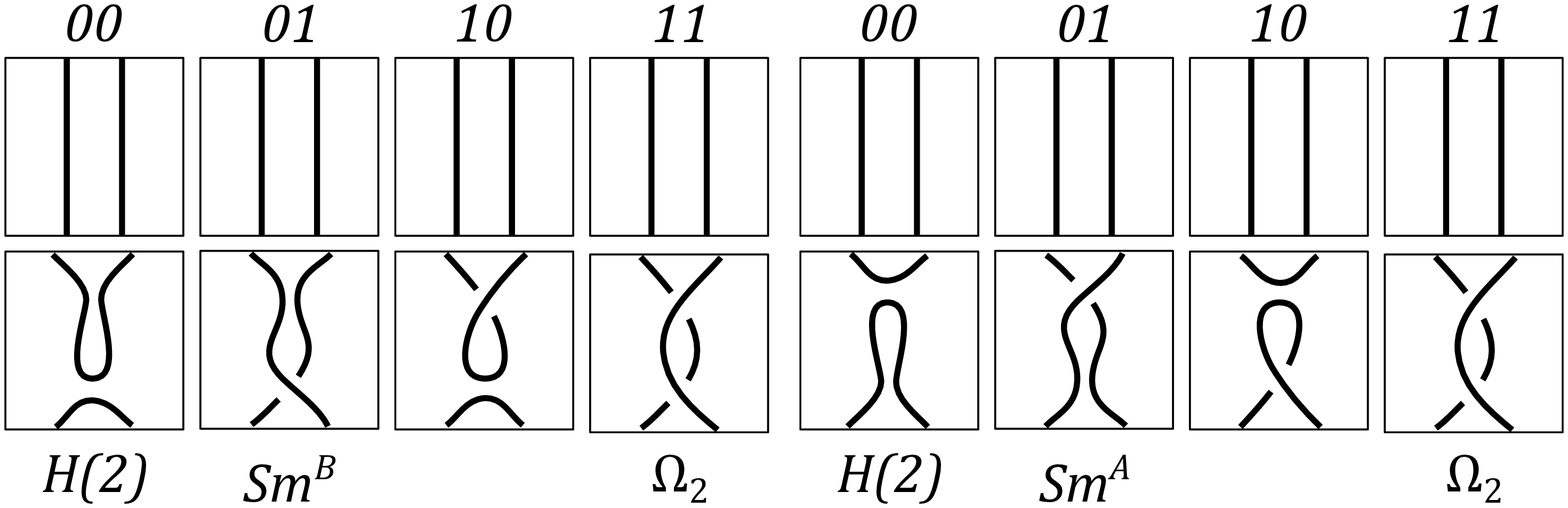}\\ \medskip
\includegraphics[width=0.7\textwidth]{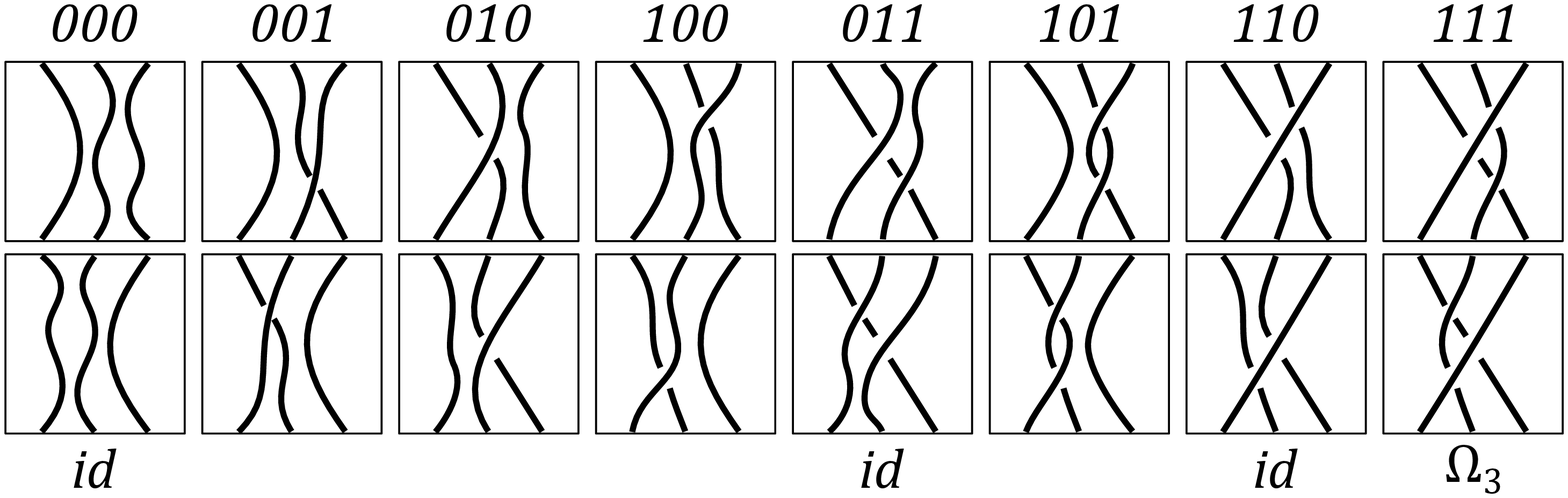}
\caption{Invariance conditions for the scheme $(Sm^A, id)$}\label{fig:id_A_smoothing}
\end{figure}

By analogy with the schemes $(Sm^{or}, id)$ and $(Sm^{unor}, id)$ we have the following table of invariant cases for the scheme $(Sm^A, id)$.

\begin{center}
\begin{tabular}{|c|c|c|}
\hline
$\mathcal M$ & $\tau$ & $\mathcal M'$\\
\hline
 & trait & $Sm^{A}$ or $Sm^B$, $H(2)^o$\\
$\mathcal M_{class}^{+}$ & index, $\exists\tau^{r\pm}=1$ & $H(2),\Omega_1,\Omega_2,\Omega_3$\\
 & index, $\tau^{r+}=\tau^{r-}=0$ & $H(2),\Omega_2,\Omega_3$\\
 & parity & $H(2),\Omega_2,2\Omega_\infty$\\
\hline
 & trait & $Sm^{A}$ or $Sm^B$, $H(2)^o$\\
$\mathcal M_{class}^{reg+}$ & index & $H(2),\Omega_2,\Omega_3$\\
 & parity & $H(2),\Omega_2,2\Omega_\infty$\\
\hline
\end{tabular}
\end{center}

The following statement is analogous to Propositions~\ref{prop:binary_or_id} and~\ref{prop:binary_unor_id}.
\begin{proposition}\label{prop:binary_A_id}
Let $F$ be a surface and $\tau$ a local transformation rule with the scheme $(Sm^{unor}, id)$ on the diagram set $\mathscr D(F)$.
\begin{itemize}
\item If $\tau$ is an index and not a parity, with the loop values $\tau^{r+}=\tau^{r-}=0$ then $f_\tau(D)=([D],n(D)-n^\tau(D) \bmod 2)\in H_1(F,\Z_2)\times\Z_2$;
\item If $\tau$ is an index and $\exists\tau^{r\pm}=1$ then $f_\tau(D)=[D]\in H_1(F,\Z_2)$;
\item If $\tau$ is a parity then  $f_\tau(D)=([D],wr^\tau_{odd}(D))$ where 
\begin{multline*}
wr^\tau_{odd}(D)=|\{v\in\mathcal C(D)\mid \tau(v)=1, ind^{un}_z(v)=0\}|\\
-|\{v\in\mathcal C(D)\mid \tau(v)=1, ind^{un}_z(v)=1\}|
\end{multline*}
is considered as of element in $\Z_4/(x=-x)$ if $[D]=0$,  and $f_\tau(D)=([D],n(D)-n^\tau(D) \bmod 2)$ if $[D]\ne 0$.
\end{itemize}
\end{proposition}

\subsubsection{Scheme $(Sm^{Kauffman}, id)$}

Let $\mathcal T_0=Sm^{Kauffman}=\{a\skcrv+a^{-1}\skcrh,a^{-1}\skcrv+a\skcrh\},\ \mathcal T_1=id=\{\skcrr,\skcrl\}$. Invariance conditions for the corresponding local transformation rule $\tau$ are shown in Fig.~\ref{fig:id_kauffman_smoothing}. We omit bulky invariance conditions for the third Reidemeister moves of types $000, 001, 010, 100$.

Assume that the functorial map is invariant in the case when all crossing are even. Then the destination knot theory should include the move $O_\delta$ where $\delta=-a^2-a^{-2}$.

We restrict ourself to the case when $\tau$ is a weak parity. Then invariance under the second and third Reidemeister moves is ensured by the moves $O_\delta$, $\Omega_2$, $CC$ and $\Omega_3$. If $\tau$ is a parity then it is enough to have the moves $O_\delta$, $\Omega_2$ and $CC$.

Since the loop values of a non constant weak parity are even, then first Reidemeister moves lead to multiplications by $(-a^3)^{\pm 1}$ like in the classical case. One can compensate these moves by normalization.

\begin{figure}[p]
\centering\includegraphics[width=0.35\textwidth]{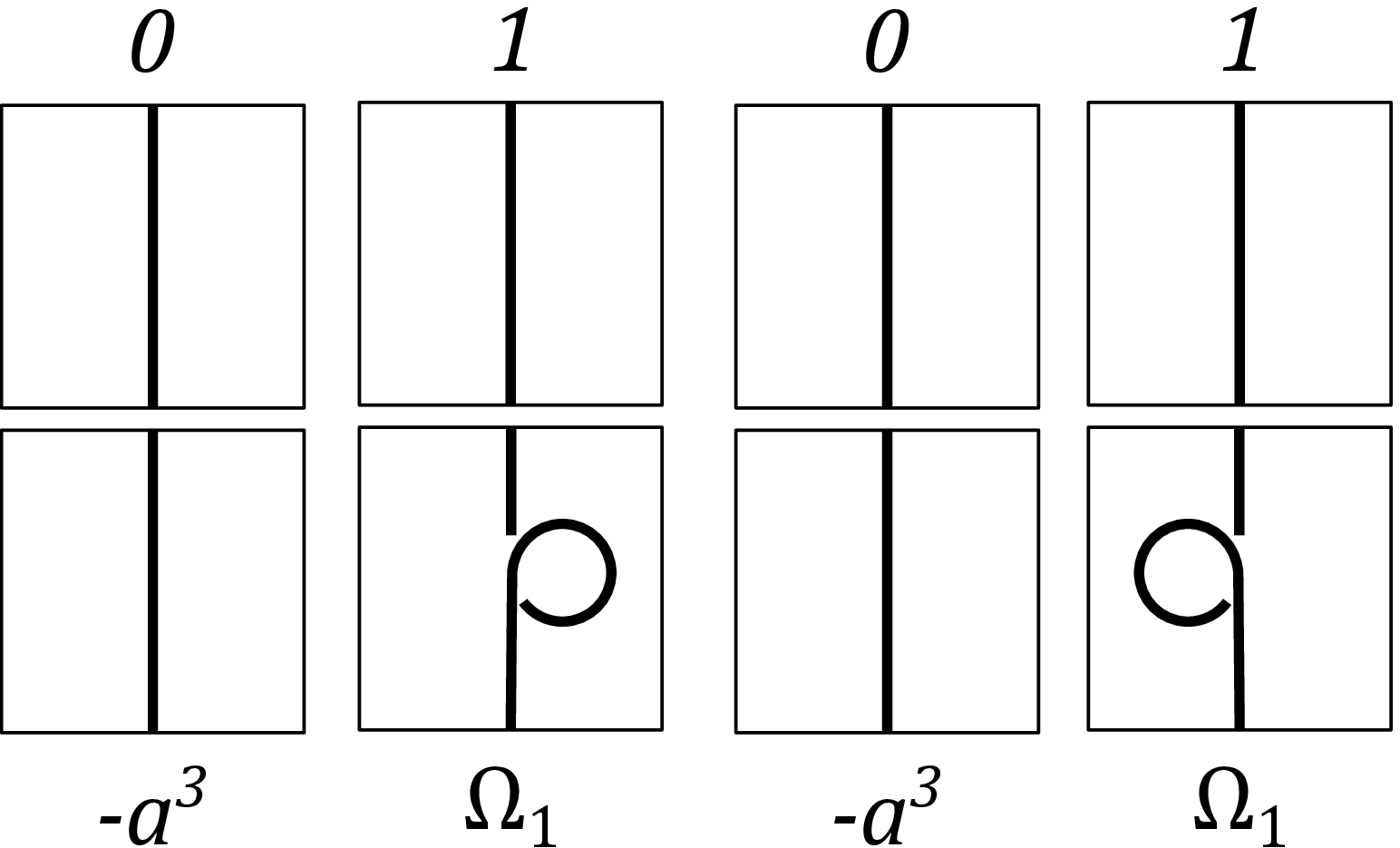}\\ \medskip
 \includegraphics[width=0.7\textwidth]{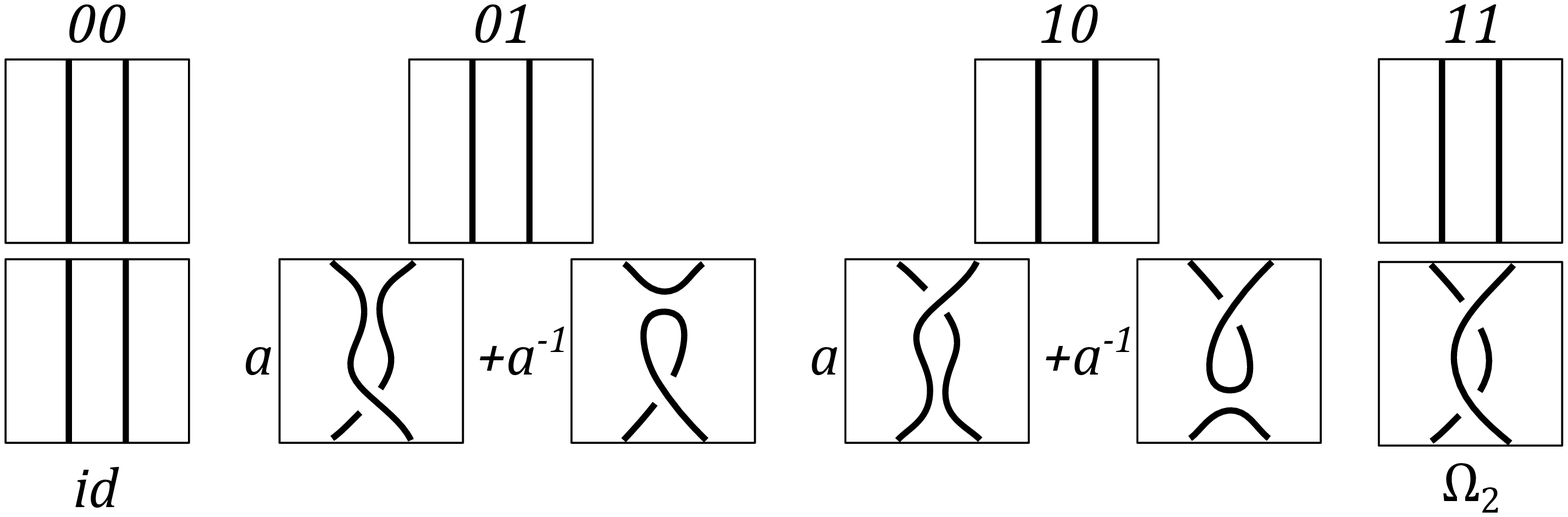}\\ \medskip
\includegraphics[width=0.8\textwidth]{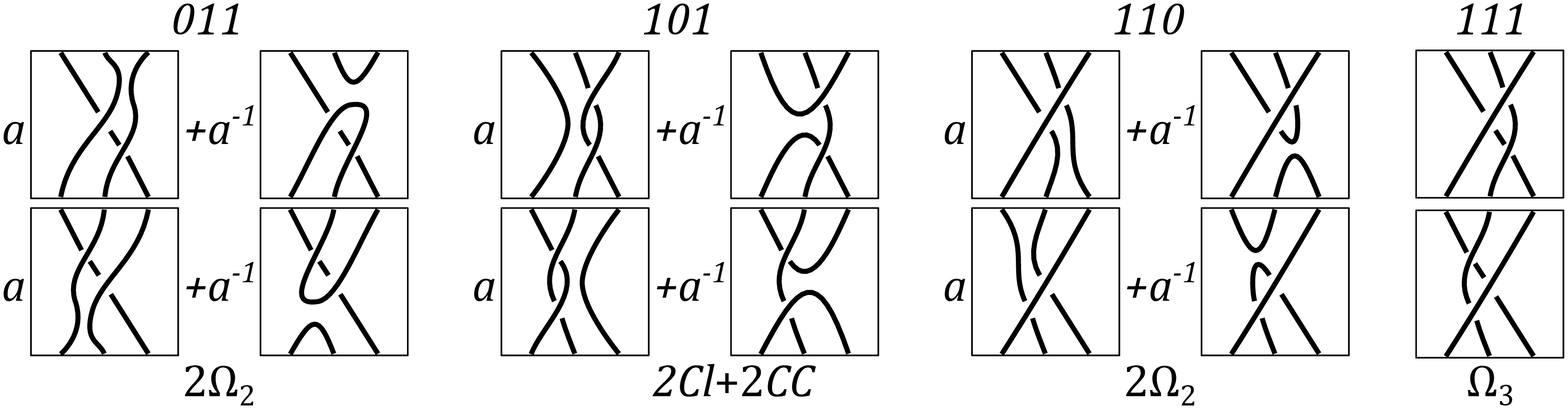}
\caption{Invariance conditions for the scheme $(Sm^{Kauffman}, id)$}\label{fig:id_kauffman_smoothing}
\end{figure}

\begin{center}
\begin{tabular}{|c|c|c|}
\hline
$\mathcal M$ & $\tau$ & $\mathcal M'$\\
\hline
$\mathcal M_{class}^{reg+}$ & weak parity & $O_\delta, CC, \Omega_2, \Omega_3$\\
& parity & $O_\delta, CC, \Omega_2$\\
\hline
\end{tabular}
\end{center}

Thus we have the following statement.

\begin{proposition}\label{prop:binary_kauffman_smoothing}
Let $\tau$ a trait with values in $\{0,1\}$ on classical knot theory in a surface $F$.
\begin{itemize}
\item If $\tau$ is a parity then the functorial map $f_\tau$ in the scheme $(Sm^{Kauffman}, id)$ is an invariant of regular knots with values in the regular doodles without trivial components (i.e. the knot theory $\{O_\delta, CC, \Omega_2\}$), and the normalized functorial map $\bar f_\tau(D)=(-a)^{-3wr(D)}f_\tau(D)$ is a knot invariant. Here $wr(D)$ is the writhe number of the diagram $D$;
\item If $\tau$ is a weak parity then the functorial map $f_\tau$ in the scheme $(Sm^{Kauffman}, id)$ is an invariant of regular knots with values in the regular flat knots without trivial components (i.e. the knot theory $\{O_\delta, CC, \Omega_2, \Omega_3\}$), and the normalized functorial map $\bar f_\tau(D)=(-a)^{-3wr(D)}f_\tau(D)$ is a knot invariant.
\end{itemize}
\end{proposition}

\begin{corollary}\label{cor:flat_binary_kauffman_smoothing}
Let $a^2=1$ and $\tau$ a weak parity on flat knots. Then the scheme $(Sm^{Kauffman}, id)$ defines a map to regular flat knot.
\end{corollary}

\begin{remark}
1. A diagram valued bracket for a parity was defined in~\cite{M3}.

2. For weak parities, a Kauffman bracket with values in flat knots was considered in~\cite{Cheng2}.

3. For traits $\tau$ which are indices but not weak parities, a Kauffman bracket polynomial was considered in~\cite{IPS}. Invariance conditions for third Reidemeister moves generate the relation $a^{-6}-a^{-2}-1+a^4=0$ in the ring $\Z[a,a^{-1}]$, so the bracket takes values in the quotient ring $\Z[a,a^{-1}]/(a^{-6}-a^{-2}-1+a^4)$.
\end{remark}

\subsubsection{Scheme $(id, V)$}

We assume here that the original knot theory $\mathcal M$ is the theory of oriented virtual knots $\mathcal M_{virt}^{+}$ or regular oriented virtual knots $\mathcal M_{virt}^{reg+}$, and the destination knot theory $\mathcal M'$ includes virtual Reidemeister moves $V\Omega_1, V\Omega_2, V\Omega_3, SV\Omega_3$.

Let $\mathcal T_0=id=\{\skcrr,\skcrl\},\ \mathcal T_1=V=\{\skcrvirt,\skcrvirt\}$. The invariance conditions for the corresponding local transformation rule $\tau$ are shown in Fig.~\ref{fig:id_virtualization}.

\begin{figure}[p]
\centering\includegraphics[width=0.35\textwidth]{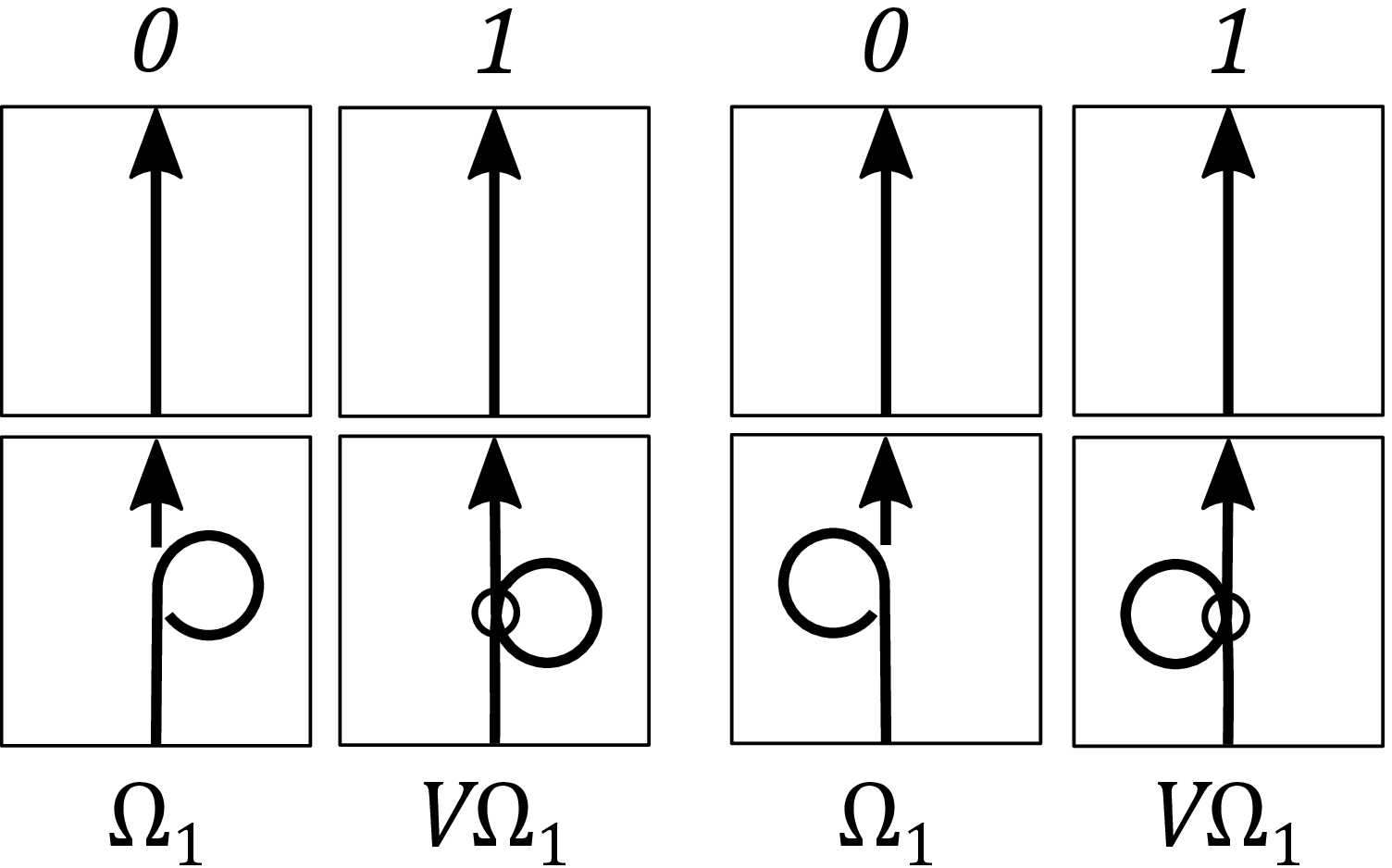}\\ \medskip
 \includegraphics[width=0.7\textwidth]{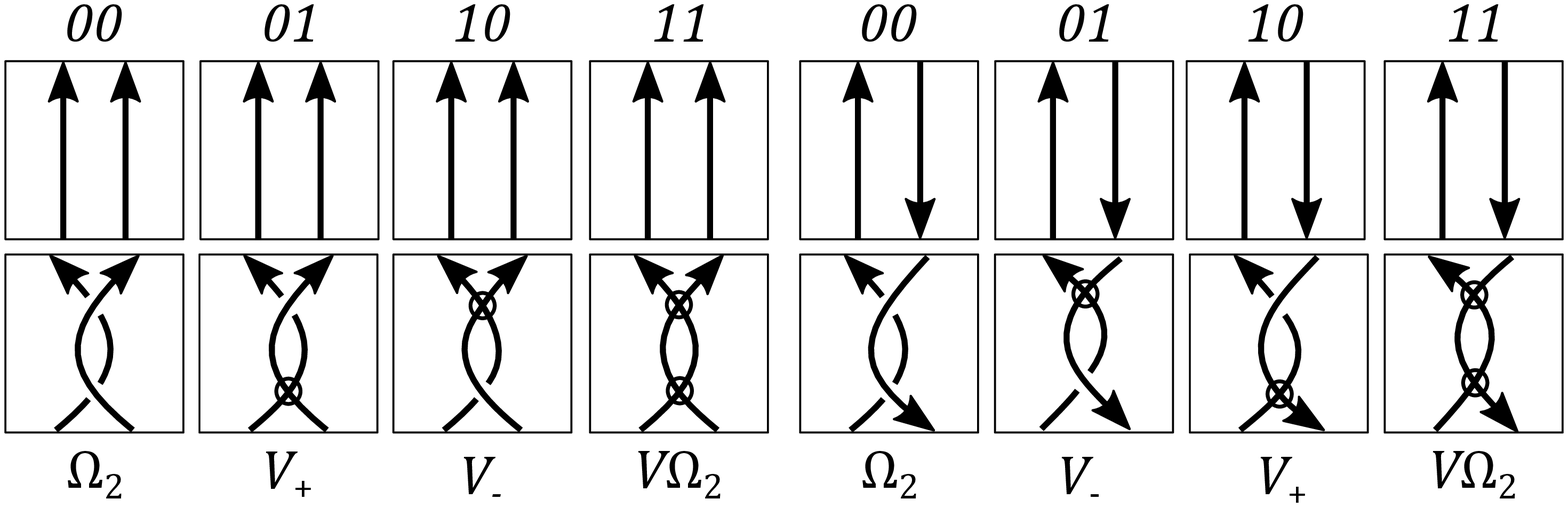}\\ \medskip
\includegraphics[width=0.7\textwidth]{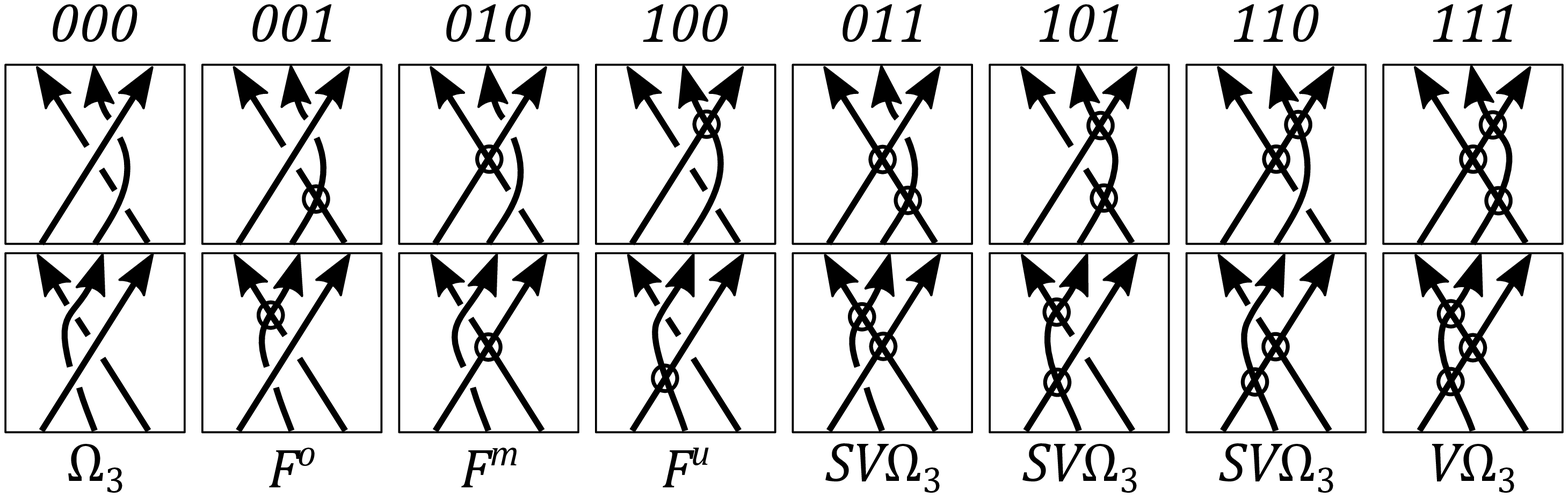}
\caption{Invariance conditions for the scheme $(id, V)$}\label{fig:id_virtualization}
\end{figure}

Let $\tau$ be not an index. Then, for example, we have a second Reidemeister move with an even positive crossing and an odd negative crossing. The invariance condition yields the move $V_+$. Hence, we can virtualize all positive crossings in the diagram $f_\tau(D)$. If there is an even negative crossing then we get the move $V_-$ and virtualize all negative crossings. Thus, the functorial map is reduced to the unary functorial map $V$.

Let $\tau$ be an index. The set of moves required by the invariance for third Reidemeister moves depends on the type of binary trait $\tau$ (Section~\ref{app:binary_trait_types}). By assumption $\mathcal M'$ includes virtual Reidemeister moves, hence, we need to check if third Reidemeister moves of types $000$, $001$, $010$, $100$ can occur. Since any three moves among the moves $\Omega_3, F^o, F^u, F^m$ generate the fourth move (see the proof of~\cite[Theorem 1]{Nel}), then $\mathcal M'$ may include one (types $e,i$), two (types $c,d,j,o,p,q$) or four moves (other types) of the set $\Omega_3, F^o, F^u, F^m$. In the first case we get the virtual knot theory, in the second, the welded knot theory or the theory of what we call \emph{fused doodles}. The last case yields the theory of fused knots.

According to Appendix~\ref{app:binary_trait_types}, the trait types $e,i,c,d,o,p$ have even loop values. Then the invariance for the first Reidemeister move gives the move $\Omega_1$. The loop values for the trait types $j,q$ are all odd that leads to the move $V\Omega_1$ which belongs to $\mathcal M'$  by the assumption.

We can summarize the reasoning above in the following table.

\begin{center}
\begin{tabular}{|c|c|c|}
\hline
$\mathcal M$ & $\tau$ & $\mathcal M'$\\
\hline
 & trait & $V,V\Omega_1, V\Omega_2, V\Omega_3$\\
 & index, $\exists\tau^{r/l\pm}=0$ & $\mathcal M_{fused}^{+}$\\
 & index, $\forall\tau^{r/l\pm}=1$ & $\mathcal M_{fused}^{reg+}$\\
$\mathcal M_{virt}^{+}$ & index of types $e$, $i$ & $\mathcal M_{virt}^{+}$\\
 & index of types $c$, $d$, $o$, $p$ & $\mathcal M_{welded}^{+}$\\
 & index of types $j$, $q$ & $\mathcal M_{fd}^{reg+}$ \\
\hline
 & trait & $V,V\Omega_1, V\Omega_2, V\Omega_3$\\
 & index & $\mathcal M_{fused}^{reg+}$\\
$\mathcal M_{virt}^{reg+}$ & index of types $e$, $i$ & $\mathcal M_{virt}^{reg+}$\\
 & index of types $c$, $d$, $o$, $p$ & $\mathcal M_{welded}^{reg+}$\\
 & index of types $j$, $q$ & $\mathcal M_{fd}^{reg+}$ \\
\hline
\end{tabular}
\end{center}

Here
\begin{itemize}
\item $\mathcal M_{virt}^{reg+}=\mathcal M_{class}^{reg+}\cup\{V\Omega_1, V\Omega_2, V\Omega_3, SV\Omega_3\}$ is the theory of regular virtual knots,
\item $\mathcal M_{welded}^{+}=\mathcal M_{virt}^{+}\cup\{F^o\}$ or $\mathcal M_{virt}^{+}\cup\{F^u\}$ is the theory of oriented welded knots,
\item $\mathcal M_{welded}^{reg+}=\mathcal M_{virt}^{reg+}\cup\{F^o\}$ or $\mathcal M_{virt}^{reg+}\cup\{F^u\}$ is the theory of regular welded knots,
\item $\mathcal M_{fused}^{reg+}=\mathcal M_{virt}^{reg+}\cup\{F^o, F^u\}$ is the theory of regular fused knots,
\item $\mathcal M_{fused}^{+}=\mathcal \mathcal M_{virt}^{+}\cup\{F^o, F^u\}$ is the theory of fused knots,
\item $\mathcal M_{fd}^{reg+}=\{\Omega_2,V\Omega_1, V\Omega_2, V\Omega_3, SV\Omega_3, F^o, F^u\}$ is the theory of oriented regular fused doodles.
\end{itemize}

Thus, the following statement holds.
\begin{proposition}\label{prop:binary_virtualizing}
1. Let $\tau$ be a binary trait with the scheme $(id, V)$ on oriented virtual links.
\begin{itemize}
  \item If $\tau$ is an index then it defines a functorial map to the oriented fused links.
  \item If $\tau$ is a week parity then it defines a functorial map to the oriented virtual links.
  \item If $\tau$ is an index of type $c,d,o$ or $p$ then it defines a functorial map to the oriented welded links.
  \item If $\tau$ is an index of type $j$ or $q$ then it defines a functorial map to the oriented regular fused doodles.
\end{itemize}

2. Let $\tau$ be a binary trait with the scheme $(id, V)$ on oriented regular virtual links.
\begin{itemize}
  \item If $\tau$ is an index then it defines a functorial map to the oriented regular fused links.
  \item If $\tau$ is a week parity then it defines a functorial map to the oriented regular virtual links.
  \item If $\tau$ is an index of type $c,d,o$ or $p$ then it defines a functorial map to the oriented regular welded links.
  \item If $\tau$ is an index of type $j$ or $q$ then it defines a functorial map to the oriented regular fused doodles.
\end{itemize}
\end{proposition}

A fused link is characterized by its linking matrix~\cite{FK}.

\begin{corollary}
1. Let $\tau$ be an index with the scheme $(id, V)$ on oriented virtual links with $n$ components and $f_\tau$ the correspondent functorial map with values in the set $\mathscr K^+(\R^2\mid \mathcal M_{fused}^{+})_n\simeq\Z^{n(n-1)}$ of $n$-component fused links. Then for any link diagram $D=D_1\cup\cdots\cup D_n$  we have $f_\tau(D)=(lk_{ij}^{ev}(D))_{1\le i\ne j\le n}$ where
\[
lk_{ij}^{ev}(D)=\sum_{c\colon D_i\mbox{ \scriptsize over } D_j,\ \tau(c)=0}sgn(c).
\]

2. Let $\tau$ be an index with the scheme $(id, V)$ on oriented regular virtual links with $n$ components and $f_\tau$ the correspondent functorial map with values in the set $\mathscr K^{reg+}(\R^2\mid \mathcal M_{fused}^{reg+})_n\simeq\Z^{n^2}$ of $n$-component regular fused links. Then for any link diagram $D=D_1\cup\cdots\cup D_n$  we have $f_\tau(D)=(lk_{ij}^{ev}(D))_{1\le i, j\le n}$.
\end{corollary}

\subsubsection{Scheme $(id, CC)$}

We will assume that the destination knot theory $\mathcal M'$ contains the classical Reidemeister moves: $\mathcal M_{class}^+\subset\mathcal M'$.

Let $\mathcal T_0=id=\{\skcrr,\skcrl\},\ \mathcal T_1=CC=\{\skcrl,\skcrr\}$. The invariance conditions for the corresponding local transformation rule $\tau$ are shown in Fig.~\ref{fig:id_CC}.

\begin{figure}[p]
\centering\includegraphics[width=0.35\textwidth]{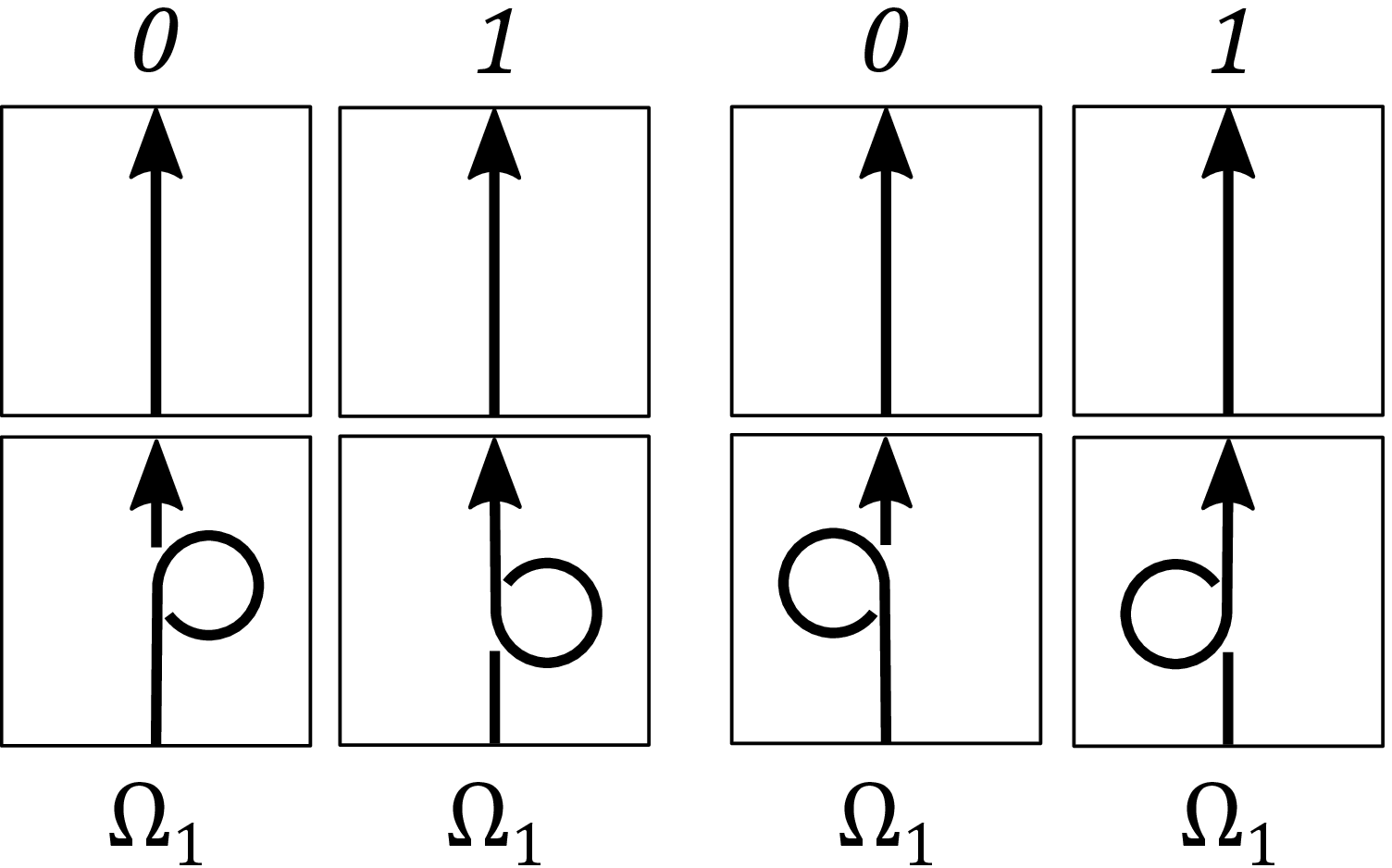}\\ \medskip
 \includegraphics[width=0.7\textwidth]{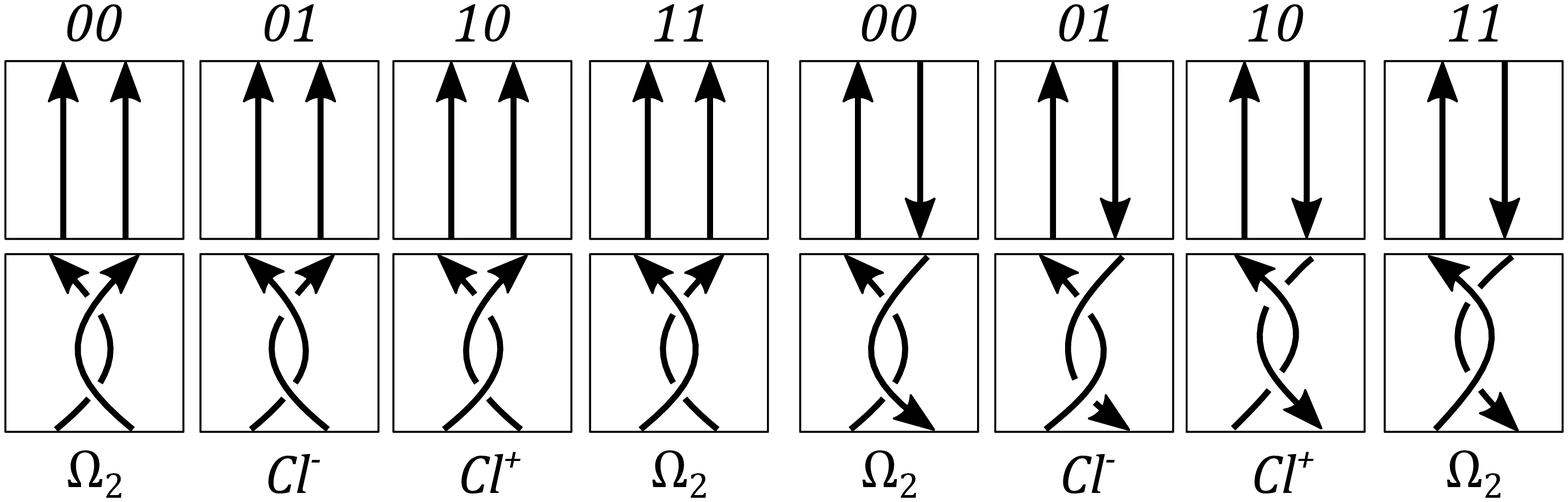}\\ \medskip
\includegraphics[width=0.7\textwidth]{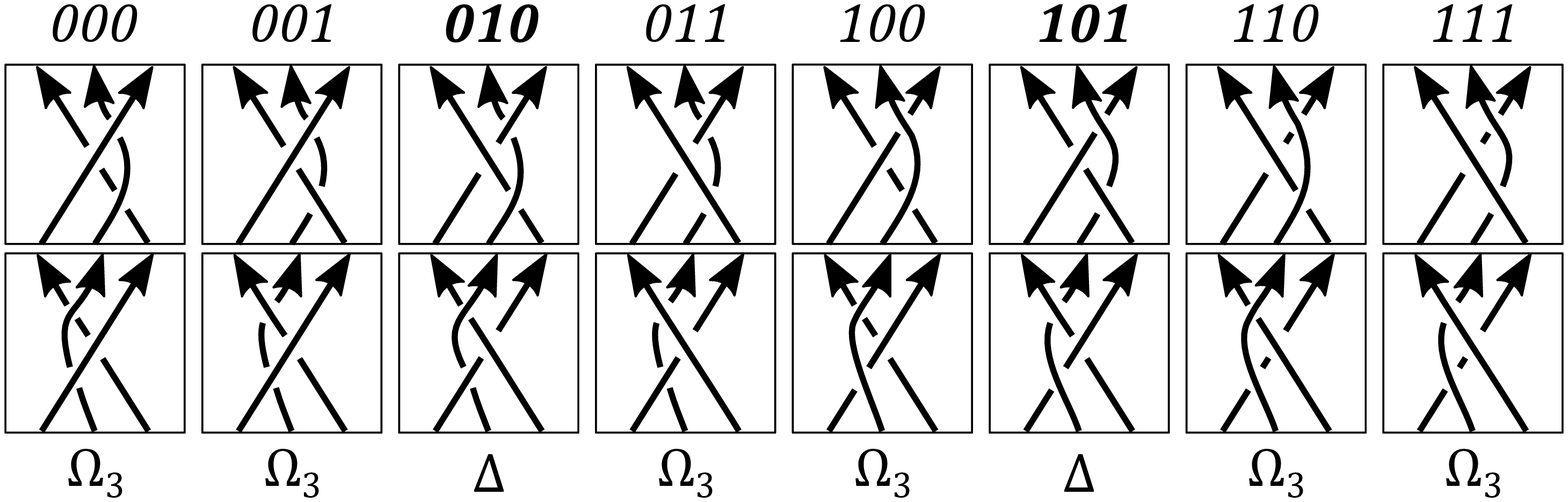}
\caption{Invariance conditions for the scheme $(id, CC)$}\label{fig:id_CC}
\end{figure}

\begin{figure}[p]
\centering\includegraphics[width=0.35\textwidth]{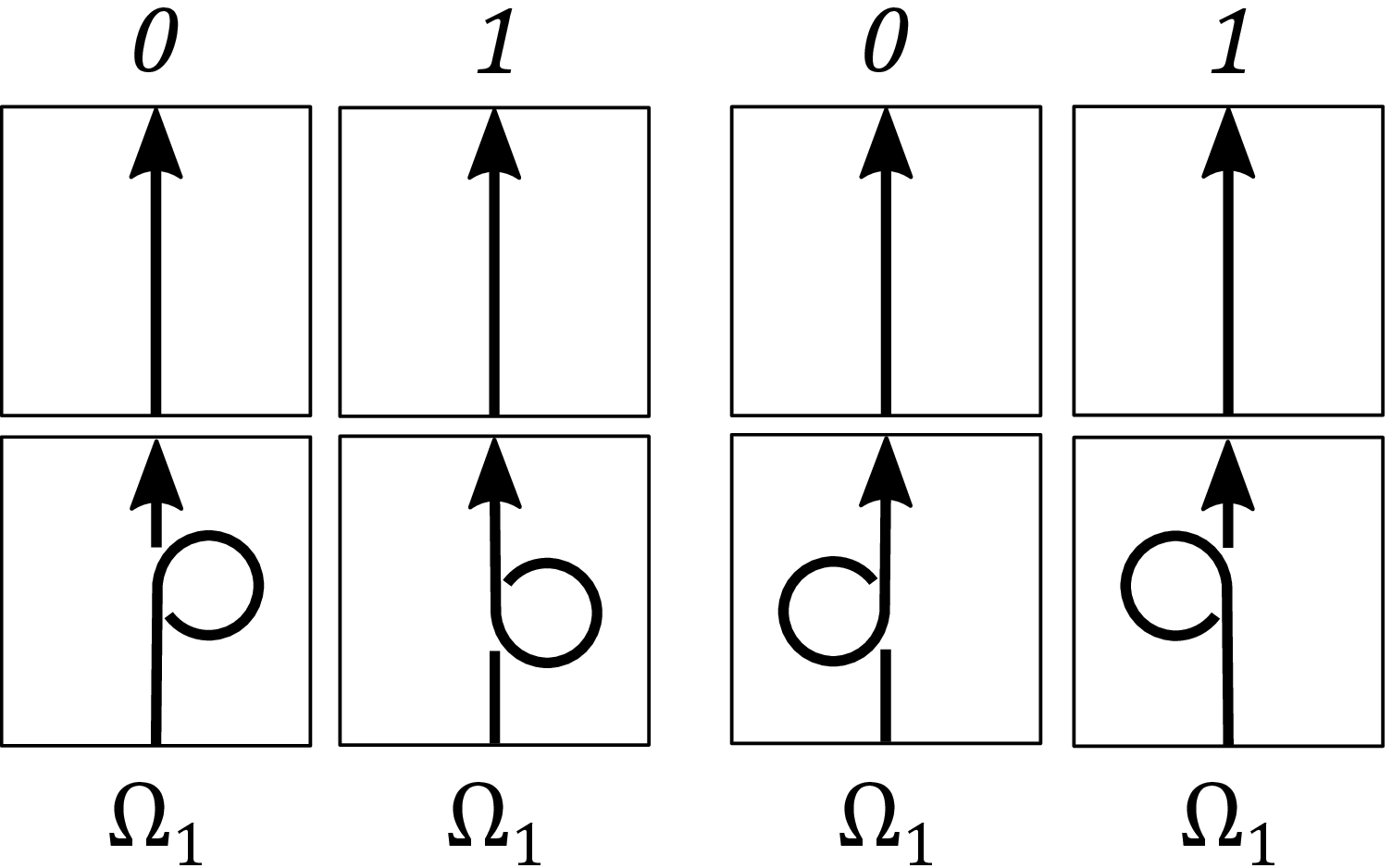}\\ \medskip
 \includegraphics[width=0.7\textwidth]{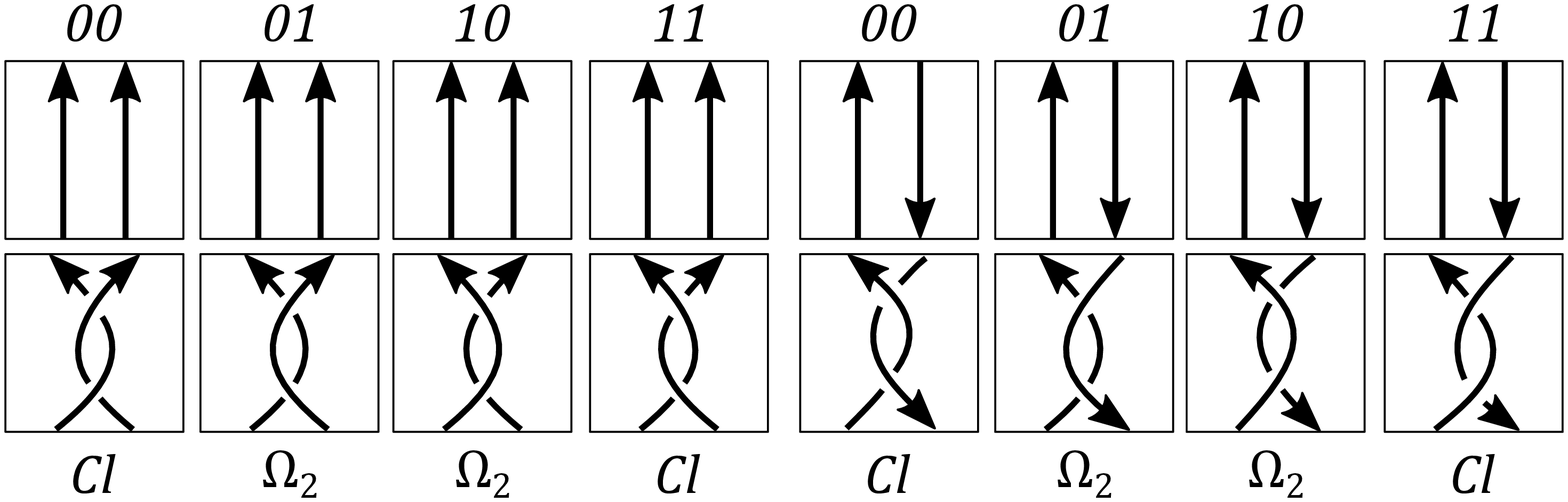}\\ \medskip
\includegraphics[width=0.7\textwidth]{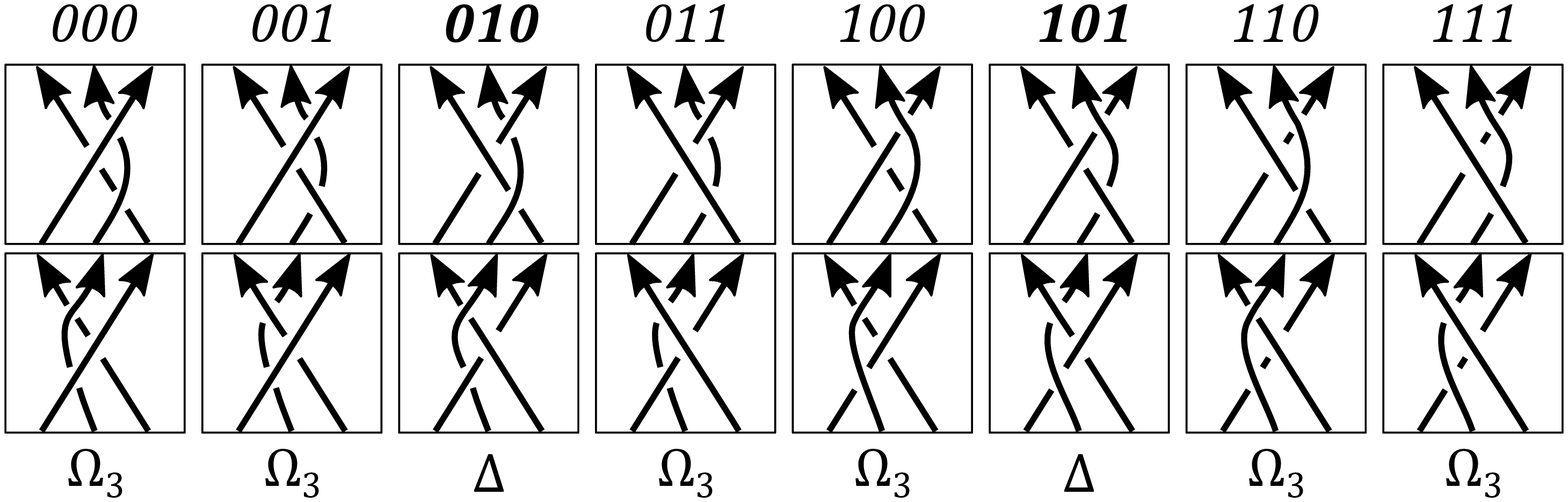}
\caption{Invariance conditions for the lifting}\label{fig:lifting_conditions}
\end{figure}

If $\tau$ is not an index then the invariance for second Reidemeister move requires the clasp move $Cl^\pm$. Together with the second Reidemeister move, this move generates the crossing change move $CC$. Then the functorial map $f_\tau$ is reduced to the projection from the classical knots to the flat knots.

Let $\tau$ be an index.

\begin{definition}\label{def:order}
An index $\tau$ with values in $\{0,1\}$ is called an \emph{order} if for any vertices $c_1,c_2,c_3$ participating in a move $\Omega_{3b}$ the combination $\tau(c_1)=\tau(c_3)\ne\tau(c_2)$ can not occur.
\end{definition}

If $\tau$ is not an order then the invariance for the third Reidemeister move gives the $\Delta$-move. If $\tau$ is an order the Reidemeister moves ensure the invariance.

The invariance for the first Reidemeister move yield requires nothing beyond the Reidemeister moves.

Thus, we get the following table.

\begin{center}
\begin{tabular}{|c|c|c|}

\hline
$\mathcal M$ & $\tau$ & $\mathcal M'$\\
\hline
 & trait & $CC,\Omega_1,\Omega_2,\Omega_3$\\
$\mathcal M_{class}^{+}$  & index & $\Delta,\Omega_1,\Omega_2,\Omega_3$\\
 & order & $\Omega_1,\Omega_2,\Omega_3$\\
\hline
 & trait & $CC,\Omega_2,\Omega_3$\\
$\mathcal M_{class}^{reg+}$  & index & $\Delta,\Omega_2,\Omega_3$\\
 & order & $\Omega_2,\Omega_3$\\
\hline
\end{tabular}
\end{center}

\begin{proposition}\label{prop:binary_crossing_change}
Let $F$ be a surface and $\tau$ a local transformation rule with the scheme $(id, CC)$ on the diagram set $\mathscr D(F)$.
\begin{itemize}
\item If $\tau$ is not an index then $f_\tau(D)$ is the shadow (flat diagram) of the diagram $D$;
\item If $\tau$ is an index then $f_\tau(D)$ can be identified with the extended homotopy index polynomial $LK(D)$ of $D$ (see Section~\ref{app:skein_modules});
\item If $\tau$ is an order then the functorial map $f_\tau$ in the tangles in the surface $F$.
\end{itemize}
\end{proposition}

\begin{remark}
A analogue of Proposition  is valid for virtual knots, links and tangles. In partucular, for an order $\tau$ on diagrams of a virtual knot (link, tangle) the correspondent functorial map takes values in diagrams of another virtual knot (link, tangle).
\end{remark}

In the next section we consider orders on diagrams in more detail.

\subsection{Crossing change maps and orders}\label{subsect:cc_map_orders}

We start the section with examples of orders.

\begin{example}\label{ex:orders}
\begin{enumerate}
\item $\tau\equiv 0$ is an order which corresponds to the identity $f_\tau=\id$.
\item $\tau\equiv 1$ is an order which corresponds to the mirror map of a knot $f_\tau=CC$.
\item Let $L=K_1\cup\cdots\cup K_n$ be a link, $\mathscr C=\{C_1,C_2\}$ a splitting on the component set, i.e. $C_1\cup C_2=\{1,\dots,n\}$ and $C_1\cap C_2=\emptyset$, and $\epsilon_1,\epsilon_2\in\{0,1\}$. Define an index $\tau^{\mathscr C}$ as follows. For any crossing $c$ with the component index $(i,j)$ we set
\[
\tau^{\mathscr C}(c)=\left\{\begin{array}{cl}
                              \epsilon_k, & i,j\in C_k, k=1,2, \\
                              0, & i\in C_1, j\in C_2, \\
                              1, & i\in C_2, j\in C_1.
                            \end{array}\right.
\]
Then $\tau^{\mathscr C}$ is an order. Informally speaking, we lift the components in $C_1$ over the components in $C_2$ and mirror the part $C_1$ or $C_2$ of the link diagram if necessary.
\item Let $K$ be a long knot. Then the order index $o$ is an order.
\item If $\tau$ is an order then $\bar\tau=1-\tau$ is an order called the \emph{mirror order} of $\tau$.
\item If $\tau_1$ and $\tau_2$ are orders then their sum $\tau_1+\tau_2$ (valued in $\Z_2$) is an order too. The corresponding functorial map is the composition of the functorial maps of the orders $\tau_1$ and $\tau_2$: $f_{\tau_1+\tau_2}=f_{\tau_1}\circ f_{\tau_2}$. Thus, the set of orders possesses a $\Z_2$-module structure.

\end{enumerate}
\end{example}

\subsubsection{Orders on knots in a fixed surface}

Let us describe orders on tangle diagrams in a fixed oriented compact surface $F$.
Let $\tau$ be a binary trait and $D=D_1\cup\cdots\cup D_l$ a tangle diagram in $F$. Let $H^{t}_{ij}\subset\pi_1(F,z_i,z_j)$ be the set of homotopy indices of crossings $c$ with the component index $(i,j)$ such that $\tau(c)=t\in\{0,1\}$. Analogously, for a long component $i$ we define subsets ${}^u\!H_{ii}^t$ for the early undercrossings and ${}^o\!H_{ii}^t$ for the early overcrossings.

\begin{proposition}\label{prop:order_homotopy_index}
Let $\tau$ be a binary index. Then $\tau$ is an order if and only if it satisfies the following conditions.
\begin{enumerate}
\item For any $t=0,1$ and $i,j,k$ one has $H^t_{ij}H^t_{jk}\subset H^t_{ik}$ provided among $i,j,k$ there are no coinciding indices of a long component.
\item For any closed component $i$ and $t=0,1$ one has $\kappa_i^{-1}H^t_{ii}H^t_{ii}\subset H^t_{ii}$ 
where $\kappa_i\in\pi_1(F,z_i)$ is the homotopy class of the component.
\item For any long component $i$, $({}^\alpha\!H^t_{ii})^\alpha({}^\beta\!H^t_{ii})^\beta\subset({}^\gamma\!H^t_{ii})^\gamma$  for any $t=0,1$ and for any combination of $\alpha,\beta,\gamma\in\{u,o\}$ except cases $\alpha=\beta\ne\gamma$. Here we denote $X^u=X$ and $X^o=X^{-1}$.
\item For any long component $i$ and another component $j\ne i$ we have
\[
({}^\alpha\!H^t_{ii})^\alpha\cdot H^t_{ij}\subset H^t_{ij},\quad H^t_{ji}({}^\alpha\!H^t_{ii})^\alpha\subset H^t_{ji}, \quad H^t_{ij} H^t_{ji}\subset({}^\alpha\!H^t_{ii})^\alpha
\]
for any $t=0,1$ and any $\alpha\in\{u,o\}$.
\end{enumerate}
\end{proposition}

\begin{proof}

1) Assume first that $\tau$ is an order.

Let $D_i$ be a closed component of the tangle. Let self-crossings $u,v,w$ of the component $D_i$ form a triangle of a Reidemeister move $\Omega_{3b}$. There are two possible cases of such a configuration (Fig.~\ref{fig:R3combination1}).

\begin{figure}[h]
\centering\includegraphics[width=0.4\textwidth]{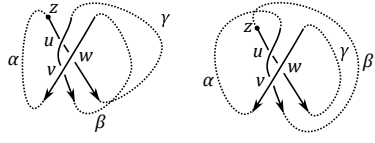}
\caption{Move $\Omega_{3b}$ on self-crossings of a closed component}\label{fig:R3combination1}
\end{figure}

In the left case in Fig.~\ref{fig:R3combination1} we have the following homotopy indices: $h(u)=\gamma$, $h(v)=\gamma\beta\gamma^{-1}$, $h(w)=\gamma\beta$. Hence, $h(v)h(u)=h(w)$. In the right case, $h(u)=\gamma\alpha$, $h(v)=\gamma\alpha\beta\gamma\alpha^{-1}\gamma^{-1}$, $h(w)=\gamma$. Then $\kappa_i^{-1}h(v)h(u)=h(w)$ where $\kappa_i=\gamma\alpha\beta\in\pi_1(F,z)$ is the homotopy type of the component.
Given $t=0$ or $1$, for any $x,y\in H^t_{ii}$ using moves $\Omega_2,\Omega_3$, we can create triangles for $\Omega_{3b}$ such that $h(u)=x$ and $h(v)=y$. Then $h(w)=xy$ in the left case and $h(w)=\kappa_i^{-1}xy$ in the right case. Since $\tau(u)=\tau(v)=t$, the definition of order implies $\tau(w)=t$. Then $xy,\kappa_i^{-1}xy\in H^t_{ii}$. Thus, $H^t_{ii}H^t_{ii}\subset H^t_{ii}$ and $\kappa_i^{-1}H^t_{ii}H^t_{ii}\subset H^t_{ii}$.

Next, consider a triangle $uvw$ for a move $\Omega_{3b}$ formed by two components (Fig.~\ref{fig:R3combination2}). In the first case, $h(u)=\gamma\in\pi_1(F,z_2)$, $h(v)=\gamma^{-1}\in\pi_1(F,z_1,z_2)$ and $h(w)=1\in\pi_1(F,z_1,z_2)$. Here we ignore small arcs in the neighbourhood of the triangle $uvw$. Hence, $h(v)h(u)=h(w)$.

\begin{figure}[h]
\centering\includegraphics[width=0.6\textwidth]{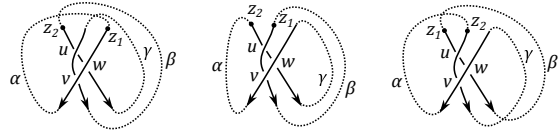}
\caption{Move $\Omega_{3b}$ on crossings of two components}\label{fig:R3combination2}
\end{figure}

In the second case, $h(u)=1\in\pi_1(F,z_1,z_2)$, $h(v)=\gamma\in\pi_1(F,z_2,z_1)$, $h(w)=\gamma\in\pi_1(F,z_2)$ and $h(v)h(u)=h(w)$. In the third case, $h(u)=1\in\pi_1(F,z_2,z_1)$, $h(v)=\gamma\in\pi_1(F,z_2,z_2)$, $h(w)=\gamma\in\pi_1(F,z_2,z_1)$ and $h(v)h(u)=h(w)$.

Using reasonings analogous to the one-component case, we get the inclusions $H^t_{ij}H^t_{jj}\subset H^t_{ij}$, $H^t_{ij}H^t_{ji}\subset H^t_{ii}$ and $H^t_{ii}H^t_{ij}\subset H^t_{ij}$.

When a triangle $uvw$ for a move $\Omega_{3b}$ is formed by three components (Fig.~\ref{fig:R3combination3}) we have $h(u)=1\in\pi_1(F,z_2,z_3)$, $h(v)=1\in\pi_1(F,z_1,z_2)$, $h(w)=1\in\pi_1(F,z_1,z_3)$ and $h(v)h(u)=h(w)$. Hence, $H^t_{ij}H^t_{jk}\subset H^t_{ik}$.

\begin{figure}[h]
\centering\includegraphics[width=0.15\textwidth]{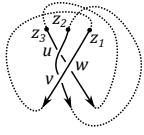}
\caption{Move $\Omega_{3b}$ on three components}\label{fig:R3combination3}
\end{figure}

Let $D_i$ be a long component of the tangle and self-crossings $u,v,w$ of the component form a triangle for $\Omega_{3b}$. There are several possible cases of such a configuration (Fig.~\ref{fig:R3combination1l}).

\begin{figure}[h]
\centering\includegraphics[width=0.6\textwidth]{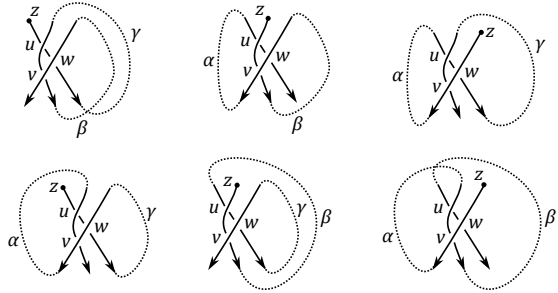}
\caption{Move $\Omega_{3b}$ on a long component}\label{fig:R3combination1l}
\end{figure}

In the first (top left) case we have the order indices $o(u)=o(v)=o(w)=-1$ and the homotopy indices $h(u)=\gamma\in\pi_1(F,z)$, $h(v)=\gamma\beta\gamma^{-1}$, $h(w)=\gamma\beta$. Then $h(v)h(u)=h(w)$, hence ${}^u\!H_{ii}^t {}^u\!H_{ii}^t\subset {}^u\!H_{ii}^t$.

In the second (top middle) case $o(u)=o(w)=1$, $o(v)=-1$ and $h(u)=\beta\alpha$, $h(v)=\beta$, $h(w)=\beta\alpha\beta^{-1}$. Then $h(u)=h(w)h(v)$, i.e. $h(v)h(u)^{-1}=h(w)^{-1}$. Thus, ${}^u\!H^t_{ii}({}^o\!H^t_{ii})^{-1}\subset({}^o\!H^t_{ii})^{-1}$. The other cases are considered analogously.

2) Let $\tau$ be a binary index which satisfies the conditions of the proposition. The conditions ensures that configurations of the move $\Omega_{3b}$ which are forbidden for an order, cannot occur. Thus, $\tau$ is an order.
\end{proof}

\begin{example}[Rough orders]\label{ex:surface_orders}
Let $T=K_1\cup\dots\cup K_n$ be a tangle in the surface $F$ and $\mathscr D_T(F)$ be the set of its diagram. Let us consider orders $\tau$ on $\mathscr D_T(F)$ such that for any $i,j\in\{1,\dots,n\}$ the set $H_{ij}^0$ or $H_{ij}^1$ is empty (hence, the other set coincides with $\pi_1(F,z_i,z_j)$). Thus, the value of the order depends only on the component index of the crossing. Such orders will be called \emph{rough orders}.

Consider a splitting $\mathscr C=\{C_1,\dots, C_m\}$ of the set of components $\{1,\dots,n\}$ into $m$ subsets, $m\le n$. Consider a vector $\mathbf\epsilon=(\epsilon_1,\dots,\epsilon_m)\in\{0,1\}^m$ and a vector $\mathbf\omega=(\omega_i)_{i\in\Lambda}$ where $\Lambda\subset\{1,\dots,n\}$ is the set of long components of the tangle. We require that for any long component $i$ such that $\omega_i=1$ the subset $C_k$ which contains $i$ is one-element, i.e. $C_k=\{i\}$.

Define an index $\tau^{\mathscr C,\mathbf\epsilon, \mathbf\omega}$ on  $\mathscr D_T(F)$ by the formula
\[
\tau^{\mathscr C, \mathbf\epsilon, \mathbf\omega}(c)=\left\{\begin{array}{cl}
                              \epsilon_k+\omega_i\cdot o(c), & i,j\in C_k, 1\le k\le m, \\
                              0, & i\in C_k, j\in C_l, k>l, \\
                              1, & i\in C_k, j\in C_l, k<l.
                            \end{array}\right.
\]
where $(i,j)$ is the component index of the crossing $c$ and $o(c)$ is the order index of the crossing. If $c$ is not a self-crossing of a long component then we suppose $o(c)=0$.

In other words, $H_{ij}^0=\emptyset$ if $i,j\in C_k$ and $\epsilon_k=1$ or $i\in C_k$, $j\in C_l$ and $k<l$. $H_{ij}^1=\emptyset$ if $i,j\in C_k$ and $\epsilon_k=0$ or $i\in C_k$, $j\in C_l$ and $k>l$. Then $\tau^{\mathscr C,\mathbf\epsilon, \mathbf\omega}$ is a rough order.

Note that $\tau^{\mathscr C,\mathbf\epsilon, \mathbf\omega}$ generalizes the order $\tau^{\mathscr C}$ in Example~\ref{ex:orders}. On the other hand, the functorial map $f_{\tau^{\mathscr C,\mathbf\epsilon, \mathbf\omega}}$ is a composition of functorial maps $f_{\tau^{\mathscr C}}$ for some splittings $\mathscr C$.

\begin{proposition}
Any rough order $\tau$ coincides with $\tau^{\mathscr C,\mathbf\epsilon, \mathbf\omega}$ for some $\mathscr C$,  $\mathbf\epsilon$ and $\mathbf\omega$.
\end{proposition}
\begin{proof}
Let $\tau$ be a rough order. Consider the digraph $G$ with the set of vertices $V(G)=\{1,\dots,n\}$ and the set of edges $E(G)=\{ij\mid H^1_{ij}\ne\emptyset
\}$. By Proposition~\ref{prop:order_homotopy_index} the graph $G$ possesses the following properties:
\begin{enumerate}
\item if $ij,jk\in E(G)$ then $ik\in E(G)$;
\item if $ij\in E(G)$ then for any $k$ either $ik\in E(G)$ or $kj\in E(G)$.
\end{enumerate}

Call two vertices $i$ and $j$ \emph{equivalent} ($i\sim j$) if $ij,ji\in E(G)$. Then $\sim$ is an equivalence relation.

Let $\bar G=G/\sim$. The graph $\bar G$ is a directed acyclic graph and $V(\bar G)$ is a poset. Let $\bar C_1$ be the set of minimal vertices in $V(\bar G)$ (i.e. the set of vertices without incoming edges). We show that for any $x\in\bar C_1$ and $y\not\in\bar C_1$ one has $ij\in E(\bar G)$.

Indeed, consider a minimal vertex $z\in V(\bar G)\setminus \bar C_1$. Then there is an edge $xz\in E(\bar G)$ where $x\in\bar C_1$. By property 2) applied to the edge $xz$, for any $y\in\bar C_1$ $yz\in E(\bar G)$, and for any $y\not\in\bar C_1$ $xy\in E(\bar G)$. We can replace $x$ with any $x'\in \bar C_1$ and get the condition that for any $y\not\in\bar C_1$ $x'y\in E(\bar G)$.

Repeating the reasonings above to the full subgraph of $\bar G$ on the vertex set $V(\bar G)\setminus \bar C_1$, we get a vertex set $\bar C_2$ etc. Finally, we get a splitting $V(\bar G)=\bar C_1\sqcup\cdots\sqcup\bar C_m$. Let $C_k\subset V(G)$, $k=1,\dots,m$, be the correspondent sets of vertices of $G$. Denote $\mathscr C=\{C_k\}_{k=1}^m$. Then for any $i\in C_k$, $j\in C_l$, $k<l$, we have $ij\in E(G)$ and $ji\not\in E(G)$.

If a set $C_k$ contains only trivial equivalence classes, then there is a bijection between $C_k$ and $\bar C_k$. By construction, there is no edges between the vertices of $\bar C_k$. Then the same holds for the vertices of $C_k$. In this case we set $\epsilon_k=0$.

Let a set $C_k$ contain a nontrivial equivalence class, i.e. $i\sim j\in C_k$. By property 1), $ii\in E(G)$, hence by property 2), for any $i'\in V(G)$ either $ii'\in V(G)$ or $i'i\in V(G)$. On the other hand, there is no edges between the vertices of $\bar C_k$. Thus, the set $\bar C_k$ has only one element, and $C_k$ is the equivalence class of the vertex $i$. In this case we set $\epsilon_k=1$.

Let $i$ be a long component and $C_k$ the set of the splitting which includes it. We set $\epsilon_i=\tau(c)$ where $c$ is an early overcrossing of $i$, and set $\omega_i=0$ if the values of $\tau$ on early overcrossings and early undercrossings of the component $i$ coincide, and $\omega_i=1$ if the values are different. Assume that $\omega_i=1$. Then ${}^u\!H_{ii}^t={}^o\!H_{ii}^{1-t}$ for some $t=0,1$. By Proposition~\ref{prop:order_homotopy_index}, for any $j\ne i$ exactly one of the edges $ij$ and $ji$ belong to $E(G)$. Hence, there is no components equivalent to $i$, and $C_k=\{i\}$.

Since the digraph $G$ determines the order $\tau$, we have $\tau=\tau^{\mathscr C,\mathbf\epsilon, \mathbf\omega}$.
\end{proof}
\end{example}

\subsubsection{Preordering}\label{subsect:preordering}

For a knot, relations like those in Proposition~\ref{prop:order_homotopy_index} appear on (pre)ordered groups.

\begin{definition}[\cite{DR}]\label{def:preordering}
Let $G$ be a group. A \emph{left-invariant preordering} on $G$ is a reflexive, transitive and complete relation $\preceq$ on $G$ such that $g\preceq g'$ implies $hg\preceq hg'$ for any $h\in G$.
\end{definition}

\begin{proposition}[\cite{DR}]\label{prop:preordering_subsets}
Let $H$ be a subset of $G$ such that $H\cup H^{-1}=G$, and $HH\subset H$. Then $g\preceq h \Leftrightarrow g^{-1}h\in H$ defines a left-invariant preordering on $G$.
\end{proposition}

\begin{definition}[cf.~\cite{LRR}]\label{def:discrete_preordering}
Let $\preceq$ be a preordering. Denote $g\sim g'$ if $g\preceq g'$ and $g'\preceq g$, and $g\prec g'$ if $g\preceq g'$ and $g\nsim g'$.

A preordering is \emph{discrete} if there exists an element $a\in G$ such that $1\prec a$ and there is no element $b$ such that $1\prec b\prec a$. The element $a$ is called a \emph{least positive element}.
\end{definition}

Let us remind some basic facts on preordered groups.
\begin{lemma}\label{lem:preorder_basics}
Let $G$ be a group with a left-invariant preordering $\preceq$. Then
\begin{enumerate}
\item $1\prec g \Leftrightarrow g^{-1}\prec 1$,
\item $g\sim 1 \Leftrightarrow g^{-1}\sim 1$,
\item $1\preceq g, 1\prec h \Rightarrow 1\prec gh, 1\prec hg$,
\item $g\preceq 1, h\prec 1 \Rightarrow gh\prec 1, hg\prec 1$,
\item $g\sim 1, h\sim 1 \Rightarrow gh\sim 1$.
\end{enumerate}
\end{lemma}

\begin{proof}
1) Let $1\prec g$. Then $g^{-1}\prec g^{-1}g=1$.

2) Let $1\preceq g, 1\prec h$. Then $1\preceq g\prec gh$ and $1\prec h\preceq hg$, hence, $1\prec gh, 1\prec hg$.

The other statements are proved analogously.
\end{proof}

\begin{lemma}\label{lem:discrete_preorder_basics}
Let $G$ be a group with a discrete left-invariant preordering $\preceq$ and $a$ is a least positive element. Then
\begin{enumerate}
\item $1\preceq g, 1\preceq h \Rightarrow 1\prec agh$,
\item $g\prec 1, h\prec 1 \Rightarrow agh\prec 1$,
\item $1\prec g$ (resp. $g\sim 1$, $g\prec 1$) $\Rightarrow 1\prec a^{\pm 1}ga^{\mp 1}$ (resp. $a^{\pm 1}ga^{\mp 1}\sim 1$, $a^{\pm 1}ga^{\mp 1}\prec 1$)
\end{enumerate}
\end{lemma}

\begin{proof}
The first statement follows from Lemma~\ref{lem:preorder_basics}.

Let $g\prec 1, h\prec 1$. Since $ag\prec a$ and $a$ is a least positive element, then $ag\preceq 1$. Hence, $agh\prec 1$.

Let $g\prec 1$. Since $a^{-1}\prec 1$ then $aga^{-1}\prec 1$ by the previous statement. If $1\prec g$ then $g^{-1}\prec 1$ and $ag^{-1}a^{-1}\prec 1$. Hence $1\prec (ag^{-1}a^{-1})^{-1}=aga^{-1}$.

Let $1\prec g$. Then $a\preceq g$, hence, $1\preceq a^{-1}g$ and $1\prec a^{-1}ga$. If $g\prec 1$ then $1\prec g^{-1}$ and $1\prec a^{-1}g^{-1}a$. Hence $1\prec (a^{-1}g^{-1}a)^{-1}=a^{-1}ga\prec 1$.

Let $g\sim 1$. If $aga^{-1}\prec 1$ (resp. $1\prec aga^{-1}$) then $g=a^{-1}(aga^{-1})a\prec 1$ (resp. $1\prec g$). Hence, $aga^{-1}\sim 1$. Analogously, $a^{-1}ga\sim 1$.
\end{proof}

\begin{definition}\label{def:order_K-invariance}
Let $F$ be a surface and $K$ a knot in the thickening $F\times [0,1]$, $z\in K$. Let $IM(K)\subset Aut(\pi_1(F,z))$ be the group of automorphisms induced by the diffeomorphisms $\Phi$ of $F\times[0,1]$ such that $\Phi$ is identical on $\partial(F\times[0,1])\cup K$ and $\Phi$ is isotopic to the identity (cf. Section~\ref{subsec:traits}). A preordering $\preceq$ is called \emph{$K$-invariant} if for any $\phi\in IM(K)$ and $g\in\pi(F)$, $1\preceq g$, one has $1\preceq\phi(g)$.
\end{definition}

\begin{remark}\label{rem:inner_monodromy_structure}
Since $IM(K)$ is generated by diffeomorphisms isotopic to the identity, the automorphisms in $IM(K)$ are inner. Moreover,
\[
IM(K)\subset\{Ad_\gamma\mid \gamma\in\pi_1(F,z), \gamma\kappa\gamma^{-1}=\kappa\}.
\]
Since in the surface group $\pi_1(F)$ any two commuting elements are proportional, if $\kappa\ne 1$ then $IM(K)=\langle Ad_\gamma\rangle$ is a cyclic group where $\kappa=\gamma^k$ for some $k\in\N$.
\end{remark}

\begin{lemma}\label{lem:nontrivial_preorder_invariance}
Let $K$ a knot in the thickening $F\times [0,1]$ such that $\kappa=[K]\in\pi_1(F)$ is nontrivial. Let $\preceq$ be a left-invariant preordering on $\pi_1(F)$ such that either $\kappa\sim 1$ or $\preceq$ is discrete and $\kappa^{-1}$ is a least positive element. Then $\preceq$ is $K$-invariant.
\end{lemma}

\begin{proof}
  If $\kappa\ne 1$ then $\kappa=\gamma^k$, $k\in\N$, for some prime element $\gamma\ne 1$. If $\kappa\sim 1$ then $\gamma\sim 1$ (if $1\prec\gamma$ then $1\prec\gamma^k=\kappa$). Let $1\preceq g$. By Lemma~\ref{lem:preorder_basics}, $1\prec\gamma g\gamma^{-1}$. Since $IM(K)\subset\langle Ad_\gamma\rangle$, the preordering $\preceq$ is $K$-invariant.

  Let $\kappa^{-1}$ be a least positive element. Then $k=1$ and $\kappa=\gamma$ (otherwise $1\prec\gamma^{-1}\prec\kappa^{-1}$). By Lemma~\ref{lem:discrete_preorder_basics}, $1\preceq g$ implies $1\preceq \kappa g\kappa^{-1}$. Thus, $\preceq$ is $K$-invariant.
\end{proof}

\begin{proposition}\label{prop:knot_order_and_preordering}
Let $F$ be a surface and $\mathscr D_K(F)$ the set of diagrams of an oriented knot $K$ in $F$. Denote $\kappa=[K]\in\pi_1(F)$ the homotopy type of the knot.

1. Let $\preceq$ be a $K$-invariant left-invariant preordering on $\pi_1(F)$ such that either $\kappa\sim 1$ or $\preceq$ is discrete and $\kappa^{-1}$ is a least positive element. Denote $H^0=\{g\in\pi_1(F)\mid 1\preceq g\}$. Then
\[
\tau(c)=\left\{\begin{array}{cl}
0, & h(c)=[D_c]\in H^0,\\
1, & h(c)=[D_c]\not\in H^0,
\end{array}\right.
\]
is an order on $\mathscr D_K(F)$.

2. Let $\tau$ be an order on $\mathscr D_K(F)$, $H^t=\{[D_c]\in\pi_1(F)\mid \tau(c)=t\}$ the set of homotopy indices of crossings with the given order value. Let $1\in H^{t_0}$. Then the subset $H^{t_0}$ defines a $K$-invariant left-invariant preordering on $\pi_1(F)$ such that $\kappa\sim 1$ or $\preceq$ is discrete and $\kappa^{-1}$ is a least positive element.
\end{proposition}

\begin{proof}
1. Let $\preceq$ be a preordering on $\pi_1(F)$. Consider the subsets $H^0=\{g\in\pi_1(F)\mid 1\preceq g\}$ and $H^1=\pi_1(F)\setminus H^0$. Since $\preceq$ is $K$-invariant, the map $\tau$ is correct.

If $\kappa\sim 1$ then $\kappa^{-1}\sim 1$. For any $g_0,h_0\in H^0$, $g_0h_0\in H^0$ and $\kappa^{-1}g_0h_0\in H^0$. For any $g_1,h_1\in H^1$, $\kappa^{-1}g_1h_1\in H^1$. Otherwise,  $g_1h_1=\kappa(\kappa^{-1}g_1h_1)\in H^0$. Then by Proposition~\ref{prop:order_homotopy_index}, $\tau$ is an order on  $\mathscr D_K(F)$.

Let $\kappa^{-1}$ be a least positive element. By Lemma~\ref{lem:discrete_preorder_basics}, $\kappa^{-1}H^tH^t\subset H^t$, $t=0,1$. Then by Proposition~\ref{prop:order_homotopy_index} $\tau$ is an order on  $\mathscr D_K(F)$.

2) Let $\tau$ be an order on $\mathscr D_K(F)$ and $\preceq$ the corresponding preordering on $\pi_1(F)$. By definition of homotopy index, the preorder $\preceq$ is $K$-invariant. Assume that $\kappa\not\sim 1$. Since $\kappa^{-1}=\kappa^{-1}\cdot 1\cdot 1\succeq 1$ by Proposition~\ref{prop:order_homotopy_index} then $1\prec\kappa^{-1}$. Let us show that $\kappa^{-1}$ is a least positive element.

Assume that there exists $g$ such that $1\prec g\prec \kappa^{-1}$. Then $g^{-1}\prec 1$ and $1\prec g^{-1}\kappa^{-1}$, hence, $\kappa g\prec 1$. By Proposition~\ref{prop:order_homotopy_index}, $\kappa^{-1}(\kappa g)g^{-1}=1\prec 1$. This contradiction implies that $\preceq$ is discrete and $\kappa^{-1}$ is a least positive element.
\end{proof}

\begin{remark}
By Lemma~\ref{lem:nontrivial_preorder_invariance}, if the knot $K$ is homotopically nontrivial then we can omit the condition of $K$-invariance in Proposition~\ref{prop:knot_order_and_preordering}.
\end{remark}

\begin{corollary}\label{cor:homology_order}
Let $F$ be a surface and $K$ an oriented knot in $F$. Let $\phi\colon H_1(F,\Z)\to \R$ be a homomorphism such that $\phi([K])=0$. Then the map
\[
\tau(c)=\left\{\begin{array}{cl}
0, & \phi([D_c])\ge 0,\\
1, & \phi([D_c])<0,
\end{array}\right.
\]
where $D_c$ is the knot half at the crossing $c$, is an order on the diagrams of the knot.
\end{corollary}
\begin{proof}
The composition of the natural projection from $\pi_1(F)$ to $H_1(F,\Z)$ and $\phi$ is a homomorphism to the real numbers $\R$. The natural ordering on $\R$ induces a preordering $\preceq$ on $\pi_1(F)$ which is left-invariant and $K$-invariant. Then the correspondent trait $\tau$ is an order.
\end{proof}

The correspondence $p\colon c\mapsto [D_c]\in H_1(F,\Z)$ is (up to sign) an example of oriented parity~\cite{Nif} (called homological parity). The composition $\phi\circ p$ is an oriented parity with coefficients in $\R$. We can generalize Corollary~\ref{cor:homology_order} to the case of any parity.

\begin{proposition}\label{prop:parity_order}
Let $p$ be an oriented parity with coefficients in $\Z$ (or $\R$) on a diagram set $\mathscr D$. Define a binary trait $\tau$ by the formula
\[
\tau(c)=\left\{\begin{array}{cl}
0, & sgn(c)p(c)\ge 0,\\
1, & sgn(c)p(c)<0.
\end{array}\right.
\]
Then $\tau$ is an order on $\mathscr D$ and defines an invariant functorial map with the scheme $(id, CC)$.
\end{proposition}

\begin{proof}
Let $u,v,w$ be crossings that participate in a move $\Omega_{3b}$ (Fig.~\ref{fig:R3combination1}-\ref{fig:R3combination1l}). By the oriented parity axiom for the third Reidemeister move we have $-p(u)-p(v)+p(w)=0$, hence, $p(w)=p(u)+p(v)$. Then
\begin{gather*}
\tau(u)=\tau(v)=0 \Longrightarrow p(u),p(v)\ge 0 \Longrightarrow p(w)\ge 0  \Longrightarrow \tau(w)=0,\\
\tau(u)=\tau(v)=1 \Longrightarrow p(u),p(v)< 0 \Longrightarrow p(w)< 0  \Longrightarrow \tau(w)=1.
\end{gather*}
Hence, there cannot be a forbidden configuration in third Reidemeister move, and $\tau$ is an order.
\end{proof}

\subsubsection{Sibling knots}\label{subsect:sibling_knots}

Functorial maps corresponding to orders induce a partial ordering on knots and links.

\begin{definition}\label{def:sibling_knots}
Knots (links, tangles) $L$ and $L'$ are called \emph{kindred} if there is an order $\tau$ on diagrams of $L$ such that the functorial map $f_\tau$ maps the diagrams of $L$ to diagrams of $L'$. The knot $L$ is \emph{elder} and $L'$ is \emph{junior}.

If there is another order $\tau'$ on diagrams of $L'$ such that $f_{\tau'}$ maps the diagrams of $L'$ to diagrams of $L$ then the knots $L$ and $L'$ are called \emph{siblings}.
\end{definition}

\begin{example}
\begin{enumerate}
\item Any knot (link, tangle) $L$ and its mirror $\bar L$ are siblings.
\item Any nontrivial classical long knot is elder to the long unknot (by the order index $o$) but they are not siblings.
\item Any link $L=K_1\cup\cdots\cup K_n$ is elder to the disjoint union of its components $L'=K_1\sqcup\cdots\sqcup K_n$. For example, the Hopf link is elder to the unlink.
\end{enumerate}
\end{example}

\begin{remark}
If $K$ and $K'$ are siblings then their crossing numbers coincide. For virtual knots, siblings have identical virtual genus. In particular, if $K$ is classical then $K'$ is classical.
\end{remark}

\begin{example}
Let us consider the simplest knots from Green's table. Crossing changes of crossings in the minimal diagram of the knot $2.1$ produces the unknot, the knot $2.1$ and its mirror knot. Hence, the knot $2.1$ has no siblings except the mirror knot. The same argument holds for the knots $3.2$, $3.5$, $3.6$, $3.7$.

The knots $3.1$, $3.3$ and $3.4$ are be obtained one from another by crossing change, hence, they can be siblings. Indeed, for example, the functorial map with the scheme $(\id,CC)$ induced by the index parity turns diagrams of the knot $3.3$ to diagrams of the knot $3.4$ with the inverse orientation (Fig.~\ref{fig:knot_33_to_34}).

\begin{figure}[h]
\centering\includegraphics[width=0.4\textwidth]{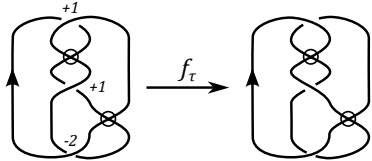}
\caption{The functorial map transforms the knot $3.3$ to the knot $\overline{3.4}$. The labels are the index values $Ind(c)=sgn(c)p(c)$ where $p$ is the index parity}\label{fig:knot_33_to_34}
\end{figure}

In Example~\ref{ex:lifting_flat_knot} below we will show that the knots $3.1$, $3.3$ and $\overline{3.4}$ are siblings.
\end{example}

\begin{remark}
The functorial map with the scheme $(id, CC)$ which corresponds to the index parity was used in the work~\cite{IKL20}.
\end{remark}


\subsection{Lifting of flat knots}\label{subsect:lifting}

\begin{definition}
Let $L$ be an oriented flat knot (link, tangle) and $\bar{\mathscr D}_{\bar L}$ the set of its diagrams. We say that $L$ \emph{can be lifted} if one can set an under-overcrossing structure at the classical crossings of all the diagrams $\bar D\in \bar{\mathscr D}_L$ so that for any flat diagrams $\bar D$ and $\bar D'$ connected by a flat Reidemeister move $\bar\Omega_1$, $\bar\Omega_3$ or $\bar\Omega_3$ the lifted diagrams $D$ and $D'$ are connected by a Reidemeister move ($\Omega_1$, $\Omega_3$ or $\Omega_3$) of the same type.

The result of a lifting of $L$ is called a \emph{flattable} knot.
\end{definition}

For classical knots any attempt of lifting fails. Any classical flat knot is the unknot. Consider the sequence of flat unknot diagrams in Fig.~\ref{fig:unknot_lifting}. Assume this sequence can be lifted. Then after three first Reidemeister moves we have a diagram with three crossings. The signs of two of these three crossing coincide. Remove the third crossing with the first Reidemeister move and get a diagram with two crossing of the same sign. The correspondent flat diagram can be unknotted with a second Reidemeister move. But this move can not be applied to the lifted diagram. Thus, the flat unknot are cannot be lifted.

\begin{figure}[h]
\centering\includegraphics[width=0.6\textwidth]{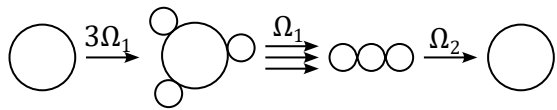}
\caption{A flat move sequence which can not be lifted}\label{fig:unknot_lifting}
\end{figure}

On the other hand, given a diagram of a flat classical long knot, one can lift it to a diagram of the long unknot by assuming that all crossings in the diagram are early undercrossings. Thus, the long unknot can be lifted.

We can describe the lifting procedure as an application of the lifting rule (Fig.~\ref{fig:lifting_rule}). Thus, we have a functorial map associated with a binary trait on flat diagrams.

\begin{figure}[h]
\centering\includegraphics[width=0.4\textwidth]{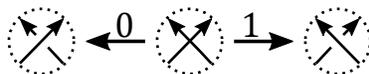}
\caption{Lifting rule}\label{fig:lifting_rule}
\end{figure}

Let $\tau$ be a trait with the scheme $(\skcrr,\skcrl)$. Let us write down the invariance conditions for the lifting rule (Fig.\ref{fig:lifting_conditions}). We will assume that the destination knot theory $\mathcal M'$ includes the Reidemeister moves.

If $\tau$ is not a signed index then we get the clasp move $Cl$. The move $Cl$ with the second Reidemeister move $\Omega_2$ generates the crossing change $CC$. Then the knot theory $\mathcal M'$ is the theory of flat knots, and $\tau$ coincides with the identity map.

Let $\tau$ be a signed index.

\begin{definition}\label{def:flat_order}
A signed index $\tau$ with values in $\{0,1\}$ on a set of flat diagrams $\bar{\mathscr D}$ is called a \emph{flat order} if for any vertices $c_1,c_2,c_3$ participating in a move $\bar\Omega_{3b}$ the combination $\tau(c_1)=\tau(c_3)\ne\tau(c_2)$ can not occur.
\end{definition}

If $\tau$ is not a flat order then the invariance for the third Reidemeister move gives the $\Delta$-move. If $\tau$ is a flat order the Reidemeister moves ensure the invariance.

The invariance for the first Reidemeister move yield requires nothing beyond the Reidemeister moves $\Omega_1$.

We can summarize the reasonings above in the following table.

\begin{center}
\begin{tabular}{|c|c|c|}

\hline
$\mathcal M$ & $\tau$ & $\mathcal M'$\\
\hline
 & trait & $CC,\Omega_1,\Omega_2,\Omega_3$\\
$\mathcal M_{flat}^{+}$  & index & $\Delta,\Omega_1,\Omega_2,\Omega_3$\\
 & order & $\Omega_1,\Omega_2,\Omega_3$\\
\hline
 & trait & $CC,\Omega_2,\Omega_3$\\
$\mathcal M_{flat}^{reg+}$  & index & $\Delta,\Omega_2,\Omega_3$\\
 & order & $\Omega_2,\Omega_3$\\
\hline
\end{tabular}
\end{center}

Thus, we come to the following statement.

\begin{proposition}\label{prop:binary_lifting}
Let $F$ be a surface and $\tau$ a local transformation rule with the scheme $(\skcrr,\skcrl)$ on the oriented flat diagram set $\bar{\mathscr D}(F)$.
\begin{itemize}
\item If $\tau$ is not a signed index then $f_\tau(D)$ is the identity map;
\item If $\tau$ is a signed index then $f_\tau(D)$ can be identified with the extended flat homotopy index polynomial of $D$;
\item If $\tau$ is a flat order then $f_\tau$ is a map in the tangles in the surface $F$.
\end{itemize}
\end{proposition}

We focus on the last case (the lifting problem).
Note that, given a tangle, we can lift its components separately. Long components can be lifted using the order signed index. Thus, one can reduce the lifting problem to lifting of a (closed) knot.

\subsubsection{Flat orders}

Let us describe flat orders on flat knot diagrams in a fixed connected oriented compact surface $F$. Let $\tau$ be a trait on diagrams of a flat oriented knot $K$ in $F$. Denote $H^{+}=\{[D_c]\mid D\in\mathscr D_K(F), c\in\mathcal C(D), \tau(c)=0\}\subset\pi_1(F)$ and $H^-=\pi_1(F)\setminus H^+$. Let $\kappa=[K]\in\pi_1(F)$ be the homotopy class of the knot.

\begin{proposition}\label{prop:flat_order}
\begin{enumerate}
\item The trait $\tau$ is a flat order if and only if for any $t=\pm$ one has $H^tH^t\subset H^t$, $\kappa^{-1}H^tH^t\subset H^t$ and $\kappa(H^t)^{-1}= H^{-t}$.
\item Let $\tau$ be a flat order. Assume that $1\in H^+$. Then the induced left-invariant preordering $\preceq$ is discrete and $\kappa^{-1}$ is a least positive element.
\item Let $\preceq$ be a discrete left-invariant preordering on $\pi_1(F)$ such that $\kappa^{-1}$ is a least positive element. Then the map
\[
\tau(c)=\left\{\begin{array}{cl}
0, & 1\preceq [D_c],\\
1, & [D_c]\prec 1,
\end{array}\right.
\]
where $D_c$ is the left half of the knot at the crossing $c$, is a flat order on the diagrams of the knot.
\end{enumerate}
\end{proposition}

\begin{proof}
  1) The proof of the first statement is analogous to the proof of Proposition~\ref{prop:order_homotopy_index}. The equality $\kappa(H^t)^{-1}= H^{-t}$ follows from the definition of signed index.

  2) Since $1\preceq 1$, $\kappa=\kappa\cdot 1^{-1}\prec 1$. The proof of the second statement of Proposition~\ref{prop:knot_order_and_preordering} shows that the preordering $\preceq$ is discrete and $\kappa^{-1}$ is a least positive element.

  3) Let $\preceq$ be a discrete preordering on $\pi_1(F)$ and $\kappa^{-1}$ a least positive element. Denote $H^+=\{g\in\pi_1(F)\mid 1\preceq g\}$ and $H^-=\{g\in\pi_1(F)\mid g\prec 1\}$. By Lemma~\ref{lem:discrete_preorder_basics}, $H^tH^t\subset H^t$ and $\kappa^{-1}H^tH^t\subset H^t$, $t=\pm$.

  Let $1\preceq g$. Then $g^{-1}\preceq 1$. Since $\kappa\prec 1$ then $\kappa g\prec 1$ by Lemma~\ref{lem:preorder_basics}. If $g\prec 1$ then $1\prec g^{-1}$. Since $\kappa^{-1}$ is a least positive element, $\kappa^{-1}\preceq g^{-1}$. Hence, $1=\kappa\kappa^{-1}\preceq \kappa g^{-1}$.

  The proof of Lemma~\ref{lem:nontrivial_preorder_invariance} shows that the preordering is $K$-invariant. Then the map $\tau$ is a well-defined trait on diagrams of $K$. By the first statement of the proposition $\tau$ is a flat order.
\end{proof}

\begin{lemma}\label{lem:lifting_criterion}
Let $G$ be the surface group of an oriented surface. Then an element $a$ is a least positive element for some discrete preordering on $G$ if and only if $a$ is not a proper power.
\end{lemma}

\begin{proof}

If $a\succ 1$ is a proper power, i.e. $a=b^k$, then $b\succ 1$. Hence, $b^{k-1}\succ 1$ and $1\prec b\prec b\cdot b^{k-1}=a$.

Let $a$ be not a proper power in $G$. Consider the normal closure $H=\langle\langle a\rangle\rangle$ of $a$. Hempel~\cite{Hem} proved that the quotient group $G/H$ is locally indicable, hence, there is a left-ordering $\preceq_1$ on it.

The surface group $G$ is a quotient of a free group $F$ by a relation
\[
w=[a_1,b_1]\cdots[a_n,b_n],\quad n\ge 0.
\]
Let $\tilde H$ be the preimage of $H$ in $F$. The group $\tilde H$ is free and is not generated by $w$, hence, there is an epimorphism $\tilde\phi\colon \tilde H\to\mathbb Z$, such that $\tilde\phi(w)=0$. Then $\tilde\phi$ defines an epimorphism  $\phi\colon H\to\mathbb Z$.

The subgroup $H$ is generated by elements $gag^{-1}$. Denote
\[
d=gcd\{\phi(gag^{-1})\mid g\in G\}.
\]
Since $\phi$ is an epimorphism, $d=1$. Then there exist $g_1,\dots, g_p\in G$ and $l_1,\dots, l_p\in\mathbb Z$ such that $\sum_{i=1}^p l_i\phi(g_iag_i^{-1})=1$. Consider the homomorphism $\phi'=\sum_{i=1}^p l_i\cdot\phi^{g_i}$ from $H$ to $\mathbb Z$ where $\phi^g(h)=\phi(ghg^{-1})$. Then $\phi'(a)=1$.

Now, define a preordering on $G$: $g_1\preceq g_2$ if and only if $g_1H\prec_1 g_2H$ or $g_1H=g_2H$ and $\phi'(g_1^{-1}g_2)\ge 0$. Then $a$ is a least positive element in the preordering $\preceq$.
\end{proof}

Proposition~\ref{prop:flat_order} and Lemma~\ref{lem:lifting_criterion} imply the following statement.
\begin{theorem}\label{thm:liftable_knots_surface}
 A flat knot $K$ in the surface $F$ can be lifted if and only if its homotopy class is not a proper power in $\pi_1(F)$. In particular, the flat unknot cannot be lifted.
\end{theorem}

\begin{remark}
In contrast to Theorem~\ref{thm:liftable_knots_surface}, N. Smythe used orderings of surface groups to show that any diagram of the flat unknot in a surface can be lifted to a diagram of the unknot~\cite{Smythe}.
\end{remark}

\begin{example}\label{ex:lifting_in_torus}
Let $F=T^2$ be the torus. By Theorem~\ref{thm:liftable_knots_surface}, a flat knot $K$ in the torus is liftable if and only if the homology class $\kappa=[K]\in H_1(T^2,\Z)$ is not a multiple of another class. If $\kappa$ is not multiple then there exists a class $\alpha\in H_1(T^2,\Z)$ such that the intersection number $\kappa\cdot\alpha=1$. Then the intersection map $\phi(x)=\alpha\cdot x$, $x\in H_1(T^2,\Z)$, defines a homomorphism from $H_1(T^2,\Z)$ to $\Z$ such that $\phi(\kappa)=-1$. The natural ordering on $\Z$ induces a discrete preordering on $H_1(T^2,\Z)$ such that $-\kappa$ is a least positive element. Thus, the functorial map $f_\tau$ of the binary trait
\[
\tau(c)=\left\{\begin{array}{cl}
0, & \phi(c)\ge 0,\\
1, & \phi(c)<0,
\end{array}\right.
\]
lifts the knot $K$.
\end{example}

\begin{remark}\label{rem:lifting_properties}
1. If a flat knot $\bar K$ lifts to knots $K_1$ and $K_2$ then the knots $K_1$ and $K_2$ are siblings. Indeed, there is a functorial map from diagrams of one knot to the diagrams of the other, given by the composition of the natural projection map from $K_i$ to $\bar K$ and the lifting $\bar K$ to $K_{3-i}$.

2. If a flat knot $\bar K$ lifts to a knot $K$ then the (flat) crossing number $c(\bar K)$ coincides with the crossing number $c(K)$ because there is a bijection between the diagrams of $\bar K$ and the diagrams of $K$.
\end{remark}

We can extend some results for knots in a fixed surface to flat virtual knots.

\begin{example}\label{ex:lifting_flat_knot}
Consider the flat knot with $3$ crossings in Fig.~\ref{fig:flat_knot32_lifting}. Its reduced based matrix $M=(D_c\cdot D_{c'})_{c,c'\in\mathcal C(D)\sqcup *}$~\cite{Nbm,Tvs} is equal to
\[
\left(
\begin{array}{cccc}
0 & 1 & 1 & -2\\
-1 & 0 & 0 & -1\\
-1 & 0 & 0 & -2\\
2 & 1 & 2 & 0
\end{array}
\right).
\]
 The vertices of the diagram form three nonzero primitive tribes~\cite{Nbm}, hence, they determine three invariant homology classes that we denote $D_1$, $D_2$ and $D_3$. Then the functions
 \[
 \phi_1(c)=(D_1+D_3-D)\cdot D_c,\  \phi_2(c)=(3D_1-2D_2+D_3)\cdot D_c, \  \phi_3(c)=(D_1+D_3)\cdot D_c,
 \]
 where $D$ is the homology class of the knot diagram, define traits with values in $\Z$ such that $\phi_i([D])=-1$, $i=1,2,3$. Like in Example~\ref{ex:lifting_in_torus}, the maps
\[
\tau_i(c)=\left\{\begin{array}{cl}
0, & \phi_i(c)\ge 0,\\
1, & \phi_i(c)<0,
\end{array}\right.
\]
$i=1,2,3$, are flat orders on the diagram of the flat knot. Then $f_{\tau_1}$ lifts the flat knot to the knot $3.1$, $f_{\tau_2}$ lifts the knot to $3.3$, and $f_{\tau_3}$ lifts to the knot to $\overline{3.4}$.

\begin{figure}[h]
\centering\includegraphics[width=0.4\textwidth]{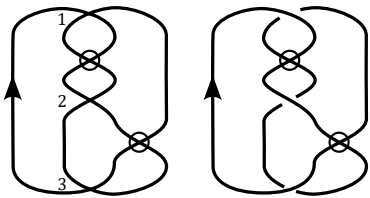}
\caption{A flat knot and its lifting}\label{fig:flat_knot32_lifting}
\end{figure}

Hence, the knots $3.1$, $3.3$ and $3.4$ (and their mirrors) are flattable. Moreover, the knots $3.1$, $3.3$ and $\overline{3.4}$ are siblings because they are liftings of the same flat knot.

On the other hand, by Remark~\ref{rem:lifting_properties}, the virtual knots $2.1, 3.2, 3.5, 3.6, 3.7$ are not flattable.
\end{example}

\section{Open questions}

We conclude the paper with some open questions.

\begin{enumerate}
\item In the paper we have considered a very limited class of functorial map -- binary functorial map. It would be interesting to find meaningful examples of ternary (quaternary, quinary) functorial maps. In most general form a functorial map $f_\tau$ would correspond to a trait
\[
\tau(v)=\sum_{T\in\mathcal T_2}\tau_T(v) T\in A[\mathcal T_2]
\]
where $A$ is a commutative ring. So, nothing prevent us to consider trait with values in series of tangles.

\item The approach exploited in the paper starts with a functorial map scheme and then finds the destination knot theory. We can look from another side and pose the following question: given two knot theories $\mathcal M$ and $\mathcal M'$, describe the functorial maps between them. The case $\mathcal M=\mathcal M'=\mathcal M_{class}$ is of the greatest interest.

\item Papers like~\cite{NOR,NOSY,IM} provide another type of bracket invariants which rely on biquandle colourings. Then one can try to unify these two approaches and define a bracket that would depend both on a crossing trait and an edge colouring.

\item  Find more examples of sibling (more general, kindred) knots. Which invariants do not recognize siblings? For example, must a sibling of a slice knot be slice too?

\item Are there any restriction on invariant values for flattable knots? Find necessary/sufficient condition for a knot to be flattable.
\end{enumerate}

The author is grateful to Yves de Cornulier a.k.a. YCor for drawing attention to ordering of groups.

\bibliographystyle{amsplain}

\appendix

\section{Types of binary traits}\label{app:binary_trait_types}

Given a binary functorial map, the invariance conditions and the required destination knot theory depend much on the possible combinations of trait values of the crossings in Reidemeister moves. The goal of this section is to classify binary traits by the set of those combinations.

Let $\tau$ be a binary trait on a set $\mathscr D$ of oriented diagrams. The knot theory we consider on $\mathscr D$ is either $\mathcal M_{class}^+=\{\Omega_{1a},\Omega_{1b},\Omega_{2c}, \Omega_{2d}, \Omega_{3b}\}$ or $\mathcal M_{class}^{reg+}=\{\Omega_{2a},\Omega_{2b},\Omega_{2c}, \Omega_{2d}, \Omega_{3b}\}$. 

\begin{figure}[h]
\centering\includegraphics[width=0.8\textwidth]{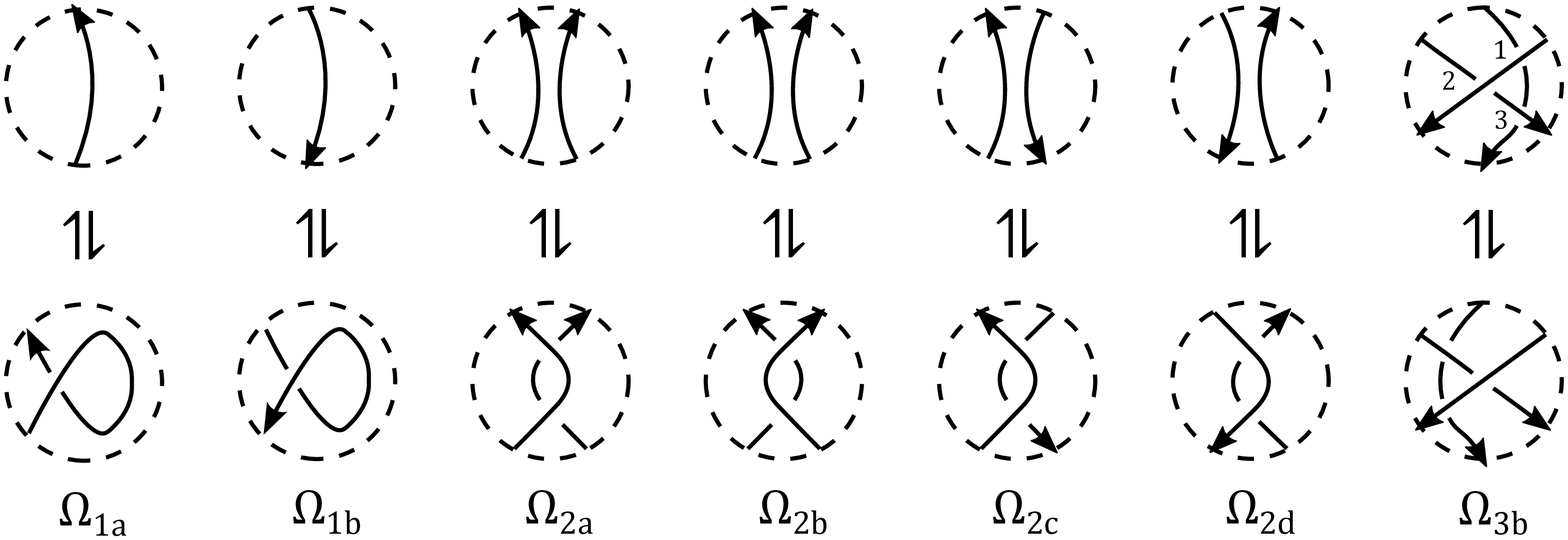}
\caption{Oriented Reidemeister moves}\label{fig:reidemeister_moves_or3}
\end{figure}

\begin{definition}
For a first Reidemeister move on a crossing $v$, its \emph{$\tau$-value} is the number $\tau(v)\in\{0,1\}$. Let $\Pi_{1a}(\tau)\subset\{0,1\}$ be the set of $\tau$-values of all moves $\Omega_{1a}$ on diagrams in $\mathscr D$. Analogously, one defines $\Pi_{1b}(\tau)$.

The \emph{$\tau$-value} of a second Reidemeister move on crossings $v$ and $w$ where $sgn(v)=1$ and $sgn(w)=-1$, is the pair $(\tau(v),\tau(w))$. Let $\Pi_{2a}(\tau)\subset \{0,1\}^2$ be the set of $\tau$-values of $\Omega_{2a}$ moves in $\mathscr D$. Analogously, $\Pi_{2b}(\tau)$, $\Pi_{2c}(\tau)$ and $\Pi_{2d}(\tau)$ are defined.

Finally, for a third Reidemeister move $\Omega_{3b}$ on crossings $v_1$, $v_2$, $v_3$ numbered as in Fig.~\ref{fig:reidemeister_moves_or3}, define its \emph{$\tau$-value} as the triple $(\tau(v_1),\tau(v_2),\tau(v_3))\in\{0,1\}^3$. The set of $\tau$-values of all move $\Omega_{3b}$ in $\mathscr D$ is denoted by $\Pi_{3b}(\tau)\subset \{0,1\}^3$.
\end{definition}

Let $D\in\mathscr D$ be a tangle diagram and $B$ be a disk such that $D$ and $\partial B$ intersect transversely and $B$ does not contain closed components of $D$. Consider deformations of the diagram $D$ with the support in $B$. For a Reidemeister move $\Omega_m$, $m=1a,1c,\dots$, denote the set of $\tau$-values of moves $\Omega_m$ which appear during those deformations by $\Pi_{m}(\tau\mid D,B)$.

The intersection $D\cap B$ is a diagram of an oriented tangle $T$ in $B$. Let $\mathscr D(T)$ be the set of diagrams of the tangle $T$.  The restriction $\tau|_{D\cap B}$ of the trait $\tau$ to the crossings of diagrams of $T$ is a trait on $\mathscr D(T)$.

\begin{proposition}\label{prop:tau_values}
For any Reidemeister move $\Omega_m$, $m=1a,\dots$,
\begin{enumerate}
\item $\Pi_{m}(\tau\mid D,B)=\Pi_{m}(\tau|_{D\cap B})$,
\item $\Pi_{m}(\tau)=\bigcup_{D,B}\Pi_{m}(\tau\mid D,B)$.
\end{enumerate}
\end{proposition}
\begin{proof}
The first statement follows from the definition of the trait $\tau|_{D\cap B}$.

By definition for any diagram $D$ and disk $B$,  $\Pi_{m}(\tau\mid D,B)\subset \Pi_{m}(\tau)$.  On the other hand, any $\tau$-value $\nu\in\Pi_{m}(\tau)$ is realised by a Reidemeister move $D\to D'$. This move occurs in a small disk $B$. Then $\nu\in\Pi_{m}(\tau\mid D,B)$. Thus, $\Pi_{m}(\tau)\subset \Pi_{m}(\tau\mid D,B)$.
\end{proof}

By~\cite{Nct} any trait $\tau$ on the diagrams of a tangle in the disk is a composition $\tau=\psi\circ\theta^u$ where $\theta^u$ is the universal trait which consists of the component index, order index and the sign, and $\psi$ is a map from the universal coefficient set $\Theta^u$ to $\Z_2$. The map $\psi$ (hence, the trait $\tau$) is uniquely determined by two matrices $C^\pm=(c_{ij}^\pm)\in M_k(\Z_2)$ and two vectors $U^\pm=(u^\pm_i)\in\Z_2^k$ where $k$ is the number of tangle components. Here the number $c_{ij}^\epsilon$, $i\ne j$, is the value of $\tau(v)$ on crossings $v$ with the component index $(i,j)$ and the sign $\epsilon$. The value $\tau(v)$ of a self-crossing $v$ of the $i$-th component of the tangle with the sign $sgn(v)=\epsilon$ is equal to $c^\epsilon_{ii}$ when $v$ is an early overcrossing, and  $c^\epsilon_{ii}+u^\epsilon_i$ when $v$ is an early undercrossing. For a trait $\tau$ determined by $C^\pm$ and $U^\pm$, we will write sometimes $\Pi_{m}(C^\pm,U^\pm)$ instead of $\Pi_{m}(\tau)$.

\subsection{First Reidemeister move}

Let $\tau$ be a binary trait. Denote $\Pi_1(\tau)=\Pi_{1a}(\tau)\cup \Pi_{1b}(\tau)$.

The set $\Pi_{1a}(\tau)$ is the set of values of the trait $\tau$ on the loops of type $r_+$ (Fig.~\ref{fig:loop_types}), and  $\Pi_{1b}(\tau)$ is the set of values on the loops of type $l_+$. Since the loop values can be defined independently, we can realize any possible $\tau$-value for the first Reidemeister move.


\begin{proposition}\label{prop:R1_tau_values}
For any nonempty sets $X,Y\subset\{0,1\}$ there exists a binary trait $\tau$ such that $(\Pi_{1a}(\tau), \Pi_{1b}(\tau))=(X,Y)$.
\end{proposition}
\begin{proof}
Let $X=\{x_1\}\cup\{x_2\}$, $Y=\{y_1\}\cup\{y_2\}$, $x_i,y_i\in\{0,1\}$. Consider the matrices
\[
C^+=\left(\begin{array}{cc}
                x_1 & 0 \\
                0 & x_2
              \end{array}\right),\quad
U^+=\left(\begin{array}{cc}
                x_1+y_1 & x_2+y_2
              \end{array}\right)
\]
and choose an arbitrary matrix $C^-$ and vector $U^-$. Let $\tau$ be the trait determined by $C^+,C^-,U^+,U^-$.
Then $(\Pi_{1a}(\tau), \Pi_{1b}(\tau))=(X,Y)$.
\end{proof}

\subsection{Second Reidemeister move}

\begin{proposition}
Let $\tau$ be a binary trait. Then $\Pi_{2a}(\tau)=\Pi_{2b}(\tau)=\Pi_{2c}(\tau)=\Pi_{2d}(\tau)$.
\end{proposition}
\begin{proof}
  Given a crossing, we can include it in a second Reidemeister move of any type (Fig.~\ref{fig:R2_moves_types}). By Proposition~\ref{prop:trait_dual_crossing} the trait values on the crossings $w_a, w_b, w_c, w_d$ coincide. Hence, a $\tau$-values appears in second Reidemeister moves of all types simultaneously.
\end{proof}
We will use notation $\Pi_{2}(\tau)$ for the set of $\tau$-values of second Reidemeister moves.

\begin{figure}[h]
\centering\includegraphics[width=0.4\textwidth]{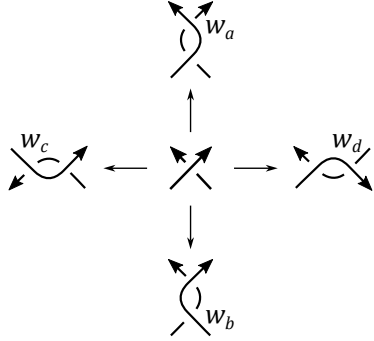}
\caption{Second Reidemeister moves on a crossing}\label{fig:R2_moves_types}
\end{figure}

Like in the case of first Reidemeister moves, one can realize any nonempty subset of $\{0,1\}^2$ as the set $\Pi_{2}(\tau)$. The trait $\tau$ here can be defined on diagrams of a tangle with two components (if the subset contains at most $2$ elements, it is enough to consider a tangle with one component, i.e. a long knot).

Thus, the list of possible cases of the values set $\Pi_{2}(\tau)$ is:
\begin{enumerate}
\item unary trait: $|\Pi_2(\tau)|=1$, i.e. $\Pi_2(\tau)=\{00\}$, or $\{01\}$, or $\{10\}$, or $\{11\}$. In this case, the trait value is determined by the sign of the crossing.
\item semi-binary trait: $\Pi_2(\tau)=\{00,01\}$ or $\{10,11\}$, or $\{00,10\}$, or $\{01,11\}$. In this case, the trait values of all positive crossings or all negative crossings coincide.
\item index: $\Pi_2(\tau)=\{00,11\}$,
\item signed index: $\Pi_2(\tau)=\{01,10\}$,
\item $|\Pi_2(\tau)|=3$, i.e. $\Pi_2(\tau)=\{00,01,10\}$, or $\{00,01,11\}$, or $\{00,10,11\}$, or $\{01,10,11\}$.
\item $\Pi_2(\tau)=\{00,01,10,11\}$.
\end{enumerate}

\subsection{Move $\Omega_{3b}$}

1. Let us assume first that a binary trait $\tau$ is defined on diagrams of a tangle $T$ in the disk $B$. Since the move $\Omega_{3b}$ includes positive crossings only, the trait $\tau$ depends on the matrix $C^+$ and the vector $U^+$. Moreover, let us suppose that $\tau$ does not depend on the order index of crossings, i.e. $U^+=0$. The last assumption is that the tangle $T$ has $k=3$ components.

The possible cases of the set $\Pi_{3b}(\tau)$ are listed in the table below. In the table we consider $C^+$ as the adjacency matrix of a component digraph. An unoriented edge denotes a pair of oppositely oriented edges, a white vertex corresponds to a vertex without loops, and a black vertex denotes a vertex with a loop. A half-black vertex means that the vertex can be of any color. Then for example, the diagram in the case $l$ denotes two digraphs, and the diagram of the case $n$ stands for eight digraphs.

\begin{center}
\renewcommand{\arraystretch}{2}
\begin{longtable}{|c|c|c|}
 \hline
Type & $\Pi_{3b}$ & Component digraph \\
 \hline
 \endfirsthead

 \hline
Type & $\Pi_{3b}$ & Component digraph \\
 \hline
 \endhead

 \hline
 \endfoot

 \hline
 \endlastfoot

a & $000$ & \ptr{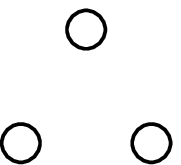} \\ \hline
b & $111$ & \ptr{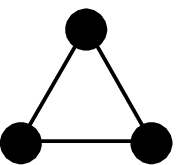} \\ \hline
c & $000, 001, 110, 111$ & \ptr{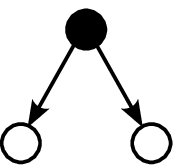},  \ptr{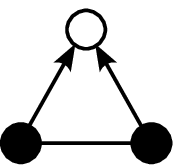}\\ \hline
d & $000, 100, 011, 111$ & \ptr{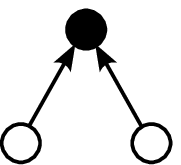},  \ptr{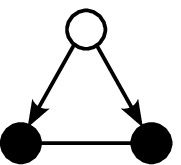}\\ \hline
e & $000, 011, 101, 110$ & \ptr{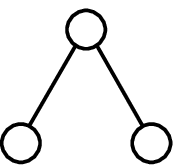}\\ \hline
f & $001, 010, 100, 111$ & \ptr{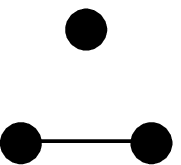}\\ \hline
g & $000, 001, 010, 100, 111$ & \ptr{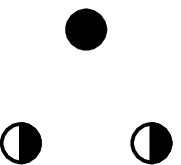}, \ptr{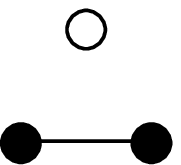}\\ \hline
h & $000, 001, 100, 011, 110$ & \ptr{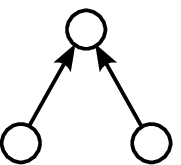}, \ptr{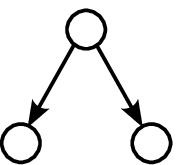}\\ \hline
i & $000, 011, 101, 110, 111$ & \ptr{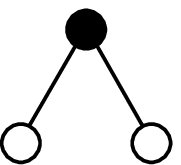}, \ptr{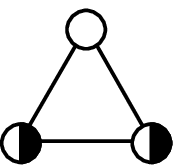}\\ \hline
j & $001, 100, 011, 110, 111$ & \ptr{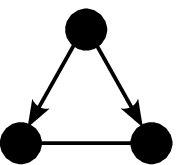}, \ptr{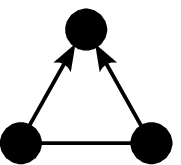}\\ \hline
k & $000, 001, 010, 100, 011, 110$ & \ptr{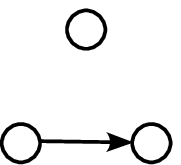}\\ \hline
l & $000, 001, 010, 100, 011, 111$ & \ptr{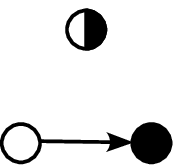}\\ \hline
m & $000, 001, 010, 100, 110, 111$ & \ptr{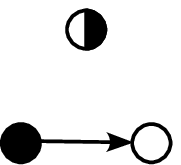}\\ \hline
n & $000, 001, 100, 011, 110, 111$ & \ptr{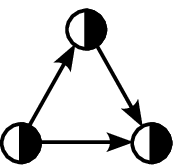}\\ \hline
o & $000, 001, 011, 101, 110, 111$ & \ptr{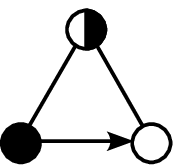}\\ \hline
p & $000, 100, 011, 101, 110, 111$ & \ptr{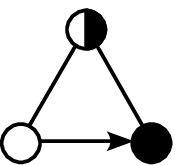}\\ \hline
q & $001, 100, 011, 101, 110, 111$ & \ptr{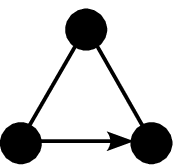}\\ \hline
r & $000, 001, 010, 100, 011, 101, 110$ & \ptr{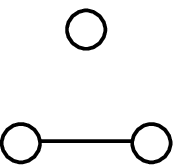}, \ptr{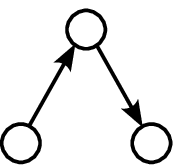}, \ptr{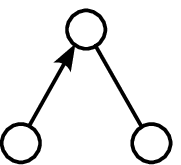}, \ptr{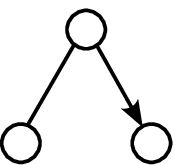}, \ptr{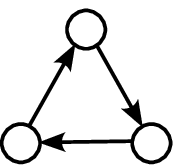}\\ \hline
s & $000, 001, 010, 100, 011, 101, 111$ & \ptr{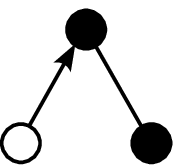}\\ \hline
t & $000, 001, 010, 100, 011, 110, 111$ & \ptr{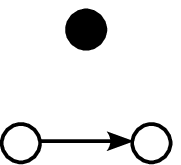}, \ptr{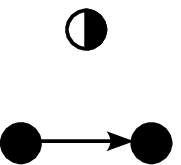}, \ptr{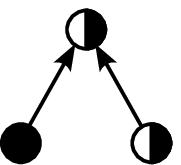}, \ptr{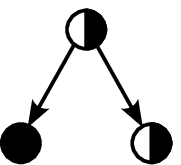}\\ \hline
u & $000, 001, 010, 100, 101, 110, 111$ & \ptr{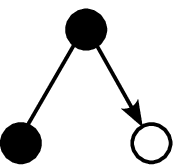}\\ \hline
v & $000, 001, 010, 011, 101, 110, 111$ & \ptr{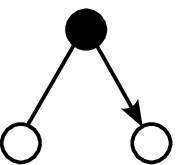}\\ \hline
w & $000, 001, 100, 011, 101, 110, 111$ & \ptr{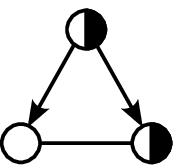}, \ptr{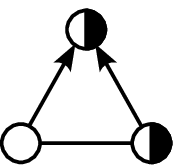}, \ptr{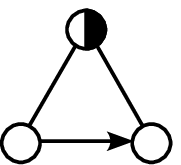}, \ptr{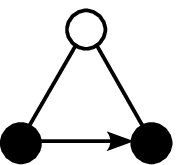}\\ \hline
x & $000, 010, 100, 011, 101, 110, 111$ & \ptr{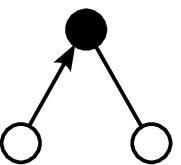}\\ \hline
y & $001, 010, 100, 011, 101, 110, 111$ & \ptr{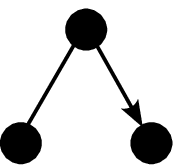}, \ptr{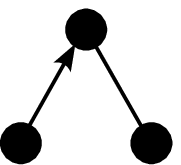}, \ptr{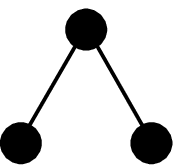}, \ptr{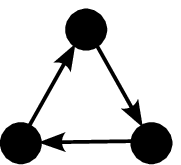}, \ptr{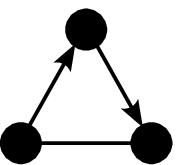}\\ \hline
 &  & \ptr{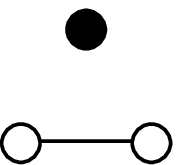}, \ptr{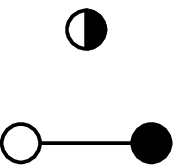}, \ptr{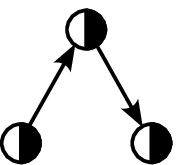} except \ptr{r2}, \\
z & $000, 001, 010, 100, 011, 101, 110, 111$ & \ptr{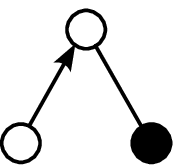}, \ptr{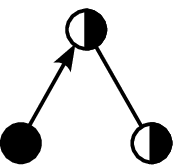}, \ptr{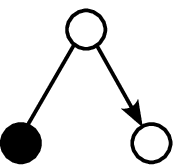}, \ptr{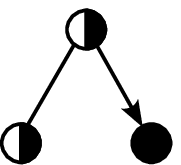}, \ptr{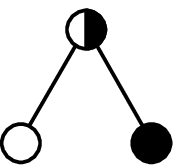}, \ptr{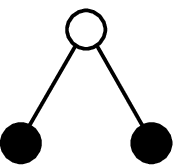}, \\
 &  & \ptr{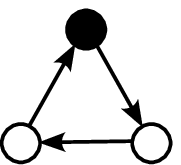}, \ptr{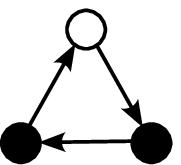}, \ptr{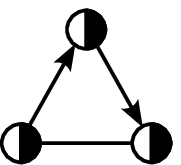} except \ptr{y5}\\
\hline
\end{longtable}
\end{center}

\begin{remark}
1. The traits of types $a$, $b$ can be realized on the diagrams of a $1$-tangle. The traits of types $a$--$j$ can be realized on the diagrams of a $2$-tangle.

2. There are two involutions on the traits. The first one is induced by switching of the trait values $0\leftrightarrow 1$, i.e. a trait $\tau$ becomes the trait $\bar\tau=1-\tau$. The other involution is induced by changing orientation of the disk $B$. By this involution a $\tau$-value $\alpha\beta\gamma\in\{0,1\}^3$ goes to $\gamma\beta\alpha$.

3. There is a partial order on the cases induced by inclusion relation of the sets $\Pi_{3b}$.

4. The type $e$ and its subsets (type $a$) are called \emph{parity}. The type $i$ and its subsets (types $a$, $b$ and $e$) are called \emph{weak parity}. The type $n$ and its subsets (types $a$, $b$, $c$, $d$, $h$ and $j$) are called \emph{order}.
\end{remark}

2. Next, consider a binary trait $\tau$ on diagrams of a tangle $T$ in the disk $B$ which depends on the order index (i.e. $U^+\ne 0$).

Take a component $i$ such that $u^+_i=1$, and consider moves $\Omega_{3b}$ on self-crossings of the component. Different cases of these moves (see Fig.~\ref{fig:R3combination1l}) produces six $\tau$-values, and $\Pi_{3b}(\tau)\supset\{000,001,100,011,110,111\}$. Then $\Pi_{3b}(\tau)$ is of type $n$, $t$, $w$ or $z$. All these type are realizable by the following choice of $C^+$ and $U^+$ correspondingly:
\begin{gather*}
(0), (1);\quad
\left(\begin{array}{cc}
                0 & 0 \\
                0 & 1
              \end{array}\right),
\left(\begin{array}{cc}
                1 & 0
              \end{array}\right);\quad
\left(\begin{array}{cc}
                0 & 1 \\
                1 & 0
              \end{array}\right),
\left(\begin{array}{cc}
                1 & 0
              \end{array}\right);\\
\left(\begin{array}{ccc}
                0 & 0 & 1\\
                0 & 0 & 0\\
                1 & 0 & 0
              \end{array}\right),
\left(\begin{array}{ccc}
                1 & 0 & 0
              \end{array}\right).
\end{gather*}
Thus, the order index does not add new cases of the set $\Pi_{3b}$.

3. Consider the general case. Given a binary trait $\tau$, by Proposition~\ref{prop:tau_values}, the set $\Pi_{3b}(\tau)$ is a union of some sets of types $a$--$z$. Considering of all possible unions leads to three new types
\begin{gather*}
a\cup b=\{000,111\},\qquad e\cup h=\{000,001,100,011,101,110\},\\
f\cup j=\{001,010,100,011,110,111\},
\end{gather*}
which we call \emph{discontinuous types} by the reasons explained below.

Assume that triangles of moves $\Omega_{3b}$ for different $\tau$-values can be gathered by diagram isotopies and Reidemeister moves to one disk $B$. This assumption holds for diagrams of tangles in a connected oriented surface $F$. Then we can realize the $\tau$-values set $\Pi_{3b}(\tau)$ on diagrams of a tangle with $k$ components ($k$ may be grater than $3$) in the disk $B$. Let us show $\Pi_{3b}(\tau)$ can not have a discontinuous type.

Assume that $\Pi_{3b}(\tau)=e\cup h$. Consider the component index graph $G_\tau$ of $\tau$. Since subsets of $\Pi_{3b}(\tau)$ can be of types $a$, $e$ or $b$ only, the full subgraph of $G_\tau$ on any three vertices coincides with the component index subgraph of type $a$, $e$ or $h$. Assume that $G_\tau$ has an oriented edge $\alpha_1\alpha_2$. Then for any component $\beta$, the subgraph on the vertices $\alpha_1, \alpha_2, \beta$ has type $h$. Hence, $G_\tau$ contains is either the oriented edge $\alpha_1\beta$ or $\beta\alpha_2$. Then for any vertices $\beta$, $\gamma$ the full subgraph of $G_\tau$ on the vertices $\alpha_1, \beta, \gamma$ or $\alpha_2, \beta, \gamma$ has type $h$. In any case, there is no unoriented edge $\beta\gamma$ in $G_\tau$. Hence, $G_\tau$ does not contain subgraphs of type $e$. Thus,  $\Pi_{3b}(\tau)$ has type $h$.

Similarly, if $G_\tau$ contains an unoriented edge then $\Pi_{3b}(\tau)$ has type $e$. If $G_\tau$ has no edges then $\Pi_{3b}(\tau)$ has type $a$. Thus, the type $e\cup h$ is not realizable on tangle diagrams in the disk. The proof for the other discontinuous types is analogous.

Note that the discontinuous types can be realized on tangle diagrams in a disconnected surface (for example, a union of two disks).

\begin{remark}
The type of $\Pi_{3b}(\tau)$ imposes some restrictions on the sets $\Pi_1(\tau)$ and $\Pi_2(\tau)$:
\begin{enumerate}
\item The type $a$ is compatible with $\Pi_2(\tau)=\{00\}$, $\{01\}$ or $\{00,01\}$. The type $b$ is compatible with $\Pi_2(\tau)=\{10\}$, $\{11\}$ or $\{10,11\}$. The other types of $\Pi_{3b}(\tau)$ are compatible with any $\Pi_2(\tau)$ which is not constant on positive crossings.
\item The set $\Pi_1(\tau)=\{0\}$ is compatible with the types $a,e,h,i,k,n,r,w,z$; $\Pi_1(\tau)=\{1\}$ is compatible with the types $b,f,g,j,n,q,t,y,z$; $\Pi_1(\tau)=\{0,1\}$ is compatible with the types $c,d,g,i,l,m,n,o,p,s,t,u,v,w,x,z$.
\end{enumerate}
\end{remark}

\section{Skein modules of knots in a fixed surface}\label{app:skein_modules}

This section is devoted to description of skein modules for some knot theories. More precisely, we will focus on sets of knots, i.e. sets of the equivalence classes of diagrams module the moves. Since the local moves we consider are $1$-term moves (like the Reidemeister moves), the skein modules of the knot theories over a ring $A$ will be free modules $A[\mathscr K]$ where $\mathscr K$ is the correspondent set of knots.

\subsection{Smoothing skein modules}

\subsubsection{Rotation number and based index polynomial}\label{subsec:diagram_invariants}

Let $F$ be a connected oriented compact surface. Consider a vector field $\vec v$ on the surface $F$ which has no singular points when $\partial F\ne\emptyset$ or $\chi(F)=0$, and has one singular point $z\in F$ otherwise. Denote $\bar\chi(F)=\chi(F)$ if $\partial F=\emptyset$ and $\bar\chi(F)=0$ otherwise.

\begin{definition}\label{def:rotation_number}
Let $\gamma_t, t\in[a,b]$, be an oriented regular closed curve in $F$ or a curve such that $\gamma\cap\partial F=\partial\gamma$. In the latter case we assume that the tangent vectors $\dot\gamma$ in the boundary points are collinear to the field $\vec v$. For any $t$ denote the angle between $\vec v(\gamma_t)$ and $\dot\gamma_t$ by $\phi(t)$. We can suppose that the function $\phi\colon (a,b)\to\R$ is continuous. Denote $\widetilde{rot}(\gamma)=\frac{\phi(b)-\phi(a)}{2\pi}\in\Z$. The residue $rot(\gamma)=\widetilde{rot}(\gamma)\bmod \bar\chi(F) \in\Z_{\bar\chi(F)}$ is called the \emph{rotation number} of the curve $\gamma$.

For a multicurve $\gamma=\gamma_1\cup\cdots\cup\gamma_k$ we define $rot(\gamma)=\sum_{i=1}^k rot(\gamma_k)$.
\end{definition}

\begin{proposition}\label{prop:rot_properties}
\begin{enumerate}
  \item $rot(\gamma)$ is invariant under regular homotopy;
  \item $rot(-\gamma)=-rot(\gamma)$;
  \item $rot(\bigcirc^\epsilon)=\epsilon$, $\epsilon=\pm$, where $\bigcirc^+$ is a positively (counterclockwise) oriented small circle and $\bigcirc^-$ is a negatively oriented small circle with does not include the singular point $z$;
  \item $rot(\gamma)$ is invariant under $Sm^+$ and $H(2)_+^o$ (Fig.~\ref{fig:unary_or_smoothing}).
\end{enumerate}
\end{proposition}
\begin{proof}
A regular homotopy of $\gamma$ which does not concern the singular point $z$, does not change the difference $\frac{\phi(b)-\phi(a)}{2\pi}$, hence, keeps the rotation number. When the curve $\gamma$ passes over $z$ the difference $\frac{\phi(b)-\phi(a)}{2\pi}$ changes by the index of the singular point which is equal to $\chi(F)$. Hence, the rotation number does not change.

The second and the third statements follow from the definition.

Given a crossing, isotope the curve to make the arcs of the curve tangent in the crossing point. Then the smoothing $Sm^+$ just rearranges the arcs of the curve and does not change the rotation number.

A move $H(2)_+^o$ adds two clockwise half-turns and $\bigcirc^+$. Then the rotation number changes by $2\cdot(-\frac 12)+1=0$.
\end{proof}

\begin{remark}\label{rem:nO}
Below we will use notation $k\bigcirc^+$, $k\in\Z$, for a disjoint union of $k$ positively oriented trivial circles when $k\ge 0$, or a union of $-k$ negatively oriented trivial circles when $k<0$.
\end{remark}

\begin{definition}\label{def:tangle_multiplicity}
For a class $\alpha\in H_1(F,\Z)$, its \emph{multiplicity} is the number
\[
\mu(\alpha)=gcd\{\alpha\cdot\beta\mid\beta\in H_1(F,\Z)\}\in\Z.
\]
Let $D$ be an oriented tangle diagram in $F$. The \emph{multiplicity} of the diagram $D$ is $\mu(D)=\mu([D])$.
\end{definition}
Note that the multiplicity depends on the homology class of the tangle, hence, it is a tangle invariant.

Fix a point $z$ on the surface $F$. We assume that $z$ is the singular point of $\vec v$ (if it exists), and $z\in\partial F$ if $\partial F\ne\emptyset$.
\begin{definition}\label{def:based_index_polynomial}
Let $D$ be an oriented tangle diagram in $F$ such that $z\not\in D$. For any crossing $v\in\mathcal C(D)$ consider a path which connects $z$ with $v$ and approaches to the crossing in the sector incident to the two incoming edges (Fig.~\ref{fig:based_index}). Define the \emph{based index} of the crossing $v$ as the intersection number $ind_z(v)=\gamma_{z,v}\cdot D\in\Z_{\mu(D)}$. The index does not depend on the choice of $\gamma_{z,v}$.

\begin{figure}[h]
\centering\includegraphics[width=0.3\textwidth]{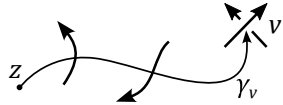}
\caption{The path $\gamma_{z,v}$}\label{fig:based_index}
\end{figure}

The \emph{based index polynomial} is the sum
\[
P_D(t)=\sum_{v\in\mathcal C(D)}sgn(v)\cdot t^{ind_z(v)}\in\Z[t,t^{-1}]/(t^{\mu(D)}-1).
\]
\end{definition}

\begin{proposition}\label{prop:based_index_polynomial}
\begin{enumerate}
\item $P_D(1)=wr(D)$ is the writhe number of the diagram;
\item $P_D(t)$ is invariant under isotopies of $D$ in $F\setminus\{z\}$;
\item if the diagram $D'$ is obtained by passing an arc of $D$ over $z$ then $P_{D'}(t)=t^{\pm 1}\cdot P_D(t)$;
\item $P_D(t)$ is invariant under $\Omega_2$, $H(2)^+$ and $O_1^+$.
\end{enumerate}
\end{proposition}
\begin{proof}
The first three statements follow from the definition and properties of intersection number.

Let $D\to D'$ be a second Reidemeister move on crossings $v$ and $w$. Then $ind_z(v)=ind_z(w)$ and $sgn(v)=-sgn(w)$. Hence
\[
P_{D'}(t)=P_D(t)+sgn(v)t^{ind_z(v)}+sgn(w)t^{ind_z(w)}=P_D(t).
\]

For a move $H(2)^+$ or $O_1^+$, we can choose paths $\gamma_{z,v}$ outside the area of the move. Then the move does not affect $P_D(t)$.
\end{proof}

\begin{corollary}\label{cor:or_invariant_after_smoothing}
Let $D\in\mathscr D_+(F)$, $V\subset \mathcal C(D)$, and $D^V_{or}$ obtained from $D$ by oriented smoothing of all the crossings $v\in V$. Then
$\widetilde{rot}(D^V_{or})=\widetilde{rot}(D)$ and $P_{D^V_{or}}(t)=\sum_{v\in\mathcal C(D)\setminus V}sgn(v)t^{ind_z(v)}$.
\end{corollary}
\begin{proof}
The equality $\widetilde{rot}(D^V_{or})=\widetilde{rot}(D)$ follows from invariance of the rotation number under $Sm^{or}$. The formula for $P_{D^V_{or}}(t)$ follows from invariance of the based index under $Sm^{or}$.
\end{proof}

\subsubsection{Skein modules of oriented diagrams}


\begin{lemma}\label{lem:H2_move}
\begin{enumerate}
\item $H(2)_+\sim \{H(2)_+^o,O_1^+\}$,
\item the move $H(2)^o_+$ can create/annihilate a pair of trivial circles with opposite orientation,
\item the move $H(2)^o_+$ can shift a trivial circle over an arc.
\end{enumerate}
\end{lemma}
\begin{proof}
The move $H(2)_+$ is a composition of the moves $H(2)_+^o$ and $O_1^+$. On the other hand, the move $H(2)_+$ applied to a bend of an arc generates $O_1^+$, and $H(2)_+$ with $O_1^+$ generates $H(2)_+^o$.

For the other statements, see Fig.~\ref{fig:H2o_two_circles} and~\ref{fig:H2o_circle_move}.
\end{proof}
\begin{figure}[h]
\centering\includegraphics[width=0.3\textwidth]{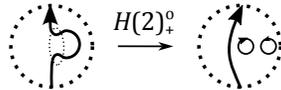}
\caption{The move $H(2)_+^o$ creates a pair of circles}\label{fig:H2o_two_circles}
\end{figure}

\begin{figure}[h]
\centering\includegraphics[width=0.3\textwidth]{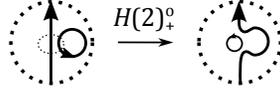}
\caption{Circle shift}\label{fig:H2o_circle_move}
\end{figure}

\begin{proposition}\label{prop:crossingless_skein_modules}
\begin{enumerate}
\item The inclusion $\mathscr D_+^0(F)\hookrightarrow \mathscr D_+(F)$ induces bijections
\begin{gather*}
\mathscr K^0_+(F|H(2)_+)\simeq \mathscr K_+(F\mid\Omega_1,\Omega_2,\Omega_3, Sm^+),\\
\mathscr K^0_+(F|H(2)_+^o)\simeq \mathscr K_+(F\mid\Omega_2,\Omega_3, Sm^+).
\end{gather*}
\item The maps $h\colon \mathscr D_+(F)\to H_1(F,\Z)$, $h(D)=[D]$, and $rot\colon \mathscr D_+(F)\to\Z_{\bar\chi(F)}$ induce bijections
\begin{gather*}
\mathscr K^0_+(F|H(2)_+^o)\simeq H_1(F,\Z)\oplus\Z_{\bar\chi(F)},\\
\mathscr K^0_+(F|H(2)_+)\simeq H_1(F,\Z).
\end{gather*}
\end{enumerate}
\end{proposition}

\begin{proof}
Consider the projection $f_{Sm^+}\colon \mathscr D_+(F)\to\mathscr D_+^0(F)$ which smoothes all crossings according to the orientation of the diagram. If two diagrams $D$ and $D'$ in $\mathscr D_+(F)$ are connected by a second or a third Reidemeister move then $f_{Sm^+}(D)$ and $f_{Sm^+}(D')$ are connected by an isotopy or a move $H(2)_+^o$ (see Section~\ref{subsec:unary_functorial_maps}). On the other hand, $H(2)_+^o$ is a composition of moves $\Omega_2$ and $Sm^+$.
A first Reidemeister move becomes the move $O_1^\pm$ after the projection. 

Thus, the projection $f_{Sm^+}$ induces bijections
\begin{gather*}
\mathscr K^0_+(F|H(2)_+^o)\simeq \mathscr K_+(F\mid\Omega_2,\Omega_3, Sm^+),\\
\mathscr K^0_+(F|H(2)_+)\simeq \mathscr K^0_+(F|H(2)_+^o,O_1^+)\simeq\mathscr K_+(F\mid\Omega_1,\Omega_2,\Omega_3, Sm^+).
\end{gather*}

Consider the map $\Phi=h\oplus rot\colon \mathscr D_+(F)\to H_1(F,\Z)\oplus\Z_{\bar\chi(F)}$. By Proposition~\ref{prop:rot_properties}, $\Phi$ is invariant under $\Omega_2$, $\Omega_3$ and $Sm^+$. Hence, it induces a map
$\Phi\colon\mathscr K_+(F\mid\Omega_2,\Omega_3, Sm^+)\to H_1(F,\Z)\oplus\Z_{\bar\chi(F)}$. The map $\Phi$ is a homomorphism where the group operation on the diagrams is the union: $D_1+D_2=D_1\cup D_2$. Note that $Sm^+$ generates the crossing change $CC$, hence, the union is well defined. The inverse to a diagram $D$ is the diagram with the reversed orientation denoted by $-D$. For any component $D_i\subset D$, the union $D_i\cup -D_i$ can be transformed with on $H(2)_+^o$ move into a pair of trivial circles with opposite orientations. These circles can be eliminated by Lemma~\ref{lem:H2_move}.

We will show that $\Phi$ is an isomorphism. Let us demonstrate first that $\Phi$ is an epimorphism. Let $(\alpha,\rho)\in H_1(F,\Z)\oplus\Z_{\bar\chi(F)}$. Consider a diagram $D$ such that $[D]=\alpha$. Then $\Phi\left(D+(\rho-rot(D))\bigcirc^+\right)=(\alpha,\rho)$.

\begin{lemma}\label{lem:H20_reduction}
Let $D$ be an oriented link diagram in $F$ such that $[D]=0\in H_1(F,\Z)$. Then $D$ is $H(2)_+^o$-equivalent to a union of trivial circles.
\end{lemma}
\begin{proof}
Since $[D]=0$, $D=\bigcup_i\partial\sigma_i$ where $\sigma_i\subset F$ is a compact connected two-dimensional surface. Then $\sigma_i$ is a sphere with handles and holes. We can look at $\sigma_i$ as a polygon with glued edges and holes. The transformation of  $\partial\sigma_i$ into a union of trivial circles is shown in Fig.~\ref{fig:H2o_boundary_reduction}. Note that $\partial\sigma_i$ turns into $\chi(\sigma_i)\bigcirc^+$.

\begin{figure}[h]
\centering\includegraphics[width=0.8\textwidth]{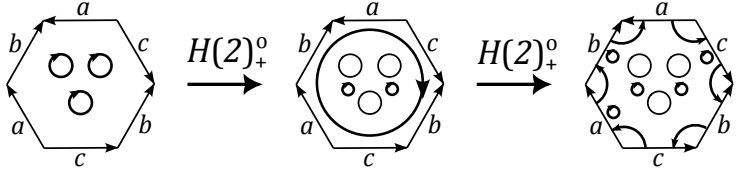}
\caption{Transformation of $\partial\sigma_i$}\label{fig:H2o_boundary_reduction}
\end{figure}
Then $D$ can be transformed into a union of trivial circles.
\end{proof}

Let us show that $\Phi$ is a monomorphism. Let $D_1,D_2\in \mathscr D_+(F)$, $[D_1]=[D_2]$ and $rot(D_1)=rot(D_2)$. Then $D_2=D_1+(D_2-D_1)\sim_{H(2)_+^o}D_1\cup k\bigcirc^+$. Since the move $H(2)_+^o$ does not change the rotation number, $k= rot(D_2-D_1)=rot(D_2)-rot(D_1)=0\pmod{\bar\chi(F)}$. Assume $k\ne 0$. Then $\partial F=\emptyset$ and $k=l\cdot\chi(F)$. Take a negatively oriented trivial circle $\bigcirc^-$ (if there is no such circles then create it with moves $H(2)_+^o$) and transform it like in Fig.~\ref{fig:H2o_boundary_reduction}. Then we get $(\chi(F)-1)$ positive circles. Hence, the algebraic number of trivial circles changes by $(\chi(F)-1)-(-1)=\chi(F)$. Repeating this operation several times we can eliminate all $k=l\chi(F)$ trivial circles.
Hence, $D_2\sim_{H(2)_+^o}D_1$.

Thus, $\Phi$ is an isomorphism.

The move $\Omega_1$ is equivalent to the move $O_1^+$ modulo $Sm^+$. The move $O_1^+$ does not change the homology class but changes the rotation number by $1$. Hence, factorization by $\Omega_1$ eliminates the rotation number that gives the last statement of the proposition.
\end{proof}

Let us pass to skein modules of diagrams with crossings. 

For a given class $\alpha\in H_1(F,\Z)$, consider the ring $R_\alpha=\Z[t,t^{-1}]/(t^{\mu(\alpha)}-1)$ and the set $S_\alpha=\Z\times R_\alpha$. The group $\Z$ acts on $S_\alpha$ by the map $\lambda$ given by the formula
$\lambda(x,f(t))=(x+\chi(F),t^{-1}f(t))$ if $\partial F=\emptyset$, and $\lambda=id$ if $\partial F\ne\emptyset$. Denote $\bar S_\alpha=S_\alpha/\lambda$.

If $\partial F=\emptyset$ and $\chi(F)\ne 0$ then $\bar S_\alpha\simeq\Z_{\chi(F)}\times R_\alpha$. If $\partial F=\emptyset$ and $\chi(F)=0$ then $\bar S_\alpha\simeq\Z\times (R_\alpha)_{cycl}$ where $(R_\alpha)_{cycl}$ is the quotient of $R_\alpha$ by the action $\lambda\colon f(t)\mapsto t^{-1}f(t)$.

Denote $S=\coprod_{\alpha\in H_1(F,\Z)}S_\alpha$ and $\bar S=\coprod_{\alpha\in H_1(F,\Z)}\bar S_\alpha= S/\lambda$.

\begin{proposition}\label{prop:skein_base_module}
The map $\Phi\colon\mathscr D_+(F)\to S$, $\Phi(D)=(\widetilde{rot}(D),P_D(t))\in S_{[D]}$, induces a bijection $\bar\Phi\colon\mathscr K_+(F|H(2)_+^o, \Omega_2)\simeq \bar S$.
\end{proposition}

\begin{proof}
By Propositions~\ref{prop:rot_properties} and~\ref{prop:based_index_polynomial}, the map $\Phi$ is invariant under diagram isotopies in $F\setminus\{z\}$, moves $\Omega_2$ and $H(2)_+^o$. If a diagram $D'$ is obtained from a diagram $D$ by the passing of an upwards oriented arc over the base point $z$ from right to left, then $\widetilde{rot}(D')=\widetilde{rot}(D)+\chi(F)$ and $P_{D'}(t)=t^{-1}P_D(t)$, i.e. $\Phi(D')=\lambda(\Phi(D))$. Hence, $\Phi$ defines a correct map $\bar\Phi\colon\mathscr K_+(F|H(2)_+^o, \Omega_2)\to \bar S$.

For any $\epsilon=\pm 1$ and any $k\in\Z$ consider a diagram $\infty_k^\epsilon$ that consists of a component with one crossing of the sign $\epsilon$ which lies inside $|k|$ concentric circles positively oriented if $k\le 0$ and negatively oriented if $k>0$, and $|k|$ trivial circles which has the opposite orientation to that of the concentric circles (Fig.~\ref{fig:infinity_dressed}).

\begin{figure}[h]
\centering\includegraphics[width=0.2\textwidth]{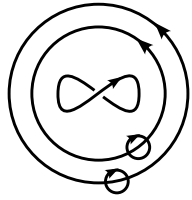}
\caption{The diagram $\infty^{+}_{-2}$}\label{fig:infinity_dressed}
\end{figure}

\begin{lemma}\label{lem:infinity_properties}
\begin{enumerate}
\item $rot(\infty^\epsilon_k)=0$;
\item if the diagram $\infty_k^\epsilon$ is located near the base point $z$, i.e. it is not separated from $z$ by an arc, then $P_{\infty_k^\epsilon}(t)=\epsilon t^k$.
\item $\infty_k^+ \sqcup\infty_k^-\sim_{H(2)_+^o,\Omega_2} \emptyset$;
\item the diagram $\infty_k^\epsilon$ can be shifted over an upwards oriented arc from left to right by moves $H(2)_+^o,\Omega_2$ to the diagram $\infty_{k+1}^\epsilon$.
\end{enumerate}
\end{lemma}
\begin{proof}
The first two statements follows from the definitions of rotation number and based index polynomial.

The proof of the third statement is given in Fig.~\ref{fig:infinity_annihilation} and~\ref{fig:infinity_dressed_annihilate}.
\begin{figure}[h]
\centering\includegraphics[width=0.6\textwidth]{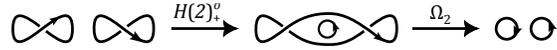}
\caption{Annihilation of the diagrams $\infty^+$ and $\infty^-$}\label{fig:infinity_annihilation}
\end{figure}

\begin{figure}[h]
\centering\includegraphics[width=0.8\textwidth]{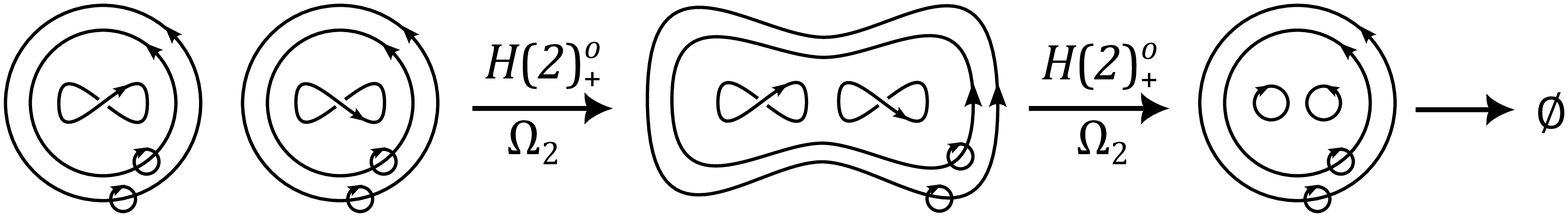}
\caption{Annihilation of the diagrams $\infty_k^+$ and $\infty_k^-$}\label{fig:infinity_dressed_annihilate}
\end{figure}

The proof of the last statement is given in Fig.~\ref{fig:infinity_pass}.
\begin{figure}[h]
\centering\includegraphics[width=0.6\textwidth]{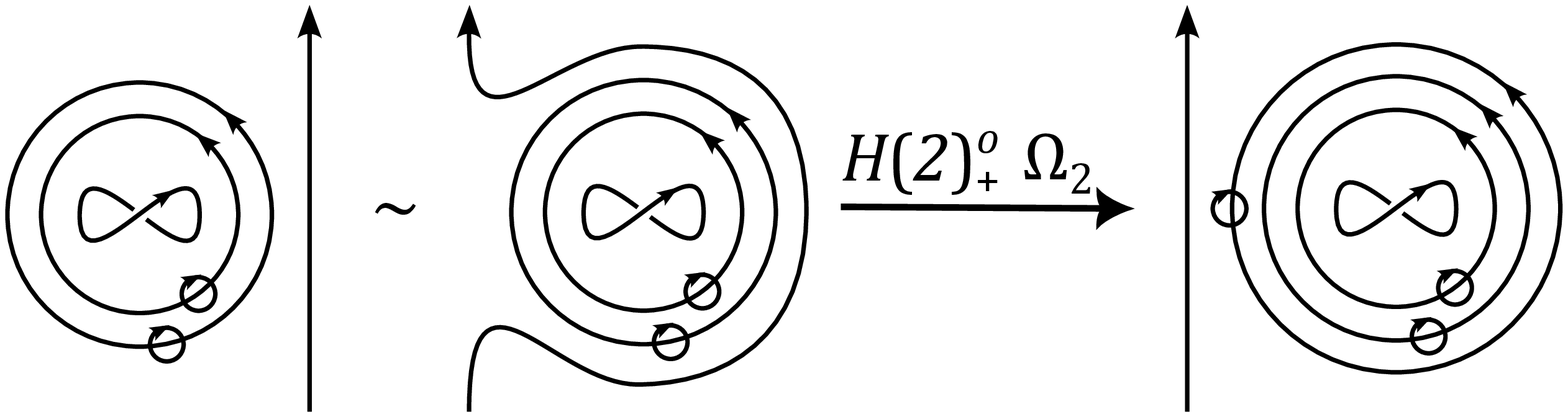}
\caption{Shift of the diagram $\infty^\epsilon_k$}\label{fig:infinity_pass}
\end{figure}
\end{proof}

\begin{remark}
We will use notation $\infty_k$ for $\infty_k^+$ and $-\infty_k$ for $\infty_k^-$.
For a polynomial $f(t)=\sum_l a_lt^l\in\Z[t,t^{-1}]$, the disjoint union $\bigsqcup_{l\in\Z}|a_l|\infty_l^{sgn(a_l)}$ of diagrams $\infty^\epsilon_k$ located near the base point $z$ will be denoted by $\sum_{l\in\Z}a_l\infty_l$ or $f(\infty)$. By Lemma~\ref{lem:infinity_properties}, $P_{f(\infty)}(t)=f(t)$.
\end{remark}

We return to the proof of the proposition. Let us show that $\bar\Phi$ is an epimorphism. Let $\alpha\in H_1(F,\Z)$  and $(x,f(t))\in S_\alpha$. Consider a diagram $D\in\mathscr D^0_+(F)$ such that $[D]=\alpha$. Denote $k=x-\widetilde{rot}(D)$ and $g(t)=f(t)-P_D(t)$. Add $k$ trivial circles $\bigcirc^+$ and the diagram $g(\infty)$ to $D$. Denote the obtained diagram by
\[
D'=D + k\bigcirc^+ + g(\infty).
\]
Then $\widetilde{rot}(D')=\widetilde{rot}(D)+k=x$ and $P_{D'}(t)=P_D(t)+g(t)=f(t)$. Thus, $\Phi(D')=(x,f(t))\in S_\alpha$.

Let us show that $\bar\Phi$ is a monomorphism. Let $D_1$ and $D_2$ be diagrams such that $[D_1]=[D_2]$ and $\bar\Phi(D_1)=\bar\Phi(D_2)\in\bar S_{[D_1]}$.
Apply moves $H(2)_+^o$ near the crossings of $D_1$ in order to replace the crossings with $\infty$-components (Fig.~\ref{fig:crossing_to_infinity}). Then shift $\infty$-components and trivial components to the base point.

\begin{figure}[h]
\centering\includegraphics[width=0.6\textwidth]{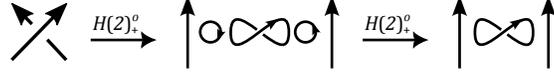}
\caption{Creation of $\infty$-components}\label{fig:crossing_to_infinity}
\end{figure}

Thus, we transform the diagram $D_1$ to a diagram
\[
D_1'=D_1^0+ k_1\bigcirc^+ +f_1(\infty),
\]
where $D_1^0$ is a diagram without crossings, $k_1\in\Z$ and $f_1(t)=\sum\in\Z[t,t^{-1}]$. Since the transformation does not touch the base point, $\widetilde{rot}(D_1)=\widetilde{rot}(D_1')=\widetilde{rot}(D_1^0)+k_1$ and $P_{D_1}(t)=P_{D_1'}(t)=f_1(t)\bmod{(t^{\mu}-1)}$ where $\mu=\mu(D_1)=\mu(D_2)$. Analogously, the diagram $D_2$ is transformed to a diagram
\[
D_2'=D_2^0+ k_2\bigcirc^+ +f_2(\infty),
\]
where $k_2=\widetilde{rot}(D_2)-\widetilde{rot}(D_2^0)$ and $P_{D_2}(t)=f_2(t)\bmod{(t^{\mu}-1)}$.

By Proposition~\ref{prop:crossingless_skein_modules}, $D_1^0$ is $H(2)_+^o$-equivalent to the diagram $D_2^0+k\bigcirc^+$ where  $k=\widetilde{rot}(D_1^0)-\widetilde{rot}(D_2^0)+r\cdot\bar\chi(F)$ and $r$ is the number of times the diagram $D_1^0$ passes over the base point during the transformation. Note that $r=0$ is $\partial F\ne\emptyset$.

Then $D_1'$ is $\{H(2)_+^o,\Omega_2\}$-equivalent to the diagram
\[
D=D_2^0+(k+k_1)\bigcirc^+ +\lambda^r(f_1)(\infty).
\]
Since $\bar\Phi(D_1)=\bar\Phi(D_2)$, there exists $s\in\Z$ such that $\Phi(D_1)=\lambda^s(\Phi(D_2))$, i.e. $\widetilde{rot}(D_1)=\widetilde{rot}(D_2)+s\cdot\bar\chi(F)$ and $P_{D_1}(t)=t^{-s}P_{D_2}(t)$. Then
\[
k+k_1=\widetilde{rot}(D_1)-\widetilde{rot}(D_2)+k_2+r\cdot\bar\chi(F)=k_2+(r+s)\bar\chi(F)
\]
and
\[
t^{-r}f_1(t)\equiv t^{-r}P_{D_1}(t)=t^{-r-s}P_{D_2}(t)=t^{-(r+s)}f_2(t)\bmod(t^{\mu}-1).
\]
Using moves $\Omega_2$ and $H(2)_+^o$, modify the diagram $D$ as follows. Take a trivial circle $\bigcirc^+$, move it so that the base point becomes its center and then apply the transformation as in Fig.~\ref{fig:H2o_boundary_reduction}. Then the number of trivial circles decrease by $\chi(F)$ and the subdiagram $\lambda^r(f_1)(\infty)$ turns into $\lambda^{r-1}(f_1)(\infty)$. Repeat this operation $r+s$ times. Then we get the diagram
\[
D'=D_2^0+k'\bigcirc^+ +g(\infty),
\]
where $k'=(k+k_1)-(r+s)\chi(F)=k_2$ and $g(t)=t^{r+s}(t^{-r}f_1(t))\equiv f_2(t)\bmod{t^{\mu}-1}$.

If $g(t)-f_2(t)\ne 0\in\Z[t,t^{-1}]$ then $g(t)-f_2(t)=(t^\mu-1)\sum_l a_l t^l$. Let $\gamma\subset F$ be a closed path which starts near the base point $z$ such that $D'\cdot\gamma=\mu$. Consider the following transformation of the diagram $D'$. Take the subdiagram $\infty^\epsilon_k$ in $D'$ (if there is no such a subdiagram create a pair $\infty^\epsilon_k$ and $\infty^{-\epsilon}_k$ using moves $\Omega_2$ and $H(2)_+^o$). Pull $\infty^\epsilon_k$ along the path $\gamma$. By Lemma~\ref{lem:infinity_properties}, $\infty^\epsilon_k$ becomes $\infty^\epsilon_{k+\mu}$. Then the part $g(\infty)$ of $D'$ becomes $g_1(\infty)$ where $g_1(t)=g(t)+\epsilon t^k(t^\mu-1)$. Applying this transformation $|a_l|$ times to $\infty_l^{-sgn(a_l)}$, $l\in\Z$, we get the diagram
\[
D''=D_2^0+k'\bigcirc^+ +h(\infty),
\]
where $h(t)=g(t)-(t^\mu-1)\sum_l a_l t^l=f_2(t)$. Hence, the diagram $D''$ is isotopic to $D'_2$. Thus, we have a sequence of $\{\Omega_2, H(2)_+^o\}$-equivalences
\[
D_1\sim D_1'\sim D\sim D'\sim D''=D_2'\sim D_2.
\]
The proposition is proved.
\end{proof}

For a class $\alpha\in H_1(F,\Z)$ denote $R_\alpha^{\Z_2}=\Z_2\otimes R_\alpha=\Z_2[t,t^{-1}]/(t^{\mu(\alpha)}-1)$ and $(R_\alpha^{\Z_2})_{cycl}=\Z_2\otimes (R_\alpha)_{cycl}$.
Let $S^{\Z_2}_\alpha=\Z\times R_\alpha^{\Z_2}$, $\bar S^{\Z_2}_\alpha=S^{\Z_2}_\alpha/\lambda$, $S^{\Z_2}=\coprod_{\alpha\in H_1(F,\Z)}S_\alpha^{\Z_2}$, and $\bar S^{\Z_2}=\coprod_{\alpha\in H_1(F,\Z)}\bar S_\alpha^{\Z_2}$.

\begin{corollary}\label{cor:smoothing_skein modules_or}
There are bijections:
\begin{enumerate}
\item $\mathscr K_+(F | H(2)_+^o,\Omega_2,\Omega_3)\simeq H_1(F,\Z)\times\Z_{\bar\chi(F)}\times\Z$ induced by the map $D\mapsto ([D],rot(D),wr(D))$,
\item $\mathscr K_+(F | H(2)_+,\Omega_2,\Omega_3)\simeq H_1(F,\Z)\times\Z$ induced by the map $D\mapsto ([D],wr(D))$,
\item $\mathscr K_+(F | H(2)_+^o,\Omega_1,\Omega_2,\Omega_3)\simeq H_1(F,\Z)\times\Z_2$ induced by the map $D\mapsto ([D],rot(D)+wr(D)\bmod 2)$,
\item $\mathscr K_+(F | H(2)_+,\Omega_1,\Omega_2,\Omega_3)\simeq H_1(F,\Z)$ induced by the map $D\mapsto [D]$,
\item $\mathscr K_+(F | H(2)_+^o,\Omega_2,2\Omega_\infty)\simeq \bar S^{\Z_2}\times\Z$ induced by the map
\[
D\mapsto (\widetilde{rot}(D),P_D(t)\bmod 2, \lfloor \frac{wr(D)}2\rfloor)\in S_{[D]}^{\Z_2}\times\Z,
\]
\item $\mathscr K_+(F | H(2)_+,\Omega_2,2\Omega_\infty)\simeq \coprod_{\alpha\in H_1(F,\Z)}(R_\alpha^{\Z_2})_{cycl}\times\Z$ induced by the map
\[
D\mapsto (P_D(t)\bmod 2, \lfloor \frac{wr(D)}2\rfloor)\in (R_\alpha^{\Z_2})_{cycl}\times\Z.
\]
\end{enumerate}
\end{corollary}

\begin{proof}
  We need to describe how additional moves $O_1^+$, $\Omega_1$, $\Omega_3$ and $2\Omega_\infty$ affect the rotational number, the based index  polynomial and the subdiagrams $\bigcirc^\pm$ and $\infty_k^\pm$.

  The move $O_1^+$ adds relations $\bigcirc^\pm=\emptyset$ and changes the rotation number by $1$. The move does not affect $\infty_k^\pm$ and the based index polynomial.

  The move $\Omega_1$ adds relations $\infty_0^\epsilon=\bigcirc^{\epsilon'}$, $\epsilon,\epsilon'=\pm$. With moves $H(2)_+^o$ and $\Omega_2$ we get relations $\infty_k^\pm=\bigcirc^+$. Then we can reduce any diagram $k\bigcirc^+ +f(\infty)$ to $a\bigcirc^+$, $a=k+f(1)\bmod 2$. The map $(k,f(t))\mapsto k+f(1)\bmod 2$ is invariant under the map $\lambda$, $\lambda(k,f(t))=(k+\bar\chi(F),t^{-1}f(t))$ because $\bar\chi(F)$ is even.

  The move $\Omega_3$ is equivalent to the move $\Omega_\infty$ (Fig.~\ref{fig:R_infinity_move}) modulo the moves $H(2)_+^o$ and $\Omega_2$. Hence, $\Omega_3$ generates the relations $\infty_k^\epsilon=\infty_{k+1}^\epsilon$. Then any diagram $f(\infty)$, $f(t)\in\Z[t,t^{-1}]$, can be transformed to $f(1)\infty_0^+$. The based index polynomial $P_D(t)$ reduces to the number $P_D(1)=wr(D)$. The move $\Omega_3$ does not affect the rotation number.

  \begin{figure}[h]
\centering\includegraphics[width=0.25\textwidth]{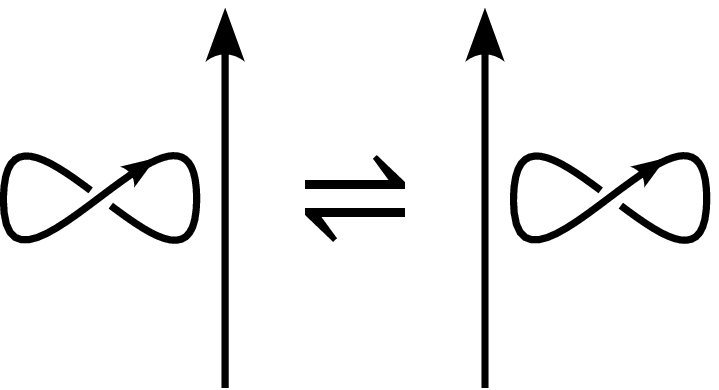}
\caption{The move $\Omega_\infty$}\label{fig:R_infinity_move}
\end{figure}

  The move $2\Omega_\infty$ (with moves $H(2)_+^o$ and $\Omega_2$) generates relations $2\infty_k^\epsilon=2\infty_{k+1}^\epsilon$. Let $f(t)\in\Z[t,t^{-1}]$, and $g(t)\equiv f(t)\in\Z_2[t,t^{-1}]$ is its $\Z_2$-reduction. We look at $g(t)$ as a polynomial with coefficients in $\{0,1\}$. Then $f(\infty)$ is $\{H(2)_+^o,\Omega_2, 2\Omega_\infty\}$-equivalent to $g(\infty)+k\infty_0^+$ where $k=f(1)-g(1)=2\lfloor\frac{f(1)}2\rfloor-2\lfloor\frac{g(1)}2\rfloor$ since $k$ is even. Thus, the equivalence class of $f(\infty)$ can be restored from $g(t)$ and $\lfloor\frac{f(1)}2\rfloor$. The move $2\Omega_\infty$ does not affect the rotation number.

  Combining these relations with the result of Proposition~\ref{prop:skein_base_module}, we get the statements of the theorem.
\end{proof}

\subsubsection{Skein modules of unoriented diagrams}

\begin{lemma}\label{lem:H20=H3}
$H(2)^o\sim H(3)$.
\end{lemma}
\begin{proof}
See Fig.~\ref{fig:H2o_to_H3} and Fig.~\ref{fig:H3_to_H2o}.
\end{proof}
  \begin{figure}[h]
\centering\includegraphics[width=0.6\textwidth]{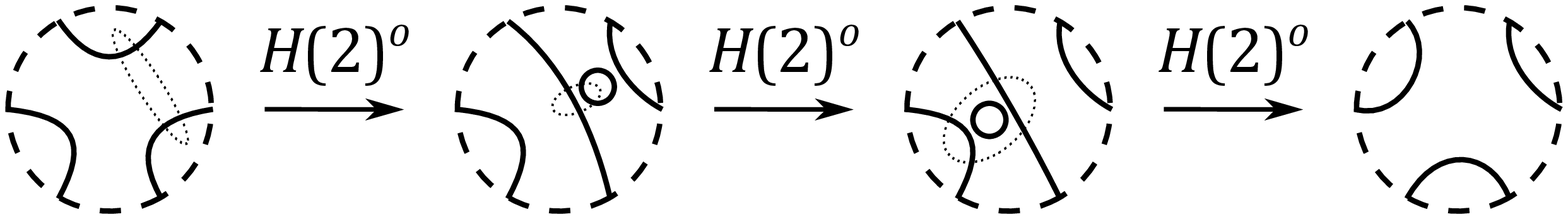}
\caption{$H(2)^o$ generates $H(3)$}\label{fig:H2o_to_H3}
\end{figure}

\begin{figure}[h]
\centering\includegraphics[width=0.3\textwidth]{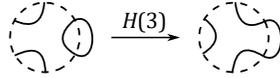}
\caption{$H(3)$ generates $H(2)^o$}\label{fig:H3_to_H2o}
\end{figure}


Fix a section $s\colon H_1(F,\Z_2)\to H_1(F,\Z)$ to the natural projection $H_1(F,\Z)\to H_1(F,\Z_2)$.

\begin{definition}\label{def:rho0_invariant}
Let $D\in\mathscr D^0(F)$. Consider any orientation $\tilde D\in\mathscr D_+^0(F)$ of the diagram $D$. Define the \emph{offset} of the diagram by the formula
\[
\rho_0(D)=\frac{s([D])\cdot\tilde D}2+m(D)\bmod 2\in\Z_2,
\]
where $m(D)$ is the number of components in $D$.
\end{definition}

\begin{proposition}\label{prop:rho0_properties}
\begin{enumerate}
\item The map $\rho_0\colon\mathscr D^0(F)\to\Z_2$ is well-defined;
\item $\rho_0$ is invariant under $H(2)^o$;
\item the moves $O_1$ and $H(2)$ change $\rho_0$ by $1$.
\end{enumerate}
\end{proposition}

\begin{proof}
  Let $D\in\mathscr D^0(F)$. Consider orientations $\tilde D$ and $\tilde D'$ of $D$ which differ by the orientation of a component $D_i$ in $\tilde D$. Then $[\tilde D']=[\tilde D]-2[D_i]\in H_1(F,\Z)$. Since $D_i$ does not intersect other components, $\tilde D\cdot D_i=0$. Then $s([D])\cdot D_i\equiv \tilde D\cdot D_i\equiv 0\bmod 2$. Hence, $s([D])\cdot D_i$ is even and $\frac{s([D])\cdot\tilde D'}2=\frac{s([D])\cdot\tilde D}2-s([D])\cdot D_i\equiv \frac{s([D])\cdot\tilde D}2\bmod 2$. Thus, the number $\rho_0(D)$ does not depend on the choice of orientation, therefore, it is well-defined.

Let $D\to D'$ be a $H(2)^o$-move. If the move can be lifted to the oriented move $H(2)_+^o\colon\tilde D\to\tilde D'$ then $m(D')=m(D)$ or $m(D')=m(D)+2$ and $[\tilde D]=[\tilde D']\in H_1(F,\Z)$. Since $[D]=[D']\in H_1(F,\Z_2)$, $s([D])=s([D'])$ and
\[
 \rho_0(D)=\frac{s([D])\cdot\tilde D}2+m(D)\equiv\frac{s([D'])\cdot\tilde D'}2+m(D')=\rho_0(D').
\]

If the move can not be oriented then it merge a component $D_i$ with itself (Fig.~\ref{fig:H2o_unoriented}). Let $[D_i]=\gamma$ and $\delta$ be the loop which the component $D_i$ is glued along. Then $D\cdot\delta=D_i\cdot\delta=1$ and $s([D])\cdot\delta\equiv 1$. The $D_i$ transforms to a component $D_i'$ whose homology class is $\gamma+2\delta$. Then $\frac{s([D])\cdot\tilde D'}2-\frac{s([D])\cdot\tilde D}2=s([D])\cdot\delta\equiv 1$. On the other hand, $m(D')=m(D)+1$. Hence, $\rho_0(D)=\rho_0(D')$.

\begin{figure}[h]
\centering\includegraphics[width=0.8\textwidth]{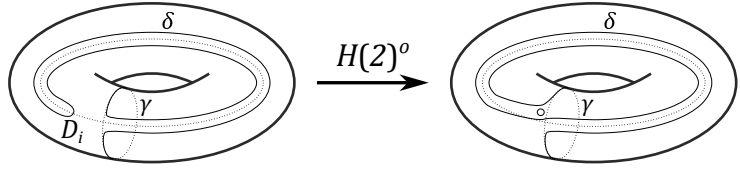}
\caption{Unorientable move $H(2)^o$}\label{fig:H2o_unoriented}
\end{figure}

Let $D\to D'$ be a move $O_1$. The move adds a trivial component. Then $\rho_0(D')=\rho_0(D)+1$.

The move $H(2)$ is a composition of moves $H(2)^o$ and $O_1$ then it changes $\rho_0$ by $1$.
\end{proof}

\begin{proposition}\label{prop:crossingless_skein_unor_modules}
\begin{enumerate}
\item The map $h\oplus\rho_0\colon \mathscr D^0(F)\to H_1(F,\Z_2)\oplus\Z_2$, $D\mapsto ([D],\rho_0(D))$, induces bijections
\begin{gather*}
\mathscr K^0(F|H(2)^o)\simeq H_1(F,\Z_2)\oplus\Z_2,\\
\mathscr K^0(F|H(2))\simeq H_1(F,\Z_2).
\end{gather*}
\end{enumerate}
\end{proposition}

\begin{proof}
By Proposition~\ref{prop:rho0_properties}, the map $\mathscr K^0(F|H(2)^o)\simeq H_1(F,\Z_2)\oplus\Z_2$ is well-defined.

For any $\alpha\in H_1(F,\Z_2)$ there exists $D\in\mathscr D^0(F)$ such that $[D]=\alpha$. Then for any $x\in\Z_2$
\[
(h\oplus\rho_0)(D\sqcup (x-\rho_0(D))\bigcirc)=(\alpha,x).
\]
Thus, the induced map is an epimorphism.

Let $D,D'\in\mathscr D^0(F)$ be diagrams such that $[D]=[D']$ and $\rho_0(D)=\rho_0(D)$. Since $[D]=[D']$, the diagram $D$ can be transformed by moves $H(2)^o$ to a diagram $D'\sqcup k\bigcirc$. Since $\rho_0(D)=\rho_0(D')$, the number $k$ is even. The unoriented analogue of the second statement of Lemma~\ref{lem:H2_move} allows on to contract a pair of trivial circles. Hence, $k$ circles can be annihilated by moves $H(2)^o$. Hence, $D\sim_{H(2)^o} D'\sqcup k\bigcirc \sim_{H(2)^o} D'$. Thus, $\mathscr K^0(F|H(2)^o)\simeq H_1(F,\Z_2)\oplus\Z_2$.

Since $H(2)\sim\{H(2)^o,O_1\}$ and $O_1$ changes $\rho_0$, $\mathscr K^0(F|H(2))\simeq H_1(F,\Z_2)$.
\end{proof}

Let us consider unoriented diagrams with crossings.
Fix a point $z\in F$ such that $z\in\partial F$ if $\partial F\ne\emptyset$. While considering tangle diagrams in the surface $F$, we will assume that they don't contain $z$.

\begin{definition}\label{def:based_index_unor}
Let $D\in\mathscr D(F)$ and $v\in\mathcal C(D)$. Choose a path $\gamma_{z,v}$ connecting $z$ with $v$ which approaches to $v$ in a sector from an overcrossing to an undercrossing (in the counterclockwise order).  The \emph{unoriented based index} of the crossing $v$ is the number $ind^{un}_z(v)=D\cdot\gamma_{z,v}\in\Z_2$.

Let $n_{odd}(D)$ be the number of the odd crossings $v$ in $D$ (such that $ind^{un}_z(v)=1$), and $n_{even}(D)$ the number of the even crossings $v$ in $D$ (such that $ind^{un}_z(v)=0$). Let $wr_{odd}(D)=n_{even}(D)-n_{odd}(D)\in\Z$.

Let $D\in\mathscr D(F)$. Apply smoothings $Sm^A$ to all crossings of $D$ to get a diagram $D_A\in\mathscr D^0(F)$. Define the \emph{offset} $\rho(D)\in\Z_2$ of the diagram by the formula
\[
\rho(D)=\rho_0(D_A)+n_{odd}(D).
\]
\end{definition}
Note that for any $D\in\mathscr D_0(F)$ $\rho(D)=\rho_0(D)$.

\begin{proposition}\label{prop:sm_unoriented_invariants}
\begin{enumerate}
\item Let $D\in\mathscr D(F)$ and $v\in\mathcal C(D)$. Let $\tilde D\in \mathscr D_+(F)$ be an orientation of $D$. Then $ind^{un}_z(v)=ind_z(v)+\frac{sgn(v)-1}2\bmod 2$ where $ind_z(v)$ is the index of $v$ in the diagram $\tilde D$.
\item $wr_{odd}$ is invariant under the moves $\Omega_2$ and $H(2)$.
\item $\rho$ is invariant under the moves $\Omega_2$ and $H(2)^o$.
\item If a diagram $D'$ is obtained from a diagram $D$ by an arc passing over the base point $z$ then
\begin{gather*}
n_{odd}(D')=n_{even}(D),\quad n_{even}(D')=n_{odd}(D),\\
 wr_{odd}(D')=-wr_{odd}(D),\quad \rho(D')=\rho(D)+wr_{odd}(D).
\end{gather*}
\item Let $D=D_1\sqcup D_2$, $D_2$ contractible and any subdiagram $D_i$ do not separate the other subdiagram from the base point. Then $\rho(D)=\rho(D_1)+\rho(D_2)$ and $wr_{odd}(D)=wr_{odd}(D_1)+wr_{odd}(D_2)$.
\item For any $k\in\Z$ and $\epsilon=\pm$, $\rho(\infty_k^\epsilon)=k\bmod 2$ and $wr_{odd}(\infty_k^\epsilon)=(-1)^k\epsilon$.
\item For any $f(t)\in\Z[t,t^{-1}]$, $\rho(f(\infty))=f'(1)$  and $wr_{odd}(f(\infty))=f(-1)$.
\end{enumerate}
\end{proposition}
\begin{proof}
The first statement follow from the definition of (un)oriented based index.

The move $H(2)$ does not change the set of crossings, hence, $n_{odd}$, $n_{even}$ and their difference are invariant. The move $\Omega$ adds one odd and one even crossings, hence, the difference $n_{even}-n_{odd}$ is invariant.

Let $D\to D'$ be a move $\Omega_2$. Then $D_A$ and $D'_A$ are connected by a move $H(2)$ (see Section~\ref{subsect:A-smoothing}), so $\rho_0(D'_A)=\rho_0(D_A)+1$. On the other hand $n_{odd}(D')\equiv n_{odd}(D)+1\bmod 2$. Thus, $\rho(D')=\rho(D)$.

Let $D\to D'$ be a move $H(2)^o$. Then $D_A$ and $D'_A$ are connected by a move $H(2)^o$. Hence, $\rho_0(D_A)=\rho_0(D'_A)$. On the other hand, $n_{odd}(D)=n_{odd}(D')$. Thus, $\rho(D)=\rho(D')$.

Let $D'$ is obtained by moving an arc of a diagram $D$ over $z$. By the definition of unoriented index, the indices of all vertices in the diagram change. Then $n_{odd}(D')=n_{even}(D),\quad n_{even}(D')=n_{odd}(D)$. The diagram $D_A$ and $D'_A$ are isotopic, hence, $\rho_0(D')=\rho_0(D)$. Thus,
\[
\rho(D')=\rho_0(D')+n_{odd}(D')=\rho_0(D)+n_{even}(D)=\rho(D)+wr_{odd}(D).
\]

Let $D=D_1\sqcup D_2$. Since $D_2$ is contractible, $s([D])=s([D_1])$ and $[\widetilde{D_A}]=[\widetilde{(D_1)_A}]\in H(F,\Z)$, hence, $\rho_0(D_A)=\rho_0((D_1)_A)$. Then $\rho(D)=\rho(D_1)+\rho(D_2)$ because $n_{odd}(D)=n_{odd}(D_1)+n_{odd}(D_2)=n_{odd}(D_1)+\rho(D_2)$. The other equality is proved analogously.

Consider the diagram $\infty^\epsilon_k$. Eliminate all crossings except the central one with moves $\Omega_2$ and get a diagram $D$. The unoriented based index of the central crossing is $k+\frac{1-\epsilon}2$. Then $wr_{odd}(\infty_k^\epsilon)=wr_{odd}(D)=(-1)^k\epsilon$. The diagram $D_A$ contains $2k+1+\frac{1+\epsilon}2$ components. Then
\[
\rho(\infty_k^\epsilon)=\rho(D)=2k+1+\frac{1+\epsilon}2+\frac{1-(-1)^k\epsilon}2\equiv 1+\frac{1+\epsilon}2+k+\frac{1-\epsilon}2\equiv k\bmod 2.
\]

Let $f(t)=\epsilon t^k$. Then $f(\infty)=\infty^\epsilon_k$. By the previous statement, $\rho(f(\infty))\equiv k\equiv f'(1)\bmod 2$ and $wr_{odd}(f(\infty))=(-1)^k\epsilon=f(-1)$. The general case follows from the fifth statement of the proposition.
\end{proof}

\begin{corollary}\label{cor:unor_invariants_after_smoothing}
Let $D\in\mathscr D_+(F)$ and $V\subset\mathcal C(D)$. Denote $P_D^V(t)=\sum_{v\in V}sgn(v)t^{ind_z(v)}$ and $V_-=\{v\in V\mid sgn(v)=-1\}$.
\begin{enumerate}
\item Let $D^V_{or}$ be obtained from $D$ by oriented smoothings of all crossings $v\in V$. Then $\rho(D^V_{or})=\rho(D)-(P_D^V)'(1)$, $wr_{odd}(D^V_{or})=wr_{odd}(D)-P_D^V(-1)$.
\item Let $D^V_{unor}$ be obtained from $D$ by unoriented smoothings of all crossings $v\in V$. Then $\rho(D^V_{unor})=\rho(D)-(P_D^V)'(1)+|V|$, $wr_{odd}(D^V_{unor})=wr_{odd}(D)-P_D^V(-1)$.
\item Let $D^V_{A}$ be obtained from $D$ by $A$-smoothings of all crossings $v\in V$. Then $\rho(D^V_{A})=\rho(D)-(P_D^V)'(1)+|V_-|$, $wr_{odd}(D^V_{A})=wr_{odd}(D)-P_D^V(-1)$.
\end{enumerate}
\end{corollary}

\begin{proof}
The diagram $D$ can be transformed to a diagram $D''=D^V_{or}\sqcup f(\infty)$ by moves $H(2)_+^o$. Then $\rho(D)=\rho(D'')=\rho(D^V_{or})+\rho(f(\infty))$ and $wr_{odd}(D)=wr_{odd}(D^V_{or})+wr_{odd}(f(\infty))$. The based indices of the crossings $v\in V$ coincide with the indices of the corresponding crossings in $f(\infty)$. Hence by Proposition~\ref{prop:sm_unoriented_invariants}, $\rho(f(\infty))=(P_D^V)'(1)$ and $wr_{odd}(f(\infty)=P_D^V(-1)$. Thus, we get the first statement of the corollary.

The diagram $D^V_{unor}$ differs from $D^V_{or}$ by $|V|$ moves $H(2)$. By Proposition~\ref{prop:sm_unoriented_invariants}, $\rho(D^V_{unor})=\rho(D^V_{or})+|V|$ and $wr_{odd}(D^V_{unor})=wr_{odd}(D^V_{or})$. Analogously, $D^V_{A}$ differs from $D^V_{or}$ by $|V_-|$ moves $H(2)$ that gives the last statement.
\end{proof}

Consider the set $Q=\Z_2\times\Z$ with the involution $\lambda\colon (x,y)\mapsto (x+y,-y)$. Denote $\bar Q=Q/\lambda$. There is a bijection $\psi\colon \Z_2\times\Z_{\ge 0}\sqcup\Z\to \bar Q$ given by the formulas $\psi(a,b)=(a,2b), (a,b)\in\Z_2\times\Z$, and $\psi(a)=(0,2a+1), a\in\Z$.

For $\alpha\in H_1(F,\Z_2)$ denote $\bar S^{un}_\alpha=\bar Q$ if $\alpha=0$ and $\partial F=\emptyset$, $\bar S^{un}_\alpha=Q$ if $\alpha=0$ and $\partial F\ne\emptyset$, and $\bar S^{un}_\alpha=\Z_4$ if $\alpha\ne 0$. Let $\bar S^{un}=\bigsqcup_{\alpha\in H_1(F,\Z_2)}\bar S^{un}_\alpha$.
Consider the map $\Phi^{un}\colon\mathscr D(F)\to\bar S^{un}$ defined by the formula
\[
\Phi^{un}(D)=(\rho(D),wr_{odd}(D))\in \bar Q
\]
if $[D]=0$, and
\[
\Phi^{un}(D)=2\rho(D)+wr_{odd}(D)\in\Z_4=\bar S^{un}_{[D]}
\]
if $[D]\ne 0$.

\begin{proposition}\label{prop:skein_base_module_unor}
The map $\Phi^{un}$ induces a bijection $\mathscr K(F | H(2)^o,\Omega_2)\simeq\bar S^{un}$.
\end{proposition}
\begin{proof}
1. By Proposition~\ref{prop:sm_unoriented_invariants}, the induced map is well-defined. We will show it is a bijection.

2. 
Let $a\in\Z_4=S^{un}_0$. Consider the diagram $D=a\infty^+_0$. Then $[D]=0$ and $\Phi^{un}(D)=a$.

Let $\alpha\in H_1(F,\Z_2)$, $\alpha\ne0$, and $x\in\bar S^{un}_\alpha$. The element $x$ is determined by a pair $(a,b)\in Q$. Consider a diagram $D$ such that $[D]=\alpha$. Let $D'=D+(a-\rho(D))\bigcirc + (b-wr_{odd})\infty_0$. Then $[D']=\alpha$ and $\Phi^{un}(D')=x$.

\begin{lemma}\label{lem:basic_skein_unor_relations}
\begin{enumerate}
  \item $2\bigcirc\sim_{H(2)^o}\emptyset$,
  \item $\infty_k^\epsilon\sim_{\{H(2)^o,\Omega_2\}}\infty^\epsilon_{k\pm 2}$,
  \item $\infty_k^\epsilon\sim_{\{H(2)^o,\Omega_2\}}\infty^{-\epsilon}_{k+1}+\bigcirc$.
\end{enumerate}
\end{lemma}
\begin{proof}
The first statement is the unoriented version of a statement in Lemma~\ref{lem:H2_move}.

Given the diagram $\infty^\epsilon_k$, we can apply the move $H(2)^o$ to two neighbouring concentric circles and then annihilate the appearing trivial circles. Then we get the diagram $\infty^\epsilon_{k-2}$.

The last relation is given in Fig.~\ref{fig:infinity_unor_relation}.
\end{proof}
\begin{figure}[h]
\centering\includegraphics[width=0.5\textwidth]{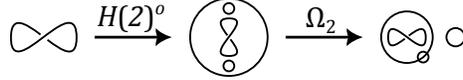}
\caption{Relation $\infty_k^\epsilon\sim\infty^{-\epsilon}_{k+1}+\bigcirc$}\label{fig:infinity_unor_relation}
\end{figure}

Let $D,D'$ be diagrams such that $[D]=[D']\in H_1(F,\Z_2)$ and $\Phi^{un}(D)=\Phi^{un}(D)$. As in Fig.~\ref{fig:crossing_to_infinity}, apply moves $H(2)^o$ to $D$ to get the diagram $D_A$ with $\infty$-components. Then pull the $\infty$-components to the base point $z$. By moves $H(2)^o$ and $\Omega_2$ the diagram $D_A$ can be transformed to $D'_A$. Then $D$ can be transformed to a diagram $D_1=D'_A+k\bigcirc+f(\infty)$.
By Lemma~\ref{lem:basic_skein_unor_relations}, we can suppose that $k\in\{0,1\}$ and $f(\infty)=b\infty$, $b\in\Z$.

At first assume that $[D]=[D']=0$. By Proposition~\ref{prop:sm_unoriented_invariants} $wr_{odd}(D_1)=(-1)^r wr_{odd}(D)$ and $\rho(D_1)=\rho(D)+r\cdot wr_{odd}(D)$ where $r$ is the number of times a diagram arc passed over $z$ while transforming $D_A$ to $D'_A$.

On the other hand, $D'\sim D_1'=D'_A+k'\bigcirc+b'\infty$ where $k'\in\{0,1\}$ and $b'\in\Z$. We can suppose that $wr_{odd}(D'_1)=wr_{odd}(D)$ and $\rho(D'_1)=\rho(D')$.

Since $\Phi^{un}(D)=\Phi^{un}(D)$, $wr_{odd}(D)=(-1)^s wr_{odd}(D')$ and $\rho(D)=\rho(D')+s\cdot wr_{odd}(D)$ for some $s=0,1$. If $s+r$ is odd then apply the transformation in Fig.~\ref{fig:H2o_boundary_reduction} to a trivial circle in $D_1$ and get a diagram
\[
D_2=D'_A+k\bigcirc+b\infty_1\sim D'_A+(k+b)\bigcirc-b\infty_0.
\]
such that $wr_{odd}(D_2)=(-1)^s wr_{odd}(D)$ and $\rho(D_2)=\rho(D)+s\cdot wr_{odd}(D)$. Then we can reassign $D_1$ to $D_2$ and assume that $s+r$ is even.

Thus, we have $\rho(D_1)=\rho(D'_1)$ and $wr_{odd}(D_1)=wr_{odd}(D'_1)$. But $wr_{odd}(D_1)=b$ and $\rho(D_1)=\rho(D'_A)+k$. There are analogous formulas for $D'_1$. Hence, $b=b'$ and $k=k'$. Thus, the diagrams $D_1$ and $D'_1$ are isotopic and $D\sim D_1=D'_1\sim D'$.

Now, let $[D]=[D']\ne 0$. Then $[D_A]\ne 0$ and there is a loop $\gamma\subset F$ such that $[D_A]\cdot\gamma=1$. Take a subdiagram $\infty_k^\epsilon$ and pull it along $\gamma$. Then the subdiagram turns into $\infty_{k+1}^\epsilon$. Hence, we have the relation $\infty_k^\epsilon\sim \infty_{k+1}^\epsilon$. By Lemma~\ref{lem:basic_skein_unor_relations}, $\infty_k^\epsilon\sim \infty_{k}^{-\epsilon}+\bigcirc$, then $2\infty_k^\epsilon\sim\bigcirc$ and $4\infty_k^\epsilon\sim\emptyset$.

Then $D_1\sim D'_A+b_1\infty$ where the number $b_1=b+2k$ is considered $\bmod 4$. Since $wr_{odd}(D_1)=(-1)^r wr_{odd}(D)$ and $\rho(D_1)=\rho(D)+r\cdot wr_{odd}(D)$ for some $r$, $2\rho(D)+wr_{odd}(D)\equiv 2\rho(D_1)+wr_{odd}(D_1)\bmod 4$. On the other hand,
$2\rho(D_1)+wr_{odd}(D_1)=2\rho(D'_A)-b_1$. Hence, $b_1=2\rho(D'_A)-2\rho(D)+wr_{odd}(D)$.

Analogously, $D'_1\sim D'_A+b'_1\infty$ where $b'_1=2\rho(D'_A)-2\rho(D')+wr_{odd}(D')$. Then $b_1=b'_1$ and
\[
D\sim D_1\sim D'_A+b_1\bigcirc\sim D'_1\sim D'.
\]
The proposition is proved.
\end{proof}

Let $Q^{\Z_4}=\Z_2\times\Z_4$ and $\lambda\colon (a,b)\mapsto (a+b,-b)$ be an involution on $Q^{\Z_4}$. Denote $\bar Q^{\Z_4}=Q^{\Z_4}/\lambda$ and $\bar S^{un,\Z_4}=\bigsqcup_{\alpha\in H_1(F\Z_2)} \bar S^{un,\Z_4}_\alpha$ where $\bar S^{un,\Z_4}_0=\bar Q^{\Z_4}$ if $\partial F=\emptyset$, $\bar S^{un,\Z_4}_0=Q^{\Z_4}$ if $\partial F\ne\emptyset$ and $\bar S^{un,\Z_4}_\alpha=\Z_4$ if $\alpha\ne 0$.

\begin{corollary}\label{cor:smoothing_skein modules_unor}
There are bijections:
\begin{enumerate}
\item $\mathscr K(F | H(2)^o,\Omega_2,\Omega_3)\simeq H_1(F,\Z_2)\times\Z_4$ induced by the map $D\mapsto ([D],2\rho(D)+wr_{odd}(D))$,
\item $\mathscr K(F | H(2),\Omega_2,\Omega_3)\simeq H_1(F,\Z_2)\times\Z_2$ induced by the map $D\mapsto ([D],n(D)\bmod 2)$ where $n(D)$ is the number of crossings in $D$,
\item $\mathscr K(F | H(2)^o,\Omega_1,\Omega_2,\Omega_3)\simeq \mathscr K(F | H(2),\Omega_1,\Omega_2,\Omega_3)\simeq H_1(F,\Z_2)$ induced by the map $D\mapsto [D]$,
\item $\mathscr K(F | H(2)^o,\Omega_2,2\Omega_\infty)\simeq \bar S^{un,\Z_4}$ induced by the map
\[
D\mapsto \left\{\begin{array}{cc}
                  (\rho(D),wr_{odd}(D)), & [D]=0, \\
                  2\rho(D)+wr_{odd}(D), & [D]\ne 0,
                \end{array}\right.
\]
\item $\mathscr K(F | H(2),\Omega_2,2\Omega_\infty)\simeq \bar\Z_4\sqcup (H_1(F,\Z_2)\setminus\{0\}\times\Z_2)$ induced by the map
\[
D\mapsto \left\{\begin{array}{cc}
                  \pm wr_{odd}(D)\bmod 4, & [D]=0, \\
                  n(D)\bmod 2, & [D]\ne 0,
                \end{array}\right.
\]
Here $\bar\Z_4=\Z_4/(x=-x)$ and $n(D)$ is the number of crossings in $D$.
\end{enumerate}
\end{corollary}

\begin{proof}
We need to describe how additional moves $O_1$, $\Omega_1$, $\Omega_3$ and $2\Omega_\infty$ affect the invariants $\rho$ and $wr_{odd}(D)$ and the subdiagrams $\bigcirc$ and $\infty_k^\pm$. We have the following table.
\begin{center}
\begin{tabular}{|c|c|c|c|c|}
\hline
move & $\rho$ & $wr_{odd}$ & relation & consequences\\
\hline
$O_1$               & $1$   & $0$       & $\bigcirc =\emptyset$ & $\infty^\epsilon_k=(-1)^k\infty^\epsilon_0$\\
$\Omega_1$          & $0,1$ & $\pm 1$   & $\infty^\epsilon_0 =\bigcirc$ & $\infty^\epsilon_k=\bigcirc=\emptyset$\\
$\Omega_3$          & $1$   & $\pm 2$   & $\infty^\epsilon_k=\infty^\epsilon_{k+1}$ & $\bigcirc=2\infty_k^\epsilon$, $4\infty^\epsilon_k=\emptyset$\\
$2\Omega_\infty$    & $0$   & $\pm 4$   & $2\infty^\epsilon_k=2\infty^\epsilon_{k+1}$ & $4\infty^\epsilon_k=\emptyset$\\
\hline
\end{tabular}
\end{center}
Note that $\Omega_3\sim\Omega_\infty$ modulo the moves $H(2)^o$ and $\Omega_2$.

The corollary follows from Proposition~\ref{prop:skein_base_module_unor} and the table above.
\end{proof}


\subsection{$\Delta$-skein module}

The aim of this subsection is to describe the skein module (more precisely, the set of equivalence classes $\mathscr K(F|\Omega_1,\Omega_2,\Omega_3,\Delta)$) generated by the $\Delta$-move on links in a fixed surface.

\subsubsection{Extended homotopy index polynomial for tangles in a fixed surface}

Let $F$ be a compact connected oriented surface.
In this section $\tau$ denotes the component index, $o$ denotes the order index and $h$ denotes the homotopy index of crossings of tangle diagrams in $F$ (see Section~\ref{subsec:traits}).

Let $D=D_1\cup\cdots\cup D_n$ be a tangle diagram in $F$ and $z_i\in D_i$.
Let $\kappa_i\in\pi_1(F,z_i)$ be the homotopy classes of the components (if the $i$-th component is long we set $\kappa_i=1$). Consider the quotient sets $\hat\pi_{ii}(F,z)=\pi_1(F,z_i)/Ad_{\kappa_i}$ and $\hat\pi_{ij}(F,z)=\kappa_i\backslash\pi_1(F,z_i,z_j)/\kappa_j$, $i\ne j$.

For $1\le i,j\le n$, $i\ne j$ denote
\[
\widetilde{LK}_{ij}=\sum_{v:\tau(v)=(i,j)} sgn(v)\cdot h(v)\in\Z[\hat\pi_{ij}(F,z)].
\]
For a closed component $D_i$ denote
\[
\widetilde{LK}_{ii}=\sum_{v:\tau(v)=(i,i), h(v)\ne 1,\kappa_i} sgn(v)\cdot h(v)\in\Z[\hat\pi_{ii}(F,z)].
\]
For a long component $D_i$ denote
\[
\widetilde{LK}^{\epsilon}_{i}=\sum_{v:\tau(v)=(i,i), o(v)=\epsilon, h(v)\ne 1} sgn(v)\cdot h(v)\in\Z[\hat\pi_{ii}(F,z)],\quad \epsilon=\pm 1.
\]

 Consider the tuple
\begin{multline*}
\widetilde{LK}(D)=\left[(\kappa_i)_{i=1}^n, (\widetilde{LK}_{ij})_{i,j}, (\widetilde{LK}^\epsilon_{i})_{i,\epsilon}\right]\in \prod_{i}\pi(F,z_i)\times\prod_{i,j}\Z[\hat\pi_{ij}(F,z)]\\ \times\prod_{i: D_i\mbox{\scriptsize\ is long}}\Z[\hat\pi_{ii}(F,z)]^{\times 2}.
\end{multline*}

\begin{definition}
The diffeomorphism group $Diff_0(F,z)$ acts on the product of fundamental group. The orbit
$LK(D)=\widetilde{LK}(D)\cdot Diff_0(F,z)$ is called the \emph{extended homotopy index polynomial} of the link $L$.
\end{definition}

\begin{theorem}\label{thm:Delta_skein_module}
Let $L,L'$ be two tangles in the surface $F$. Then the following conditions are equivalent:
\begin{enumerate}
\item $L$ can be transformed to $L'$ by Reidemeister moves and $\Delta$-moves;
\item $LK(L)=LK(L')$.
\end{enumerate}
\end{theorem}

\begin{corollary}[H. Murakami, Y. Nakanishi~\cite{MN}]
Two oriented classical links $L$ and $L'$ are $\Delta$-equivalent if and only if their linking numbers coincide.
\end{corollary}
\begin{proof}
Indeed, in the classical case $\kappa_i=1$ and $\widetilde{LK}_{ij}=lk(D_i,D_j)\cdot 1$ for $i\ne j$ and  $\widetilde{LK}_{ii}=0$. Thus, $\Delta$-equivalence classes of links are determined by the linking numbers.
\end{proof}

\begin{definition}
Let $T$ be a tangle in $F\times [0,1]$ and $\gamma$ an oriented curve such that $\gamma\cap T =\partial\gamma$. Transform the tangle $T$ in a small neighbourhood of $\gamma$ as shown in Fig.~\ref{fig:clasp}. The part of the new tangle $T$ lying in the neighbourhood of $\gamma$ is called a \emph{clasp}.
\begin{figure}[h]
\centering\includegraphics[width=0.4\textwidth]{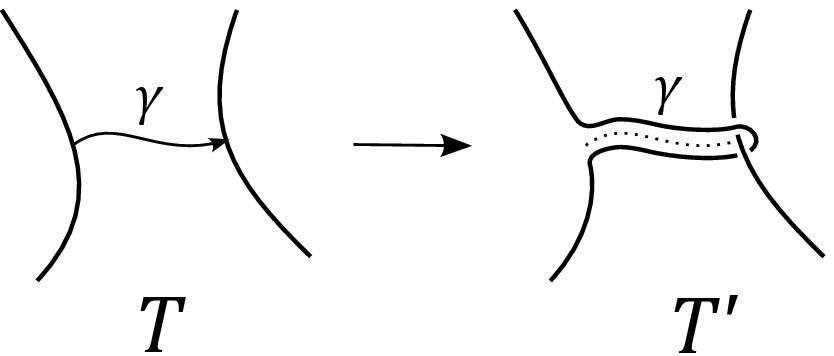}
\caption{A clasp}\label{fig:clasp}
\end{figure}
\end{definition}

\begin{lemma}[Clasp lemma]\label{lem:clasp_lemma}
Let $T_1$ and $T_2$ be homotopy equivalent tangles in $F\times [0,1]$. Then there exists an isotopy $T_1\simeq T_1'$ such that $T_1'$ is obtained from $T_2$ by adding several clasps.
\end{lemma}
\begin{proof}
A homotopy from $T_1$ to $T_2$ can be presented by a sequence of spatial isotopies and crossing changes. Any crossing change can be replaced by an isotopy which adds a clasp to the tangle (Fig.~\ref{fig:cross_change1}). Since clasps survive isotopies, in the end we come to a diagram $T_1'$ which differs from $T_2$ by several clasps.
\begin{figure}[h]
\centering\includegraphics[width=0.5\textwidth]{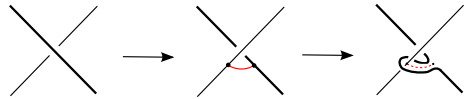}
\caption{Emulating a crossing change by adding a clasp}\label{fig:cross_change1}
\end{figure}
\end{proof}

\begin{lemma}[Switching lemma]\label{lem:switching_lemma}
Let $L$ be a tangle diagram in the surface $F$ and $a,b$ two crossings such that $\tau(a)=\tau(b)$, $o(a)=o(b)$, $h(a)=h(b)$ and $sgn(a)=-sgn(b)$. Let $L'$ be the diagram obtained from $L$ by crossing switching at $a$ and $b$.
Then $L$ and $L'$ are $\Delta$-equivalent.

\begin{figure}[h]
\centering\includegraphics[width=0.6\textwidth]{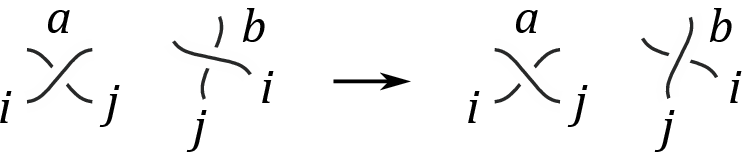}
\end{figure}
\end{lemma}

\begin{proof}

Since $a$ and $b$ lie on the same component $i$, we can pull the crossing $b$ to $a$ using second and third Reidemeister moves. Since $h(a)=h(b)$ the long arc connecting $a$ and $b$ is contractible. By Lemma~\ref{lem:clasp_lemma} contract the arc to a small loop with clasps (Fig.~\ref{fig:contraction} left). Then split the clasps on the loop by adding crossings with second Reidemeister move (Fig.~\ref{fig:contraction} right).

\begin{figure}[h]
\centering\includegraphics[width=0.8\textwidth]{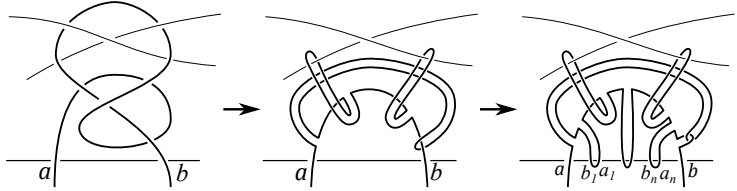}
\caption{Contraction of an arc to a small loop with clasps}\label{fig:contraction}
\end{figure}

Then move the small arcs $ab_1, a_1b_2,\dots, a_nb$ over the correspondent clasps as shown in Fig.~\ref{fig:move_over_clasp} using moves $\Omega_2$, $\Omega_3$ and $\Delta$.
\begin{figure}[h]
\centering\includegraphics[width=0.8\textwidth]{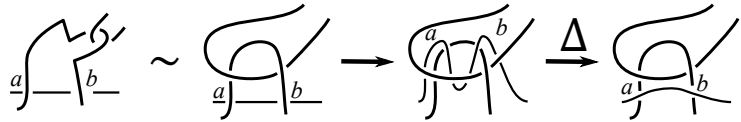}
\caption{Moving the arc $ab$ over a clasp}\label{fig:move_over_clasp}
\end{figure}

Remove the intermediate crossings $a_1,b_1,\cdots a_n,b_n$ with second Reidemeister moves. Return the loop with clasps to the long arc by the inverse isotopy. Thus, we switched the crossings of the small arc $ab$.

Finally, pull the crossing $b$ back to the original position. Since the crossing $b$ was switched, we may need $\Delta$-moves to pass over crossings (Fig.~\ref{fig:passing_crossing}).
\end{proof}

\begin{figure}[h]
\centering\includegraphics[width=0.6\textwidth]{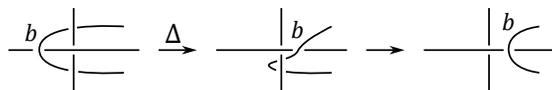}
\caption{Passing a crossing}\label{fig:passing_crossing}
\end{figure}

\begin{proof}[Proof of Theorem~\ref{thm:Delta_skein_module}]
The moves $\Omega_1, \Omega_2, \Omega_3$ and $\Delta$ do not change the indices $\tau, o, h$ of crossings, hence, do not change $\widetilde{LK}$. Isotopies which moves points $z_i$ can turn $\widetilde{LK}$ to another representative of extended homotopy index polynomial.

Let $LK(L)=LK(L')$. Apply an isotopy to $L'$ so that $\widetilde{LK}(L)=\widetilde{LK}(L')$. Since $[L_i]=[L'_i]$, by Lemma~\ref{lem:clasp_lemma} we can isotope the tangle $L'$ to a tangle $L''$ which is $L$ with some clasps added. Since $\widetilde{LK}(L)=\widetilde{LK}(L'')$, the contribution of the clasp crossings is zero. The crossings in the clasp stalks appear in pairs with opposite signs, so their contribution to $\widetilde{LK}_{ij}$ annihilates. Hence, the total contribution of the crossings in the clasp heads vanishes.

Let $v_1$ and $v_2$ be the crossings of a clasp head. Assume that the contribution term $sgn(v_1)h(v_1)$ of $v_1$ is not zero. Since $sgn(v_1)=sgn(v_2)$, the terms of $v_1$ and $v_2$ cannot contract. Then there exists another clasp head crossing $w$ which gives the opposite contribution to $v_1$. Hence, $\tau(v_1)=\tau(w)$, $o(v_1)=o(w)$, $h(v_1)=h(w)$ and $sgn(v_1)=-sgn(w)$. By Lemma~\ref{lem:switching_lemma} we can switch the crossing $v_1$ and $w$. After switching we can remove these two clasps with second Reidemeister moves. Thus, we can eliminate all clasps with nontrivial heads.

Let $v$ be a crossing of a clasp head which contributes zero to $\widetilde{LK}$. Then $\tau(v)=(i,i)$ and $h(v)=1$ (or $h(v)=\kappa_i$ if the component is closed). Create with a first Reidemeister move a crossing $w$ on the component $i$ such that $o(w)=o(v)$, $h(w)=h(v)$ and $sgn(w)=-\sgn(v)$. Then switch the crossings $v$ and $w$ by Lemma~\ref{lem:switching_lemma}. Now we can remove the clasp with second Reidemeister moves and the crossing $w$ with a first Reidemeister move.

Thus, we can remove all clasps of $L''$, hence, $L\sim_\Delta L''\sim_\Delta L'$.
\end{proof}

\begin{remark}
An analog of Theorem~\ref{thm:Delta_skein_module} holds for regular tangles. We should add the writhes $wr_i(D)=\sum_{v\colon \tau(v)=(i,i)}sgn(v)$ to the extended homotopy index polynomial (or, equivalently, allow the homotopy values $1$ and $\kappa_i$ in the definition of $\widetilde{LK}_{ii}$ and $\widetilde{LK}_{i}^\epsilon$). Then the new invariant classifies $\Delta$-equivalence classes of regular tangles.

In the proof the only difference appears at the last step. Given a clasp head crossing $v$ with zero contribution, we can create a pair of loop crossings $w$ and $w'$ with moves $\Omega_2$, $\Omega_3$. Then we switch the crossings $v$ and $w$ and contract the clasp. Thus, in the end we get a diagram which coincides with $L$ up to some loops. Since the writhes of the tangle diagrams coincide, the additional loops can be removed by moves $\Omega_2$, $\Omega_3$.
\end{remark}

\section{Table of invariant binary functorial maps}\label{app:functorial_map_table}

\begin{center}
\begin{longtable}{|c|c|c|c|}
 \hline
Scheme & $\mathcal M$ & $\tau$ & $\mathcal M'$\\
 \hline
 \endfirsthead

 \hline
Scheme & $\mathcal M$ & $\tau$ & $\mathcal M'$\\
 \hline
 \endhead

 \hline
 \endfoot

 \hline
 \endlastfoot

& & trait & $H(2)$\\
$(Sm^{or}, Sm^{unor})$ & $\mathcal M_{class}^+$ & index, $\Pi_1=\{1\}$ & $H(2)^o$\\
& & index, $\Pi_1\ni 0$ & $H(2)$\\
\hline
$(Sm^{or}, Sm^{unor})$ & $\mathcal M_{class}^{reg+}$ & trait & $H(2)$\\
& & index & $H(2)^o$\\
\hline
& & trait & $Sm^{or}_\pm, H(2)_+$\\
& & index, $\Pi_1=\{0\}$  & $H(2)_+,\Omega_2,\Omega_3$\\
$(Sm^{or}, id)$ & $\mathcal M_{class}^+$ & index, $\Pi_1=\{1\}$  & $H(2)^o_+,\Omega_1,\Omega_2,\Omega_3$\\
& & index, $\Pi_1=\{0,1\}$  & $H(2)_+,\Omega_1,\Omega_2,\Omega_3$\\
& & parity & $H(2)_+,\Omega_2,2\Omega_\infty$\\
\hline
&  & trait & $Sm^{or}_\pm, H(2)_+$\\
$(Sm^{or}, id)$ & $\mathcal M_{class}^{reg+}$  & index & $H(2)_+^o,\Omega_2,\Omega_3$\\
& & parity  & $H(2)^o_+,\Omega_2,2\Omega_\infty$\\
\hline

& & trait & $Sm^{A}$ or $Sm^B$, $H(2)^o$\\
$(Sm^{unor}, id)$  & $\mathcal M_{class}^{+}$ & index, $\Pi_1\ni 1$  & $H(2)^o,\Omega_1,\Omega_2,\Omega_3$\\
& & index, $\Pi_1=\{0\}$  & $H(2)^o,\Omega_2,\Omega_3$\\
& & parity & $H(2)^o,\Omega_2,2\Omega_\infty$\\
\hline
& & trait & $Sm^{A}$ or $Sm^B$, $H(2)^o$\\
$(Sm^{unor}, id)$ & $\mathcal M_{class}^{reg+}$ & index & $H(2)^o,\Omega_2,\Omega_3$\\
& & parity & $H(2)^o,\Omega_2,2\Omega_\infty$\\
\hline

& & trait & $Sm^{A}$ or $Sm^B$, $H(2)^o$\\
$(Sm^A, id)$ & $\mathcal M_{class}^{+}$ & index, $\Pi_1\ni 1$ & $H(2),\Omega_1,\Omega_2,\Omega_3$\\
& & index, $\Pi_1=\{0\}$ & $H(2),\Omega_2,\Omega_3$\\
& & parity & $H(2),\Omega_2,2\Omega_\infty$\\
\hline
& & trait & $Sm^{A}$ or $Sm^B$, $H(2)^o$\\
$(Sm^A, id)$ & $\mathcal M_{class}^{reg+}$ & index & $H(2),\Omega_2,\Omega_3$\\
& & parity & $H(2),\Omega_2,2\Omega_\infty$\\
\hline

$(Sm^{Kauffman}, id)$  & $\mathcal M_{class}^{reg+}$ & weak parity & $O_\delta, CC, \Omega_2, \Omega_3$\\
& & parity & $O_\delta, CC, \Omega_2$\\
\hline

& & trait & $V,V\Omega_1, V\Omega_2, V\Omega_3$\\
& & index, $\Pi_1\ni 0$ & $\mathcal M_{fused}^{+}$\\
& & index, $\Pi_1=\{1\}$ & $\mathcal M_{fused}^{reg+}$\\
$(id, V)$ & $\mathcal M_{virt}^{+}$ & weak parity & $\mathcal M_{virt}^{+}$\\
& & index of types $c$, $d$, $o$, $p$ & $\mathcal M_{welded}^{+}$\\
& & index of types $j$, $q$ & $\mathcal M_{fd}^{reg+}$ \\
\hline
& & trait & $V,V\Omega_1, V\Omega_2, V\Omega_3$\\
& & index & $\mathcal M_{fused}^{reg+}$\\
$(id, V)$ & $\mathcal M_{virt}^{reg+}$ & weak parity & $\mathcal M_{virt}^{reg+}$\\
& & index of types $c$, $d$, $o$, $p$ & $\mathcal M_{welded}^{reg+}$\\
& & index of types $j$, $q$ & $\mathcal M_{fd}^{reg+}$ \\
\hline

& & trait & $CC,\Omega_1,\Omega_2,\Omega_3$\\
$(id, CC)$ &$\mathcal M_{class}^{+}$  & index & $\Delta,\Omega_1,\Omega_2,\Omega_3$\\
& & order & $\Omega_1,\Omega_2,\Omega_3$\\
\hline
& & trait & $CC,\Omega_2,\Omega_3$\\
$(id, CC)$ &$\mathcal M_{class}^{reg+}$  & index & $\Delta,\Omega_2,\Omega_3$\\
& & order & $\Omega_2,\Omega_3$\\
\hline

& & trait & $CC,\Omega_1,\Omega_2,\Omega_3$\\
lifting & $\mathcal M_{flat}^{+}$  & signed index & $\Delta,\Omega_1,\Omega_2,\Omega_3$\\
& & flat order & $\Omega_1,\Omega_2,\Omega_3$\\
\hline
& & trait & $CC,\Omega_2,\Omega_3$\\
lifting & $\mathcal M_{flat}^{reg+}$  & signed index & $\Delta,\Omega_2,\Omega_3$\\
& & flat order & $\Omega_2,\Omega_3$\\
\hline
\end{longtable}
\end{center}

\end{document}